\documentclass[]{amsart}

\usepackage[Symbol]{upgreek}

\usepackage{graphicx}

\usepackage{amssymb}
\usepackage{epic}
\usepackage{mathrsfs}
\usepackage{accents}
\usepackage{array}
\usepackage{stmaryrd}
\usepackage{enumitem}
\usepackage{longtable}
\usepackage[shortcuts]{extdash}

\usepackage{etoolbox} % for \patchcmd

\usepackage{tikz}
\usepackage{pgfplots}
\pgfplotsset{compat = newest}

\usepackage[hidelinks]{hyperref}

\input{xy}
\xyoption{matrix}
\xyoption{arrow}

\numberwithin{equation}{section}
\makeatletter
\renewcommand{\tocsection}[3]
 { \indentlabel{\@ifnotempty{#2}{\parbox{2.5em}{\ignorespaces#1 #2.}\quad}}#3}

\makeindex

\newcommand{\bbold}{\mathbb}

\def\R { {\bbold R} }
\def\Q { {\bbold Q} }
\def\Z { {\bbold Z} }
\def\C { {\bbold C} }
\def\N { {\bbold N} }

\def\c {\mathcal{C}}

\def\g {\operatorname{g}}
\def \I{\operatorname{I}}

\def \Ex{\operatorname{E}}
\def \Dx{\operatorname{D}}

\def \order{\operatorname{order}}

\def \ex{\operatorname{e}}

\renewcommand\epsilon{\varepsilon}

\def \d{\operatorname{d}}

\def \ev{\operatorname{e}}
\def \bar {\overline}
\def \<{\langle}
\def \>{\rangle}

\def \tilde {\widetilde}

\def \hat {\widehat}

\def \((  {(\!(}
\def \)) {)\!)}

\def \Li{\operatorname{Li}}
\def \res{\operatorname{res}}

\def \k {{{\boldsymbol{k}}}}

\DeclareMathSymbol{\precequ}{\mathrel}{symbols}{"16}
\DeclareMathSymbol{\succequ}{\mathrel}{symbols}{"17}

\def \nasymp{\not\asymp}

\renewcommand{\Re}{\operatorname{Re}}
\renewcommand{\Im}{\operatorname{Im}}

\newtheorem{theorem}{Theorem}[section]

\newtheorem{lemma}[theorem]{Lemma}
\newtheorem{prop}[theorem]{Proposition}
\newtheorem{cor}[theorem]{Corollary}

\newtheorem*{theoremA}{Theorem A}
\newtheorem*{theoremB}{Theorem B}
\newtheorem*{cor1}{Corollary 1}
\newtheorem*{cor2}{Corollary 2}
\newtheorem*{cor3}{Corollary 3}
\newtheorem*{cor4}{Corollary 4}
\newtheorem*{cor5}{Corollary 5}

\theoremstyle{definition}

\newtheorem{definition}[theorem]{Definition}

\theoremstyle{remark}
\newtheorem*{example}{Example}
\newtheorem*{examples}{Examples}
\newtheorem{exampleNumbered}[theorem]{Example}
\newtheorem{examplesNumbered}[theorem]{Examples}

\newtheorem*{remarks}{Remarks}
\newtheorem*{remark}{Remark}

\newtheorem{remarkNumbered}[theorem]{Remark}

\newcommand{\abs}[1]{\lvert#1\rvert}

\def \Ric{\operatorname{Ri}}

\let\oldi\i
\let\oldj\j
\renewcommand\i{\relax\ifmmode{\boldsymbol{i}}\else\oldi\fi}
\renewcommand\j{\relax\ifmmode{\boldsymbol{j}}\else\oldj\fi}

\renewcommand\leq{\leqslant}
\renewcommand\geq{\geqslant}
\renewcommand\preceq{\preccurlyeq}
\renewcommand\succeq{\succcurlyeq}
\renewcommand\le{\leq}
\renewcommand\ge{\geq}

\DeclareMathAlphabet{\mathbf}{OML}{cmm}{b}{it}

\DeclareFontFamily{U}{fsy}{}
\DeclareFontShape{U}{fsy}{m}{n}{<->s*[.9]psyr}{}
\DeclareSymbolFont{der@m}{U}{fsy}{m}{n}
\DeclareMathSymbol{\der}{\mathord}{der@m}{182}

\DeclareSymbolFont{der@m}{U}{fsy}{m}{n}
\DeclareMathSymbol{\derdelta}{\mathord}{der@m}{100}

\newcommand\ndeg{\operatorname{ndeg}}

\DeclareSymbolFont{imag@m}{OT1}{cmr}{m}{ui}
\DeclareMathSymbol{\imag}{\mathord}{imag@m}{105}

\DeclareFontFamily{OMS}{smallo}{}
\DeclareFontShape{OMS}{smallo}{m}{n}{<->s*[.65]cmsy10}{}
\DeclareSymbolFont{smallo@m}{OMS}{smallo}{m}{n}
\DeclareMathSymbol{\smallo}{\mathord}{smallo@m}{79}

\DeclareFontFamily{OMS}{largerdot}{}
\DeclareFontShape{OMS}{largerdot}{m}{n}{<->s*[.8]cmsy10}{}
\DeclareSymbolFont{largerdot@m}{OMS}{largerdot}{m}{n}
\DeclareMathSymbol{\largerdot}{\mathord}{largerdot@m}{15}

\DeclareMathSymbol{\llambda}{\mathord}{der@m}{108}
\DeclareMathSymbol{\rrho}{\mathord}{der@m}{114}

\def \upg{\upgamma}
\def \Upg{\Upgamma}
\def \upl{\uplambda}
\def \Upl{\Uplambda}
\def \upo{\upomega}

\def \Upd{\Updelta}

\newcommand{\equationqed}[1]{\[\pushQED{\qed}#1 \qedhere\popQED\]\let\qed\relax}
\newcommand{\alignqed}[1]{\begin{align*}\pushQED{\qed} #1 \qedhere\popQED\end{align*}\let\qed\relax}

\makeatletter
\newcommand{\dminus}{\mathbin{\text{\@dminus}}}

\newcommand{\@dminus}{%
  \ooalign{\hidewidth\raise1ex\hbox{\bf.}\hidewidth\cr$\m@th-$\cr}%
}
\makeatother

\def\ddeg{\operatorname{ddeg}}

\def \Caz{\mathcal{C}^0_a}
\def \Cao{\mathcal{C}^1_a}
\def \Cat{\mathcal{C}^2_a}

\def \Caom{\mathcal{C}^{\omega}_a}

\def \Gr{\mathcal{C}^r}
\def \Gn{\mathcal{C}^n}

\def \Go{\mathcal{C}^1}
\def \Gt{\mathcal{C}^2}
\def \Gi{\mathcal{C}^{<\infty}}
\def \Ginf{\mathcal{C}^{\infty}}
\def \Gom{\mathcal{C}^{\omega}}
\def \inv{\operatorname{inv}}
\def \Sol{\operatorname{Sol}}
\def \sign{\operatorname{sign}}

\def \Caz{\mathcal{C}^0_a}

\def \Cao{\mathcal{C}^1_a}
\def \Cat{\mathcal{C}^2_a}

\def \Can{\mathcal{C}^n_a}
\def \Caom{\mathcal{C}^{\omega}_a}
\def \Calinf{\mathcal{C}^{<\infty}}

\def \Caln{\mathcal{C}^{n}}

\makeatletter
\renewcommand\part{\@startsection{part}{0}%
  \z@{\linespacing\@plus\linespacing}{.5\linespacing}%
  {\normalfont\bfseries\centering}}
\makeatother

\makeatletter
\renewcommand\theindex{\@restonecoltrue\if@twocolumn\@restonecolfalse\fi
  \columnseprule\z@ \columnsep 35\p@
  \twocolumn[\@xp\part\@xp*\@xp{\bf Index}\bigskip]%
  \let\item\@idxitem
  \parindent\z@  \parskip\z@\@plus.3\p@\relax
  \small}

\makeatletter
\newcommand{\smallbullet}{} % for safety
\DeclareRobustCommand\smallbullet{%
  \mathord{\mathpalette\smallbullet@{0.6}}%
}
\newcommand{\smallbullet@}[2]{%
  \vcenter{\hbox{\scalebox{#2}{$\m@th#1\bullet$}}}%
}
\makeatother
  
\renewcommand{\marginpar}[1]{}

\begin{document}

\title{Revisiting Second-Order Linear Differential Equations over Hardy Fields}
% Revisited}
\author[Aschenbrenner]{Matthias Aschenbrenner}
\address{Kurt G\"odel Research Center for Mathematical Logic\\
Universit\"at Wien\\
1090 Wien\\ Austria}
\email{matthias.aschenbrenner@univie.ac.at}

\author[van den Dries]{Lou van den Dries}
\address{Department of Mathematics\\
University of Illinois at Urbana-Cham\-paign\\
Urbana, IL 61801\\
U.S.A.}
\email{vddries@illinois.edu}

\author[van der Hoeven]{Joris van der Hoeven}
\address{CNRS, LIX (UMR 7161)\\ 
Campus de l'\'Ecole Polytechnique\\  91120 Palaiseau \\ France}
\email{vdhoeven@lix.polytechnique.fr}

\date{March, 2026}

\begin{abstract} 
We review second-order homogeneous linear differential equations 
%on the real line 
with coefficient functions whose germs  lie in a Hardy field (and hence are strongly non-oscillating). 
We prove a conjecture of Boshernitzan~(1982): the oscillating solutions to such an equation are given by  amplitude and phase functions with germs in a bigger Hardy field, and hence
oscillate in a very regular way. We give sharp conditions for the uniqueness of such germs, 
study their asymptotic behavior, and use this to obtain information about the zeros and critical points of oscillating solutions. 
\end{abstract}

\pagestyle{plain}
 
\maketitle

%\bigskip

%\tableofcontents

\section*{Introduction}

\noindent
Second-order ordinary linear differential equations are ubiquitous
in applications of mathematics. For instance, most  ``special functions'' of mathematical physics
satisfy such  equations, often with rational function coefficients; see~\cite[p.~1]{BW} for an explanation. Examples  encountered below are the Airy, Bessel, and Coulomb wave functions.
%\begin{quote}  {\it  Indeed, these functions  were discovered through the study of physical problems: vibrations, heat flow, equilibrium, and so on. The associated equations are partial differential equations of second order. In some coordinate systems, these equations can be solved by separation of variables, leading to the second-order ordinary differential equations in question.}\/
%\end{quote}
In this paper we study the asymptotic behavior of solutions to second-order homogeneous   linear
differential equations whose coefficient functions are {\it non-oscillating at $+\infty$ in a
very strong sense.}\/ 
%This includes many classical    differential equations giving rise to special functions.
 
To be precise, here is some notation and terminology: $\mathcal C$ is the ring of germs at $+\infty$ of continuous functions $[a,+\infty)\to\R$
where~${a\in\R}$, and for each  $n$,  $\c^n$ is its subring consisting of those $g\in\c$ with an
    $n$-times continuously differentiable  representative~$[a,+\infty)\to\R$, ${a\in\R}$. Then~$\mathcal C^{n+1} \subseteq\mathcal C^n$, and for $g\in\mathcal C^{n+1}$ we have its derivative~${g'\in\mathcal C^n}$.
Let  $f\in\c$ and consider   the  differential equation
\begin{equation}\tag{$\ast$}\label{eq:ast}
Y''+fY\ =\ 0.
\end{equation}
A (real) {\it solution}\/ to \eqref{eq:ast} is a germ $y\in\c^2$ such that $y''+fy=0$ (in $\c$).
%Recall that 
The solutions to \eqref{eq:ast}  form a $2$-dimensional
subspace of the $\R$-linear space $\c^2$. 
Some equations~\eqref{eq:ast} have oscillating solutions, where
a germ $g\in\c$ {\it oscillates}\/  if it has a representative~$g\colon [a,\infty)\to\R$ ($a\in\R$) such that 
$g(t)=0$ for arbitrarily large~$t\ge a$ and~$g(t)\ne 0$ for arbitrarily large~$t\ge a$. 
By Sturm's Separation Theorem~\cite[Chap\-ter~2, \S{}6, Theorem~7]{BR}, if {\it some}\/ solution to \eqref{eq:ast} oscillates, then so does {\it every}\/ nonzero solution to~\eqref{eq:ast}. In this case we say that $f$ {\it generates oscillation.}\/ (So $f=1$ generates oscillation.) It is well-known that the oscillating solutions to~\eqref{eq:ast} can be represented
by   a {\it phase function}~\cite[Chapter~5]{Boruvka}:  a function~${\phi\colon [a,+\infty)\to\R}$~($a\in\R$)
 with~${\phi(t)\to+\infty}$ as~$t\to+\infty$ whose germ (also denoted by $\phi$) is  in $\c^3$, with~${\phi'(t)>0}$ for all sufficiently large $t\ge a$, such that 
the solutions to \eqref{eq:ast} are exactly the germs~$y=c\cos(\phi+d)/\sqrt{\phi'}$ with $c,d\in\R$.

For $g, h\in \c$ we set $g<_{\ev} h$ (or $h>_{\ev} g$) iff $g(t)< h(t)$ {\em eventually\/}, that is for all sufficiently large $t$; here and below $g$ also denotes any representative of its germ. Thus $g\in \c$ is non-oscillating iff $g>_{\ev} 0$, or $g=0$, or $g<_{\ev} 0$. For other notations and conventions used below, see the end of this introduction.  

The maps $g\mapsto g'\colon\c^{n+1}\to\c^n$ restrict to a {\em derivation\/}  on the sub\-ring~$\mathcal C^{<\infty}:=\bigcap_n \mathcal C^n$ of $\c$.
%which satisfies the Product Rule (that is, a {\it derivation}\/ on $\c^{<\infty}$).
If~$f$ in \eqref{eq:ast} lies in $\mathcal C^{<\infty}$,
then so does each solution to \eqref{eq:ast}. 
In this paper we restrict  $f$  further to be {\it hardian}\/, that is, $f\in \mathcal C^{<\infty}$ and
for each polynomial~${p\in\Q[Y_0,\dots,Y_n]}$, the germ $p(f,f',\dots,f^{(n)})$ is non-oscillating. 
This concept formalizes the 19th century idea (du Bois-Reymond, Borel, Lindel\"of) of ``functions of regular growth''.
Many naturally occurring non-oscillating differentially algebraic functions ---for example, any univariate rational function with real coefficients, and~$\ex^x$ and $\log x$---have hardian germs.
%whereas that of~$\ex^{\cos x}$ is not. 
This is also the case for certain differentially transcendental functions: the Riemann $\zeta$-function, Euler's $\Gamma$-function~\cite{Rosenlicht83}, and
even some  functions ultimately growing faster than each iterate of $\ex^x$~\cite{Boshernitzan86}. 

If $g\in\c^{<\infty}$ is hardian, then so is $g'$; but a sum of hardian germs is not always hardian.
A {\it Hardy field}\/ (Bourbaki~\cite{Bou}) is a subfield of~$\mathcal C^{<\infty}$  closed under the derivation $g\mapsto g'$ of $\mathcal C^{<\infty}$, and thus consists of hardian germs that can be manipulated algebraically.
Any hardian germ $h$ generates a Hardy field $\Q(h, h', h'',\dots)$.  A Hardy field is naturally a differential field, and also an ordered
field with the (total, strict) ordering given by $<_{\ev}$, which we indicate then also by $<$. 

 A basic  Hardy field is the field $\R(x)$ of germs of rational functions with real coefficients,
where~$x=$~the germ of the identity function. A larger Hardy field, discovered by Hardy~\cite{Har12a}, consists of the 
germs of {\it logarithmic-exponential functions} ({\it LE-functions}\/, for short):
 the real-valued functions  obtainable in finitely many steps from real constants and the identity function using addition, multiplication,  division,  taking logarithms, and exponentiating. Examples include  the germs of 
the function given for large positive~$x$ by
$x^r$ ($r\in\R$), $\ex^{x^2}$, and~$\log\log x$. (We denote this Hardy field by $H_{\text{LE}}$; a precise definition is at the beginning of Section~\ref{sec:prelims}.)
Each Hardy field extends to a Hardy field~$H\supseteq \R$  which is {\it Liouville closed}:  real closed, 
closed under exponentiation, and containing for each $h\in H$ a $g\in\c^1$ 
%(hence any $g\in \c^1$) 
with~$g'=h$;
such $H$ is also closed under logarithms of positive germs, and contains~$\R(x)$.
(That $H$ is real closed means that $H[\imag]\subseteq\Calinf[\imag]$, where~$\imag^2=-1$, is an algebraically closed field.)
For this and other
basic facts on Hardy fields, see~\cite{Boshernitzan81, Boshernitzan82,  Ros}.
O-minimal expansions of the real field are a natural source of Hardy fields~\cite{Miller}.

{\it In the rest of this introduction $H$ is a Hardy field containing~$f$.}\/
It is well-known (\cite[Theorem~16.7]{Boshernitzan82}, \cite[Corollary~2]{Ros}, or Corollary~\ref{cor:char osc} below) that
if  no solution to \eqref{eq:ast} oscillates, then every solution to \eqref{eq:ast} is {\it $H$-hardian,}\/
that is, contained in a Hardy field extending  $H$.
For example, the germs of the  $\R$-linearly independent solutions~$\operatorname{Ai},\operatorname{Bi}\colon\R\to\R$ 
to the Airy equation~$Y''-xY=0$ given by
\begin{align*}
\operatorname{Ai}(t)\ &=\ \frac{1}{\pi}\int_0^\infty \cos\left(\frac{s^3}{3}+st\right)ds, \\
\operatorname{Bi}(t)\ &=\ \frac{1}{\pi}\int_0^\infty \left[\exp\left(-\frac{s^3}{3}+st\right) +\sin\left(\frac{s^3}{3}+st\right)\right]ds
\end{align*}
are $H$-hardian. We focus on the case where \eqref{eq:ast} has an oscillating 
 (hence non-hardian) solution. 
 Our first main result shows that 
     these oscillating solutions may nevertheless be described  
by  amplitude   and phase functions with $H$-hardian germs:

\begin{theoremA}  
Suppose \eqref{eq:ast} has an oscillating solution. Then there are $H$-hardian germs~${g>0}$, $\phi>\R$  such that the solutions to    \eqref{eq:ast} are exactly the germs $cg\cos(\phi+d)$ where $c,d\in\R$.
Any such $H$-hardian germs~$g$,~$\phi$ lie in a common Hardy field extension of $H$, and $g^2 \phi'\in\R$. 
\end{theoremA}

\noindent
The next theorem gives sufficient conditions for $(g, \phi)$ to be unique up to multiplying~$g$ by a positive constant
and adding a constant to $\phi$. 
%(In general, $(g,\phi)$ is not unique in that sense.)  
To state this theorem we use more notation and terminology:   put~$H^>:=\{h\in H:h>0\}$ and~$H^{>\R}:=\{{h\in H:h>\R}\}$, and
for~$h\in H^>$,  let~$h^\dagger:=h'/h=(\log h)'$ be the logarithmic derivative of~$h$;
then~$h^\dagger > 0$ for all~$h\in H^{>\R}$.
We call a germ {\it perfectly $H$-hardian}\/ if it is  $E$-hardian for each Hardy field extension~$E$ of $H$,
and {\it perfectly hardian}\/ if it is perfectly $\Q$-hardian. %\marginpar{or use ``absolutely $H$-hardian''?}
%(In \cite{ADH6} we also called these germs ``absolutely tame''.)
%Thus {\em $\Q$-hardian\/} is equivalent to {\em hardian}.  
Every perfectly hardian~$y\in\Calinf$ is differentially algebraic (over~$\Q$), by~\cite[Theorem~14.4]{Boshernitzan82} (see also~\cite[Theorem~5.20]{ADH5}).

\begin{theoremB}
Each of the following conditions on $H$ 
ensures that if \eqref{eq:ast} has an oscillating solution, then any $g$, $\phi$ as in Theorem~\textup{A} are perfectly $H$-hardian, with~$g$ unique up to multiplying by a positive real number and $\phi$  up to  
adding a real number:
\begin{enumerate}
\item[\textup{(a)}] there exists $\upl\in H$ such that $\upl-h^{\dagger\dagger} \le h^\dagger/n$ for all $h\in H^{>\R}$ and $n\ge 1$.
\item[\textup{(b)}] for all $\upo\in H$ such that $\upo/4$ generates oscillation, there exists 
$h\in H^{>\R}$ with $2h^{\dagger\dagger}{}'-(h^{\dagger\dagger})^2+(h^\dagger)^2 \le\upo$;
\end{enumerate}
\end{theoremB}

\noindent
Conditions \textup{(a)}, \textup{(b)} may seem strange, but are  optimal: 
 if neither \textup{(a)} nor~\textup{(b)} holds, then for some $f\in H$
there are infinitely many  $H$-hardian but not perfectly $H$-hardian pairs $(g,\phi)$ in Theorem~A, up to equivalence given by multiplying $g$ with a positive real constant and
adding a real constant to $\phi$:
Remark~\ref{rem:non-uniqueness}.
%In Section~\ref{sec:E(H) upo-free} we show they cannot be dropped: Theorems~\ref{thm:Bosh} and~\ref{thm:upo-freeness of the perfect hull}.
For readers familiar with~[ADH] we mention that  \textup{(a)} is equivalent to $H$  not being $\upl$-free, and that \textup{(b)}
holds whenever~$H$ is $\upo$-free. Thus (a) holds if $H$ has finite rank in the sense of Rosenlicht~\cite{Rosenlicht83}.
Any $H$ has by  \cite{ADH5} a Hardy field extension, differentially algebraic over $H$, that satisfies (b). 
The Hardy field $H_{\text{LE}}$ of LE-functions satisfies~\textup{(b)}.

If our hardian germ $f$ is differentially algebraic, then the Hardy field~$\Q(f,f',\dots)$
generated by $f$ satisfies~\textup{(a)}. This  leads to a handy description of all solutions to~\eqref{eq:ast} in this case. To state it, 
recursively define the (perfectly hardian) iterated logarithms~$\ell_0,\ell_1,\dots$ of~$x$ by~$\ell_0:=x$ and~$\ell_{n+1}:=\log \ell_n$, and   put
$$\upo_n \, :=\,  \frac{1}{\ell_0^2}+\frac{1}{(\ell_0\ell_1)^2}+\cdots+\frac{1}{(\ell_0\cdots\ell_{n})^2}.$$
By the remarks before Theorem~A, if   $f$ does not generate oscillation, then each solution to \eqref{eq:ast} is perfectly $H$-hardian.
If $f\le\upo_n/4$ for some~$n$, then $f$ does not generate oscillation; the converse holds when
$f$ is differentially algebraic by~\cite[Theorem~17.7]{Boshernitzan82}; see also \cite[Corollary~7.9]{ADH5}. We can add to this: 

\begin{cor1}
Suppose $f$ is differentially algebraic. Then $f$ generates oscillation iff 
 $f>\upo_n/4$ for all $n$. In this case,  there is an $H$-hardian germ $\phi>\R$ such that~$y\in\c^2$ is a solution to \eqref{eq:ast} iff   $y=c(\phi')^{-1/2}\cos(\phi+d)$ for some~$c,d\in\R$, and
any such  $\phi$ is perfectly $H$-hardian and  unique up to  
adding a real constant.
\end{cor1}

\noindent
From $f$ we obtain crude information about the asymptotics of the phase~$\phi$. The next corollary gives one example;
for more such results, see
Section~\ref{sec:perfect applications}.
Notation for~${g,h\in\c}$: 
$g\prec h$ (or $h\succ g$)  iff~$h(t)\ne 0$ eventually and $g(t)/h(t)\to 0$ as $t\to+\infty$,
and 
$g\sim h$ iff $g-h\prec h$, that is, $h(t)\ne 0$ eventually and $g(t)/h(t)\to 1$ as $t\to+\infty$.
%$h(t)\ne 0$ eventually and~$g(t)/h(t)\to 1$ as $t\to+\infty$.

\begin{cor2}
Suppose $f\succ x^{-2}$ and $f>0$. Then $f$ generates oscillation, and any~$\phi$  as in Theorem~\textup{A} is perfectly $H$-hardian, unique up to adding a real constant, and satisfies ${\phi\sim  h}$ where $h\in\c^1$ is such that
 $h'=\sqrt{f}$.
In particular, in case~${f\sim cx^{-2+r}}$~\textup{(}$c,r\in\R^>$\textup{)} we have~${\phi\sim \frac{2\sqrt{c}}{r}x^{r/2}}$ for such $\phi$.   
\end{cor2}

\noindent
It is important that any $H$-hardian germs $g$, $\phi$ as in Theorem~A 
 are differentially algebraic over~$H$:  for $z:=2\phi'$ we have
 $$2zz''-3(z')^2+z^4 - 4fz^2\  =\  0, \qquad g=c\sqrt{z} \text{ for some }c\in\R^>.$$
 Consequence: if $f$ has a $\c^{\infty}$-representative $(a,+\infty)\to\R$ for some $a\in\R$, then so do~$g$,~$\phi$; likewise
 with ``real-analytic'' in place of ``$\c^{\infty}$''; see \cite[Section~7]{ADH2}. 
 Often $\phi$ (and hence $g$) can be described more explicitly in terms of $f$: 
 %for example in case {\em generalized quadratures} in the sense of Liouville.
%The next corollary restricts the germs of phase functions that arise this way:

\begin{cor3}
Suppose $f$ is differentially algebraic and generates oscillation. 
Let  $\phi$ be as in Corollary~\textup{1}, and suppose
    $\phi$ lies   in
the smallest Liouville closed Hardy field containing $H:=\R(f, f', f'',\dots)$. Then $\phi'=\sqrt{h}$ for some $h\in H^>$.
\end{cor3}

\noindent
We can test, for any $f\in\Q(x)^\times$ as input, whether $f$ generates oscillation: compute~$c\in \Q^\times$ and
$k\in \Z$ such that $f(t)=ct^k+O(t^{k-1})$ as $t\to +\infty$; then $f$ generates oscillation iff either ($k>-2$ and $c>0$), or 
($k=2$ and $c > \frac{1}{4}$). We can also test for~$f\in \Q(x)^\times$ generating oscillation whether
some $\phi$  satisfies the condition in Corollary 3; see the remarks following Corollary~\ref{cor:phi' quadratic}.

\begin{example}[{Chebyshev's equation~\cite[2.235]{Kamke}}]
%on the interval $(1,+\infty$)}]
This is the linear differential equation
$$(x^2-1)Z''+xZ'+\alpha Z\ =\ 0\qquad (\alpha\in\R^\times).$$
A germ $z\in\c^2$  is a solution to it iff    $y:=(x^2-1)^{1/4}z$ is a solution to~\eqref{eq:ast}
with
$$f\ =\ f_\alpha\ :=\ \frac{(\alpha+\frac{1}{4})x^2-\alpha+\frac{1}{2}}{(x^2-1)^2}\in H:=\R(x).$$
If $\alpha>0$, then $f$ generates oscillation, and \eqref{eq:ast} has the (Liouvillian) phase function 
$$\phi\ =\ \beta\operatorname{arcosh}(x)\  :=\  \beta\log \left(x+{\sqrt {x^{2}-1}}\right),\qquad( \beta:=\sqrt{\alpha},\ 
\phi' =  \frac {\beta}{\sqrt{x^{2}-1}}).$$
If $\alpha<0$, then  $f$ does not generate oscillation and the above Chebychev equation has the $\R$-linearly independent hardian solutions $$\exp\big({-\beta\operatorname{arcosh}(x)}\big), \qquad \exp\big(\beta\operatorname{arcosh}(x)\big), \qquad(\beta:=\sqrt{-\alpha}). $$
It is not an accident  that the set of $\alpha\in\R^\times$ such that $f_\alpha$ generates oscillation   is an interval, in this case ~$(0,+\infty)$: see~\cite[Corollary~3]{ADH2}.)
\end{example}

\noindent
With our hardian phase $\phi$ from Theorem~A we control  the zeros and critical points of oscillating
solutions to \eqref{eq:ast} by hardian germs.
This is made explicit 
in the next two corollaries, where $y\in\c^2$ denotes an oscillating solution to~\eqref{eq:ast}; we also take
 an~${e\in\R}$ and a continuous representative $[e,+\infty)\to\R$ of~$f$
  and a
twice continuously differentiable representative $[e,+\infty)\to\R$
of $y$, also denoted by~$f$,~$y$, respectively, such that~${y''+fy=0}$ on $[e,+\infty)$.
Then the subset~$y^{-1}(0)$ of~$[a,+\infty)$ of zeros of~$y$ is closed, has no limit point, and is unbounded. This yields the enumeration~${s_0<s_1<s_2< \cdots}$ of $y^{-1}(0)$ with $s_n\to+\infty$ as~$n\to\infty$. Increasing~$e$ if necessary and restricting $f$,~$y$ accordingly we arrange that $f(t)\ne 0$ for all~$t\in [e,+\infty)$. Then each interval~$(s_n,s_{n+1})$
contains exactly one zero $t_n$ of $y'$, and  $t_0 < t_1 < t_2 < \cdots$  enumerates the zeros of $y'$ that are $\ge s_0$. Moreover: 

\begin{cor4}
There are  
continuous functions~$\zeta_0,\zeta_1,u\colon [n_1,+\infty)\to\R$, $(n_1\in\N$, $n_1\ge e)$
with hardian germs
such that $\zeta_0(t), \zeta_1(t)\to +\infty$ as $t\to +\infty$ and
\begin{enumerate}
\item[\textup{(i)}]   $\zeta_0$, $\zeta_1$ are strictly increasing  with $H$-hardian  compositional inverses; and
\item[\textup{(ii)}] $s_n=\zeta_0(n)$, $t_n=\zeta_1(n)$,  $\abs{y(t_n)}=u(n)$, for $n\ge n_1$.
\end{enumerate}
If $f$ is differentially algebraic, then we can choose such $\zeta_0$, $\zeta_1$, $u$ to have differentially algebraic germs.
\end{cor4}

\noindent
If $g$, $\phi$ in  Theorem~A satisfy $\phi\succ \log x$, as is the case for $f\succ x^{-2}$ by Corollary~2, then~$\abs{y(t_n)}\sim g(t_n)$   as $n\to\infty$.
%$$\abs{y(t_n)}\ \sim\ g(t_n)\  \text{ as } n\to\infty.$$
(In contrast, the hardian germ $u$ interpolates {\em precisely\/} the values~$\abs{y(t_n)}$ of $\abs{y}$, eventually.)
We do not know whether~$\zeta_0$,~$\zeta_1$,~$u$ can always be chosen to have $H$-hardian germs.
(For a case where they can, see~\cite{ADHbf}.)
Nevertheless, interpolating  the sequence $(s_n)$ by a ``hardian'' function  provides in some cases good access to asymptotic properties of  $(s_n)$:

\begin{cor5}
Suppose $f\sim cx^{-2+r}$ \textup{(}$c\in\R^>$, $r\in\R^\ge$\textup{)}.
\begin{enumerate}
\item[\textup{(i)}]
If $r>0$, then
 $s_n\sim \big(r\pi n/2\sqrt{c}\big)^{2/r}$ as $n\to\infty$; and
\item[\textup{(ii)}] if~$r=0$ and ${f\in\R(x)}$, then $c>\frac{1}{4}$ and
$\log s_n \sim   \pi n/\sqrt{c-\frac{1}{4}}$ as~$n\to\infty$.
\end{enumerate}
\end{cor5}

%\noindent
%At the root of the proof of the last statement in Corollary~3 is the fact that

\noindent
For the rest of this introduction we   focus on the case   $f\in \R(x)$, ${f\sim c}$~(${c\in\R^>}$), which appears in various applications. 
Then there are~$c_n\in\R$~($n\ge 1$) with
$$f \ \sim\ c+c_1 x^{-1}+c_2 x^{-2}+\cdots,$$
where we use the sign $\sim$ not in the sense of comparing germs, but to indicate an 
{\em asymptotic expansion}:
 for a sequence $(h_n)$ in $\c$ with $h_0\succ h_1\succ h_2 \succ\cdots$ we say that~$h\in \c$ has
the asymptotic expansion $$h\  \sim\ a_0h_0 + a_1h_1 + a_2h_2+\cdots \qquad(a_n\in\R)$$ if
$h-(a_0h_0+\cdots+ a_nh_n)\prec h_{n}$ for all $n$ (and then the sequence $a_0, a_1, a_2,\dots$ of coefficients is uniquely determined by
$h, h_0, h_1, h_2,\dots$).  Let $\phi$ be as in Corollary~1, so 
$\phi$ is perfectly hardian and unique up to addition of a constant.    Corollary~2 gives~$\phi\sim \sqrt{c}\,x$, so $\phi'\sim \sqrt{c}$
by $\phi$ being hardian, 
suggesting we can  improve this 
to an asymptotic expansion
for $\phi$ as follows:  with $z:=2\phi'$,
make an ``An\-satz''
\begin{align*} z\ &\sim\ z_0+z_1 x^{-1}+z_2 x^{-2}+z_3 x^{-3}+z_4 x^{-4}+\cdots \quad (z_n\in\R,\ z_0=2\sqrt{c}),\text{ so}\\
\phi\ &\sim \  \frac{z_0}{2} x + \frac{z_1}{2}\log x+ \text{const}-\frac{z_2}{2}x^{-1}- \frac{z_3}{4} x^{-2}-\frac{z_4}{6}x^{-3}-\cdots
\end{align*}
by termwise integration. Now formally substitute the ``Ansatz'' for $z$ in 
\begin{align*} 
\sigma(z) 	&\ :=\ \frac{2zz''-3(z')^2+z^4}{z^2}, \text{ which gives}  \\
			&\  z_0^2 + (2 z_0 z_1 )x^{-1} + (z_1^2 + 2 z_0 z_2 )x^{-2} + \frac{2 (z_0 z_1 z_2  + 2 z_1 +  z_0^2 z_3)}{z_0}x^{-3} + {}\\ &\qquad\qquad \frac{-7 z_1^2 + 2 z_0^2 z_1 z_3  + z_0^2 z_2^2  + 12 z_0 z_2  + 2 z_0^3 z_4 }{z_0^2}x^{-4} + \cdots
\end{align*}
and  successively obtain $z_1,z_2,\dots$ in terms of $c_1, c_2,\dots$ by comparing
coefficients of powers of~$x^{-1}$ on both sides of the equality $\sigma(2\phi')=4f$.\marginpar{to be justified in [8]} 

Relying on our model theoretic transfer results from \cite{ADH2},
we shall rigorously justify this heuristic  in~\cite{ADHbf}, a companion   to the present paper.
(The reader will meet the differential rational function $\sigma$ again in Section~\ref{sec:diffalg} below.) %\marginpar{add remarks about improved way to compute $1/\phi'$?}

\medskip\noindent
Let $f\in \R(x)$ and $f\sim c\in \R^{>}$ as before, so $f$ generates oscillation. Let $a\in \R$ be such that $f$ has no pole on $[a,+\infty)$ and let $y\colon [a,+\infty)\to \R$ be twice continuously differentiable with $y''+fy=0$, $y\ne 0$. Let
 $s_0<s_1<\cdots$ be the sequence of zeros of $y$. 
Then~${s_n \sim \pi n/\sqrt{c}}$ as $n\to\infty$, by Corollary~5(i).
Corollary~4  gives an~$n_0\in\N$ and a continuous strictly increasing~$\zeta\colon [n_0,+\infty)\to\R$,
with differentially algebraic hardian germ such that~${\zeta(n)=s_n}$ for all $n\ge n_0$.
In~\cite{ADHbf} we   also  obtain an asymptotic expansion for $\zeta$ (and hence for $s_n$),
and we apply these Hardy field techniques to obtain   classical results about Bessel functions and their zeros,   traditionally proved via
complex analysis, 
in an intuitive ``algebraic'' way. The same techniques are readily applicable to other special functions; with this in
mind, here is another classical linear differential equation:

\begin{example}[Coulomb wave equation]
This is the equation~\eqref{eq:ast} for 
$$f\ =\ 1-2\eta \,x^{-1}-l(l+1)\,x^{-2}\in\R(x)\qquad (\eta\in\R,\ l\in\N).$$
Its solutions are the germs of Coulomb wave functions (which give rise to solutions to the  Schr\"odinger equation for the Coulomb potential  in spherical coordinates;
cf.~\cite[\S{}3.6]{BW}).
Clearly $f>\upo_n/4$ for all $n$, so there is a perfectly hardian germ~${\phi>\R}$, unique up to addition of a constant,
such that the solutions to \eqref{eq:ast} are the germs~$y=c(\phi')^{-1/2}\cos(\phi+d)$ where~$c,d\in\R$. 
The germ $\phi$ is differential algebraic, and $\phi\sim x$ by Corollary~2.
In fact,   the approach sketched above yields an asymptotic expansion
%\marginpar{changed $\eta/4$ into $\eta/2$ in RHS} 
$$\phi\  \sim\  x-\eta\log x + \text{const}+ \frac{1}{2}\big( l(l+1) + \eta^2 \big) x^{-1} + \frac{\eta}{4}\big( l(l+1)+\eta^2 -1 \big)x^{-2} +  \cdots $$ 
Hence the germ $y$ of each Coulomb wave function satisfies, for some $c,d\in\R$:
%\marginpar{changed $\eta/4$ into $\eta/2$ in RHS}
\begin{multline*}
y\ =\  c\cos \!\left(x-\eta\log x +d+\frac{1}{2}\big( l(l+1) + \eta^2 \big) x^{-1} + \frac{\eta}{4}\big( l(l+1)+\eta^2 -1 \big)x^{-2}\right)\\ {}+ O(x^{-3}).
\end{multline*}
This improves the relation $y=c\cos(x-\eta\log x+d )+o(1)$   found 
in the literature; see~\cite[(6.7.11)]{BW}, \cite[(31.0.3)]{Temme}.
An asymptotic expansion of the zeros of~$y$ is in~\cite{GST}.
(Such an asymptotic expansion can also be obtained using tools 
in~\cite{ADHbf}.) \marginpar{to check in [8]}
\end{example}

\noindent
%We finish  by remarking that 
In a series of influential papers,~\cite{Bremer17, Bremer18,
Bremer23, BremerRokhlin16, BremerRokhlin17, HBR, HBRV}, Bremer et al.~(with precursors in~\cite{SV90, SV12}) have
%Heitman, and Rokhlin and others  
used~non-oscillatory phase functions  $\phi$ to give efficient algorithms for evaluating
various special functions and approximating their zeros. Here ``non-oscillatory'' means (at least) that~$\phi$, $\phi'$, $\phi''$ 
 are non-oscillating in the sense of this introduction (see \cite[(23) and subsequent discussion]{Bremer17}), and ideally,
 that these functions admit an asymptotic expansion as before. Quoting from \cite[p.~1]{HBR}:
``Phase functions have been extensively studied [\dots] Despite this long history, a useful property of phase functions appears to have been overlooked. Specifically, that when
[the function~$f$] is nonoscillatory, solutions of [the equation \eqref{eq:ast}] can be accurately represented using a nonoscillatory phase function.
This is somewhat surprising~[\dots].''
Theorems~A and~B above (and~\cite{ADHbf}) provide a natural, large class of linear differential equations~\eqref{eq:ast} 
to which the methods of Bremer et al.~may   be applicable, since their solutions
can be represented at $+\infty$ by phase functions which are hardian,   hence   strongly non-oscillating.

\subsection*{Boshernitzan's work} 
Theorem~A above in the case where $f$ is perfectly hardian  proves Conjecture~4  from~\cite[\S{}20]{Boshernitzan82}.
Theorems~A and~B   also hold with the self-adjoint  equation \eqref{eq:ast}
replaced by a  more general linear differential equation
\begin{equation}\tag{$\ast\ast$}\label{eq:astast}
Y''+f_0Y'+f_1Y\ =\ 0,
\end{equation}
where $f_0,f_1\in H$.  
Boshernitzan's seminal article~\cite{Boshernitzan87} announced the \eqref{eq:astast} version of
Theorem~A   as part of his Theorem~5.4.
But a proof of that never appeared. (An email exchange about this between Boshernitzan and one of us in 2014 was inconclusive.)  
  Theorem~\ref{thm:Bosh} below is a strengthening of Theorem~A, which 
includes the conditions \textup{(a)}, \textup{(b)}, for the uniqueness of the  germs~$g$,~$\phi$ from Theorem~B.
(See the remarks following Theorem~\ref{thm:Bosh}.)
 Our proof of this theorem relies on our earlier work~\cite{ADH5, ADH4, ADH2} on constructing dif\-fer\-en\-tial\-ly algebraic Hardy field extensions, the relevant results of which are reviewed in Section~\ref{sec:prelims}. 
  Boshernitzan~\cite[Remark on p. 117]{Boshernitzan87} did prove \marginpar{not yet checked} that every inhomogeneous linear differential equation
$$ Y''+f_0Y'+f_1Y\ =\ g,$$
where $f_0,f_1,g\in H$, has an $H$-hardian solution. This also follows from
 \cite{ADH2}; see Corollary~\ref{cor:maxhardymainthm, 2} below.

\subsection*{Organization of the paper}
We made an effort to make this paper accessible to readers unfamiliar with the material in our book [ADH] and~\cite{ADH5, ADH4, ADH2}.  This partly explains its length. 
Sections~\ref{sec:diffalg} and~\ref{sec:ADH}   set the stage by giving a brisk exposition of the algebraic approach to ``tame'' asymptotic analysis developed in [ADH].
In Sections~\ref{sec:germs} and~\ref{sec:prelims} we discuss germs, second-order linear differential equations, and Hardy fields, including the results from \cite{ADH5, ADH2} used later.
In Section~\ref{sec:perfect applications} we then prove Theorems~A and~B and Corollaries~1,~2, and~3.
Corollaries~4 and~5 and related results about the zeros of oscillating solutions to \eqref{eq:ast}  are proved
in Section~\ref{sec:zeros}. Finally, in Section~\ref{sec:E(H) upo-free} we prove the optimality of conditions \textup{(a)}, \textup{(b)} 
in Theorem~B for guaranteeing the ``uniqueness-up-to-constants'' of the phase function. 

\subsection*{Notations and conventions} 
We follow the conventions from [ADH]. Thus  $m$,~$n$   range over the set~$\N=\{0,1,2,\dots\}$ of natural numbers.  
If $A$ is an abelian group written additively, then $A^{\ne}:=A\setminus\{0\}$.
For a commutative ring $R$ (always with multiplicative identity $1$),  $R^\times$ denotes the multiplicative group of units of $R$.
Thus if $K$ is a field then $K^{\ne}=K^\times$. By convention, the ordering of an ordered abelian group and of an ordered field is total. 
If $\Gamma$ is an ordered abelian group written additively, then $\Gamma^>:=\{\gamma\in\Gamma:\gamma>0\}$,
and likewise with $\ge$, $<$, or $\le$ in place of~$>$.
Let $S$ be a totally ordered set and~$A\subseteq S$. We say that $A$ is {\it upward closed}\/ in $S$ 
({\it downward closed}\/ in $S$) if
it contains each~$s\in S$ such that $s\ge a$ for some~$a\in A$
($s\le a$ for some $a\in A$, respectively). We denote by~$A^\uparrow$ the smallest subset of~$S$ containing $A$
which is upward closed in~$S$, and by~$A^\downarrow$ the smallest subset of $S$ containing $A$
which is downward closed in $S$.

\subsection*{Acknowledgements}
Joris van der Hoeven has been supported by an ERC-2023-ADG grant for the
ODELIX project (number 101142171). Funded by the European Union. Views and opinions expressed are however those
of the authors only and do not necessarily reflect those of the European
Union or the European Research Council Executive Agency. Neither the European
Union nor the granting authority can be held responsible for them.

\section{Differential Algebra}\label{sec:diffalg}

\noindent 
In this section  we  quickly review some basic concepts of differential algebra with focus on linear differential operators.
% as well the {\it universal exponential extension}\/ of an algebraically closed differential field and {\it eigenvalues}\/ of linear differential operators  from  \cite{ADH4}.

 \subsection*{Differential rings}
{\it Throughout this section $R$ is a differential ring}:  a commutative ring equipped with a derivation~${\der\colon R \to R}$ and containing (an isomorphic copy of)~$\Q$~as a subring. 
Then $C_R:=\ker\der$ is a subring of~$R$, called the ring of constants of~$R$, and~${\Q\subseteq C_R}$.  
We  set $C=C_R$ if $R$ is understood from the context.
When the derivation $\der$ of~$R$ is clear from the context and $a\in R$, then we also denote $\der(a),\der^2(a)$ by $a', a''$,
and $\der^n(a)$ by~$a^{(n)}$; for~$a\in R^\times$ we set~$a^\dagger:=a'/a$ (the {\it logarithmic derivative}\/  of $a$). Thus $(ab)^\dagger=a^\dagger + b^\dagger$ for~$a,b\in R^\times$, and so~$R^\dagger:= (R^\times)^\dagger$ is a subgroup of the additive group of $R$. 
We say that $R$ is {\it closed under integration}\/ if $\der(R)=R$.
 If $R$ is an integral domain, then the derivation on $R$ extends uniquely
to a derivation on the fraction field $F$ of $R$; we call $F$ equipped with this derivation the
{\it differential fraction field}\/ of $R$.

 \subsection*{Differential polynomials}
Let $R\{Y\}=R[Y, Y', Y'',\dots]$ denote the differential ring of differential polynomials in a differential indeterminate $Y$ over $R$. Each~${y\in R}$ yields the differential ring morphism~$P\mapsto P(y)\colon R\{Y\}\to R$ over $R$
with~$Y\mapsto y$. Let 
$P=P(Y)\in R\{Y\}$. Then~$\deg P$ denotes the (total) degree of~$P$, with~$\deg 0:=-\infty<\N$. We say that $P$ has {\it order}\/ at most~$r\in \N$ if~${P\in R[Y,Y',\dots, Y^{(r)}]}$. 
For~$a\in R$ we let~$P_{+a}:=P(a+Y)$ and~$P_{\times a}:=P(aY)$.  
%For such $\i$ we set 
%$$|\i|\ :=\ i_0+i_1+ \cdots + i_r, \qquad \|\i\|\ :=\ i_1+2i_2 + \dots + ri_r.$$
%The {\em degree\/} of $P\ne 0$ is 
%$$\deg P:=\max\big\{|\i|:\ P_{\i}\ne 0\big\}\in \N,$$ 
%and the {\em weight} of $P$ is
%$$\wt P:=\max\big\{\|\i\|:\ P_{\i}\ne 0\big\}\in \N.$$
If $P$ is  homogeneous   of degree $d\in\N$, then its {\it Riccati transform}\/~${\Ric(P)\in R\{Z\}}$, [ADH, p. 300],
satisfies   $$\Ric(P)(z)\ =\ P(y)/y^d\quad \text{ for $y\in R^\times$, $z=y^\dagger$.}$$
We put $R_n:=\Ric(Y^{(n)})\in\Q\{Z\}$, so 
\begin{equation}\label{eq:Rn}
R_0=1,\quad R_1=Z,\quad R_2=Z^2+Z',\quad\dots.
\end{equation}
We have the recursion $R_{n+1}=ZR_n+R_n'$, and by an easy induction on $n$ this  yields~$R_n=Z^n+\text{lower degree terms}$.

 \subsection*{Linear differential operators}
The ring~$R[\der]$ of {\it linear differential operators}\/ over the differential ring~$R$  is the ring extension of $R$ generated over $R$
by an element~$\der$ (using here the same symbol denoting the derivation of~$R$),
such that $R[\der]$ is  free as a left $R$-module with basis
$1=\der^0, \der=\der^1, \der^2, \der^3,\dots$, and
$\der a = a \der + a'$  for  $a\in R$; see [ADH, 5.1].  
Thus each $A\in R[\der]$ has  the form 
\begin{equation}\label{eq:lin op}
A\ =\ a_0+a_1\der+\cdots+a_n\der^n \qquad
(a_0,\dots,a_n\in R),
\end{equation}
and for such $A$ and $y\in R$ we put 
$$A(y)\ :=\ a_0y+a_1y'+\cdots+a_ny^{(n)}\in R.$$
Then $(AB)(y)=A\big(B(y)\big)$ for $A,B\in R[\der]$ and $y\in R$.
For $A$ as in \eqref{eq:lin op} and $R_0,R_1,\dots$ as in \eqref{eq:Rn} we also put 
$$\Ric(A)\ :=\ a_0R_0+a_1R_1+\cdots+a_nR_n \in R\{Z\},$$ 
so $\Ric(A)(z)=A(y)/y$ for  $y\in R^\times$, $z=y^\dagger$.

\medskip
\noindent
The kernel of $A\in R[\der]$ is by definition the   $C_R$-submodule
$$\ker A \ =\  \big\{ y\in R:\, A(y)=0 \big\}$$
of $R$. If we want to indicate the dependence on $R$, we write $\ker_R A$ for $\ker A$. 
For~$A\in R[\der]^{\neq}$ there are unique $n$ and elements $a_0,\dots,a_n$  of $R$ with $a_n\neq 0$ such that~\eqref{eq:lin op} holds. Then $\order A:=n$ is the {\it order}\/ of $A$, and we say that $A$ is {\it monic}\/ if~$a_n=1$. 
Let  $u\in R^\times$. For $A\in R[\der]$ we set $A_{\ltimes u}:=u^{-1}Au\in R[\der]$, the {\it twist}\/ of~$A$ by~$u$. If~$A$ is monic, then so
is $A_{\ltimes u}$, and $A\mapsto A_{\ltimes u}$ is an automorphism of the ring~$R[\der]$ 
which is the identity on $R$ [ADH, p.~243]; its inverse is $B\mapsto B_{\ltimes u^{-1}}$. 
 We say that~$A$ {\it splits}\/ over~$R$ if~$A=a(\der-b_1)\cdots(\der-b_n)$
for some~$a\in R^\times$  and~$b_1,\dots,b_n\in R$. 
If~$A$ splits over $R$, then so does $aAb$ for $a,b\in R^\times$. 
For $r\in\N$, we say that $R$ is {\it $r$-linearly closed}\/ if every monic $A\in R[\der]$ of order~$\le r$ splits over $R$. We also call $R$ {\it linearly closed}\/ if it is $r$-linearly closed for all $r\in\N$.
 %Let $\phi\in R^\times$. Then we have the ring $R^\phi[\derdelta]$ of linear differential operators over the differential ring~$R^\phi$ (with derivation $\derdelta=\phi^{-1}\der$), and we have a ring isomorphism~$A\mapsto A^\phi\colon R[\der]\to R^\phi[\derdelta]$;  it is the identity on $R$, with $\der^\phi=\phi\derdelta$.

\medskip
\noindent
In this paper we focus on the order $2$ case. Accordingly, let
$$A\ =\  \der^2+a_1\der+a_0,\qquad a_0, a_1\in R.$$
Assume also that $u\in R^\times$ satisfies $u^\dagger=-\frac{1}{2}a_1$. Put~$\tilde A:=4A_{\ltimes u}$;
then $$\tilde A\ =\  4\der^2+f\ \text{ where }f:=4a_0-2a_1'-a_1^2 \quad (\text{[ADH, example following~5.1.13]}).$$
Moreover,
$\tilde A$ splits over~$R$ iff $A$ does, and  
$y\mapsto uy\colon \ker \tilde A\to\ker A$ is an isomorphism
of $C_R$-modules. These remarks often allow us to reduce to the case $A=4\der^2+f$ where~$f\in R$.
(As in~[ADH],   the role of the factor~$4$ is just to make some later expressions prettier.)   
We define the function
$$\omega\colon R\to R, \qquad \omega(z)\ :=\ -4R_2(z/2)\ =\ -(2z'+z^2),$$
which is related to the Schwarzian derivative; cf.~[ADH, 5.2].
For $A=4\der^2+f$ (${f\in R}$) and  $y\in R^\times$ we then have
$A(y) = y\big(f-\omega(2y^\dagger)\big)$, hence~$f\in\omega(2R^\dagger)$ iff~${R^\times\cap\ker A\ne\emptyset}$.
As we shall see, an equally important role for parametrizing solutions to order $2$ linear differential equations  is played by the function
$$\sigma\colon R^\times\to R,\qquad \sigma(y)\ :=\    \omega(-y^\dagger)+y^2\  =\  
\frac{2yy''-3(y')^2 +y^4}{y^2}$$ 
first studied by Kummer~\cite{Kummer}.
Note:  for $y\in R^\times$ we have $\sigma(y)=\sigma(-y)$, and if~$\imag$ is an element of a differential ring extension of $R$ with $\imag^2=-1$, then
$\sigma(y)=\omega(-y^\dagger+y\imag)$. We also remark that for $f\in R$ and $y\in R^\times$, the equality $\sigma(y)=f$ is equivalent to~$P(y)=0$ where~$P(Y):=2YY''-3(Y')^2+Y^4-fY^2 \in R\{Y\}$.

\subsection*{Compositional conjugation}
Let   $\phi\in R^\times$. We let $R^{\phi}$ be the {\it compositional conjugate}\/ of~$R$ by~$\phi$: the differential ring
with the same underlying ring as~$R$ but with derivation~$\phi^{-1}\der$ (usually denoted by~$\derdelta$) instead of $\der$. 
We then have an $R$-algebra isomorphism~$P\mapsto P^\phi\colon R\{Y\}\to R^\phi\{Y\}$ with~$P^\phi(y)=P(y)$ for all~${y\in R}$;
see~[ADH, 5.7]. We also have the ring isomorphism $A\mapsto A^\phi\colon R[\der]\to R^\phi[\derdelta]$ with~$a^\phi=a$ for~$a\in R$ and $\der^\phi=\phi\derdelta$, via which we identify
$R[\der]$ with $R^\phi[\derdelta]$. Then
\begin{align*}
\der^1\  &=\ \phi\cdot \derdelta \\
\der^2\  &=\ \phi^2\cdot\derdelta^2 + \phi\cdot\derdelta(\phi)\cdot\derdelta \\
\der^3\  &=\  \phi^3\cdot\derdelta^3 + 3\phi^2\cdot\derdelta(\phi)\cdot\derdelta^2 
+ \big(\phi^2\cdot\derdelta^2(\phi)+\phi\cdot \derdelta(\phi)^2\big)\cdot\derdelta \\
&\vdots\\
\der^n\ &=\ G^n_n(\phi)\cdot\derdelta^n+G^n_{n-1}(\phi)\cdot\derdelta^{n-1}+\cdots+G^n_1(\phi)\cdot\derdelta
\end{align*}
where $G^n_k\in\Q\{X\}\subseteq R^\phi\{X\}$ ($k=1,\dots,n$) are homogeneous of degree $n$ and independent of $R$, $\phi$. Note that then for $n\ge 1$,
$$(Y^{(n)})^\phi\  =\  G^n_n(\phi)Y^{(n)}+G^n_{n-1}(\phi)Y^{(n-1)}+\cdots+G^n_1(\phi)Y' \qquad\text{in $R^\phi\{Y\}$.}$$
In the context of differential rings of germs, compositional conjugation corresponds to a change of variables; see Section~\ref{sec:germs}.

\subsection*{Differential fields}
{\it In the rest of this section  $K$ is a differential field,}\/
that is, a differential ring whose underlying ring is a field.
Then  $C:=C_K$ is a subfield of $K$ and is algebraically closed in $K$; cf.~[ADH, 4.1.1].
If $L$ is an algebraic field extension of $K$, then the derivation on $K$ extends uniquely to a derivation on $L$ [ADH, 1.9.2],
and the constant field of the differential field $L$ is the algebraic closure of $C$ in~$L$~[ADH, 4.1.2].
 Moreover, $K\{Y\}$ is an integral domain; its differential fraction field  is denoted by~$K\<Y\>$.  Let~$y$ be an element of a differential field extension~$L$ of~$K$. Then~$K\{y\}$ and $K\<y\>$ denote respectively the   differential subring of~$L$ generated by~$y$ over $K$, and the differential fraction field of $K\{y\}$ in~$L$. 
We call~$y$ 
{\it differentially algebraic\/} over $K$ if  $P(y)=0$ for some~$P\in K\{Y\}^{\ne}$; otherwise~$y$ is  called {\it differentially transcendental\/} over $K$.
  As usual in [ADH], the prefix~``$\d$'' abbreviates ``differentially'', so ``$\d$-algebraic'' means ``differentially algebraic''.
We say that $L$ is 
 {\it $\d$-algebraic\/} over $K$ if each $y\in L$ is $\d$-algebraic over $K$.
 See [ADH, 4.1] for more on this. 
Given $r\in\N$,
we say that $K$ is {\it $r$-linearly surjective}\/ if $A(K)=K$ for each $A\in K[\der]^{\ne}$ of order~$\le r$.
If  $K$ is closed under integration and~$K^\dagger=K$, then~$K$ is $1$-linearly surjective. (See~[ADH, example following~5.5.22].) If $K$ is $r$-linearly closed and $1$-linearly surjective, then it is $r$-linearly surjective.

\subsection*{Linear differential operators over differential fields}
Let~$A\in K[\der]^{\ne}$ have order $r\in\N$. Then   the subspace $\ker A$ of the $C$-linear space $K$ has dimension~$\le r$; cf.~[ADH, 4.1.14]. 
By [ADH, 5.1.21], if~$r\ge 1$ and~$u$ is a nonzero element of a differential field extension of~$K$ with $g:=u^\dagger\in K$, then
\begin{equation}\label{eq:5.1.21}
A\in K[\der](\der-g)\quad\Longleftrightarrow\quad A(u)=0.
\end{equation}
In particular, for $r=2$ we have 
$\ker A\ne\{0\}  \Rightarrow \text{$A$ splits over $K$}$.
If also~$K^\dagger=K$, then
%$\dim_C\ker A=2  \Leftrightarrow  $
$\ker A\ne\{0\} \Leftrightarrow  
\text{$A$ splits over $K$.}$
Moreover:

\begin{lemma}\label{lem:2.4.5}
Suppose $K$ is closed under integration and $K^\dagger=K$. Then 
$$\dim_{C}\ker A=r  \quad \Longleftrightarrow\quad  \text{$A$ splits over $K$.}$$
\end{lemma}
\begin{proof}
We show ``$\Rightarrow$'' by induction on $r$. The case~$r=0$ is trivial.  Let $\dim_{C}\ker A=r\ge 1$. 
Now~\eqref{eq:5.1.21} yields~$B\in K[\der]$ and $g\in K$ such that~$A=B(\der-g)$.
Then $\dim_C\ker B=r-1=\order B$ by  [ADH, 5.1.12], so by the inductive hypothesis, $B$ 
splits over~$K$; hence so does~$A$. For~``$\Leftarrow$''  use~[ADH, 5.1.27(iv)] with $r=1$.
\end{proof}

\noindent
{\em In the rest of this subsection, $A=4\der^2+f$ where $f\in K$}. 
Then
\begin{align} 
\label{eq:omega 2}
 \ker A\ne\{0\}  &\quad \Longleftrightarrow\quad  f\in\omega(2K^\dagger). 
\intertext{In [ADH, 5.2] we  also showed:}
\label{eq:omega}
\text{$A$ splits over $K$}  &\quad\Longleftrightarrow\quad f\in \omega(K).
\end{align}
The next lemma 
relates the kernels of the operators $A$ and  ${B:=\der^3+f\der+(f'/2})$. It  goes back to Appell~\cite{Appell}.
(See  \cite[Proposition~4.26(1)]{vdPS} for a more general version.)

\begin{lemma} \label{lem:Appell}
Let   $y,z\in \ker A$. Then $yz\in \ker B$. 
Moreover, if~$y$,~$z$ is a basis of the $C$-linear space $\ker A$, then~$y^2, yz, z^2$ is a basis of $\ker B$.
\end{lemma}
\begin{proof}
We have
$$(yz)'\ =\ y'z+yz',\quad (yz)''\ =\ y''z+2y'z'+yz''\ =\ 2y'z'-(f/2)yz,$$
hence
\begin{align*}
(yz)'''\ 	&=\  2y''z'+2y'z''-(f'/2)yz-(f/2)(yz)' \\
		&=\  -(f/2)(yz'+y'z)-(f'/2)yz-(f/2)(yz)' \\
		&=\  -f(yz)'-(f'/2)yz,
\end{align*}
and so $yz\in \ker B$.
Suppose  $ay^2+byz+cz^2=0$ where $a,b,c\in C$, not all zero;
we claim that then $y$, $z$ are  $C$-linearly dependent.  We have
  $a\neq 0$ or $c\neq 0$, and so we may assume $a\neq 0$, $z\ne 0$. 
Then  $u:=y/z$ satisfies~$au^2+bu+c=0$, so~$u\in C$, hence~$y\in Cz$. 
\end{proof}

\subsection*{The case of real operators} 
{\it In this subsection $K=H[\imag]$ where~$H$ is a real closed differential subfield of $K$ 
with
$H^\dagger=H$ and $\imag^2=-1$.}\/ Then $K$ is algebraically closed; so $C$ is algebraically closed
and $K^\dagger$ is divisible.
Let $A = 4\der^2+f$, $f\in H$. 
By [ADH, 5.2] we have:
\begin{equation}\label{eq:sigma}
\text{$A$ splits over $K$}  \quad\Longleftrightarrow\quad f\in \omega(H)\cup\sigma(H^\times).
\end{equation}
Suppose  $y\in H^\times$ and $f=\sigma(y)\notin\omega(H)$. Put $g:=\frac{1}{2}(-y^\dagger+y\imag)$.
Then [ADH, p.~262] gives~${A=4({\der+g})({\der-g})}$.
Also $y\imag\notin K^\dagger$: otherwise $g\in K^\dagger$, hence $\ker_K A\ne \{0\}$  by~\eqref{eq:5.1.21},
so $\ker_H A\ne \{0\}$, and thus   $f\in\omega(H)$ by~\eqref{eq:omega}, a contradiction.

\medskip
\noindent
Following [ADH, p.~263] we now indicate a
differential field extension~$L$ of~$K$  such that $C_L=C$ and $\dim_C \ker_L A = 2$, and  specify  
a basis of $\ker_L A$.
First, take~${r\in K}$ with $r^2=y$. Next,  [ADH, 4.1.9] provides an element $e(y) \ne 0$ in a differential field
extension of $K$ with $e(y)^\dagger=\frac{1}{2}y\imag$. (Think of $e(y)$ as~$\exp(\frac{1}{2}\imag\int y)$.)
Put~$L:=K\big(e(y)\big)$.
Then $y\imag\notin K^\dagger$ gives $e(y)\notin K$, so $e(y)$ is transcendental over $K$ and~${C_L=C}$ by~[ADH, 4.6.11, 4.6.12]. We also obtain from~[ADH, 4.6.11, 4.6.12] that for each unit $e$ of a differential ring extension $R$
of $K$ with $e^\dagger=\frac{1}{2}y\imag$, the differential subring~$K[e,e^{-1}]$ of $R$ is an integral domain, $e$ is transcendental over $K$,
and thus we have a differential field isomorphism $K\langle e\rangle=K(e) \to L$ over $K$ with $e\mapsto e(y)$.

\medskip
\noindent
Next set
$$e(-y)\ :=\ e(y)^{-1}, \qquad y_1\ :=\ \frac{e(y)}{r}, \qquad y_2\ :=\ \frac{e(-y)}{r}.$$
Then $2y_1^\dagger = -y^\dagger + y\imag$ and $2y_2^\dagger = -y^\dagger-y\imag$, so
$L=K(y_1,y_2)=K\langle y_1,y_2\rangle$, and
$$\omega(2y_1^\dagger)\ =\ \sigma(y)\  =\  \sigma(-y)\ =\ \omega(2y_2^\dagger)\ =\  f,\ \text{ so }\ y_1, y_2\in \ker_L A.$$
Also $(y_1/y_2)^\dagger = y_1^\dagger - y_2^\dagger = y\imag \ne 0$, hence 
 $y_1, y_2$ are linearly independent over $C$.

\medskip
\noindent
To illustrate these remarks, we now consider again  the   operator  
$$B\ :=\ \der^3+f\der+(f'/2)\in H[\der]$$ 
to obtain a fact
used in the proof of Lemma~\ref{lem:B, 2} below:

\begin{lemma}\label{lem:kerB}
Let $y\in H^\times$ and~$\sigma(y)=f\notin\omega(H)$. Then  $\dim_{C_H} \ker_H  B=1$.
\end{lemma}
\begin{proof}
Take $r$, $e(y)$, $y_1$, $y_2$, $L$ as above. Then 
$$y_1^2\ =\ \frac{e(y)^2}{y}, \qquad y_1y_2\ =\ \frac{1}{y},\qquad y_2^2\ =\ \frac{e(-y)^2}{y}$$
is a basis of the $C$-linear space $\ker_L B$ by  Lemma~\ref{lem:Appell}.
Now $e(y)$ is transcendental over $K$, so $\ker_H B = H\cap \ker_L B = C_H(1/y)$.
\end{proof}

\subsection*{Splittings and derivatives} 
This subsection is only used for proving Lemma~\ref{lem:param zeros of derivatives}, 
and thus for Corollary~3 in the introduction, but not for the proof of  Theorems~A and~B.
Let $A\in K[\der]^{\ne}$ be monic of order~$r\in\N$ with~$a_0:=A(1)\neq 0$.
Let $A^\der$ be the unique element of $K[\der]$ such that  $A^\der \der = \der A - a_0^\dagger A$.
Then $A^\der$ is monic of order $r$, and if~$A\in H[\der]$ with $H$ a differential subfield of $K$, then also $A^\der\in H[\der]$.

\begin{examples}
If $\order A=0$ then $A^\der=1$, and if $\order A=1$ then $A^\der=\der+(a_0-a_0^\dagger)$.
If $A=\der^2+a_1\der+a_0$ ($a_0,a_1\in K$, $a_0\ne 0$),  then
$$A^\der\ =\ \der^2+(a_1-a_0^\dagger)\der+(a_1'+a_0-a_1a_0^\dagger).$$
\end{examples}

\noindent
If  $A(y)=0$ with $y$ in a differential ring extension of $K$, then~$A^\der(y')=0$. Also:

\begin{lemma}\label{lem:A^der}
Let $R$ be a differential integral domain extending $K$. Suppose the differential fraction field of $R$ has constant field $C$, and  $\dim_C \ker_R A=r$.
Then $$\ker_R A^\der=\{y':y\in\ker_R A\}$$ and $\dim_C \ker_R A^\der=r$.
\end{lemma}
\begin{proof}
Let $y_1,\dots,y_r$ be a basis of the $C$-linear space $\ker_R A$, and assume towards a contradiction that $c_1y_1'+\cdots+c_ry_r'=0$ with $c_1,\dots,c_r\in C$ not all zero. 
 Then~$y:=c_1y_1+\cdots+c_ry_r\in\ker^{\neq}_R A$ and~$y'=0$, so $y\in C^\times$.  Hence $0=A(y)=yA(1)=ya_0$ and thus $a_0=0$, contradicting our assumption in this subsection that $a_0\ne 0$.
\end{proof}

%\noindent
%Let $\Univ=\Univ_K$ and $f_1\ex(\lambda_1),\dots,f_r\ex(\lambda_r)\in\Univ^\times$ be a basis of the $C$-linear space $\ker_{\Univ} A$, where~$f_j\in K^\times$ and $\lambda_j\in\Lambda$ for $j=1,\dots,r$. Then by Lemma~\ref{lem:A^der},
%$$(f_1'+\lambda_1f_1)\ex(\lambda_1),\dots,(f_r'+\lambda_rf_r)\ex(\lambda_r)\in\Univ^\times$$ 
%is a basis of the $C$-linear space $\ker_{\Univ} A^\der$. Hence by~\eqref{eq:2.4.6}:

%\begin{cor}\label{cor:A^der}
%If $\dim_C \ker_{\Univ} A=r$, then $\mult_\alpha(A)=\mult_\alpha(A^\der)$ for all~$\alpha\in K/K^\dagger$, so $\Sigma(A)=\Sigma(A^\der)$, and both $A$, $A^\der$ split over $K$.
%\end{cor}

%\noindent
%Using \eqref{eq:2.4.8}, Corollary~\ref{cor:A^der} yields: if $K$ is $1$-linearly surjective when~$r\geq 2$ and~$A$ splits over~$K$, then~$A^\der$ splits over $K$.

\section{Asymptotic Differential Algebra}\label{sec:ADH}

\noindent
Here we introduce
various objects studied in~[ADH] and recall a few basic facts about them proved there.  
We also include a short proof of a theorem of Kolchin (Corollary~\ref{cor:Kolchin}) and
 some technical results (Lemmas~\ref{lem:ADH 14.2.18} and~\ref{lem:upl-free}) needed later.

\subsection*{Valued fields}
A (Krull) valuation on a field $K$ is a surjective map~$v\colon K^\times \to \Gamma$ onto an ordered abelian 
group $\Gamma$ (additively written) satisfying the usual laws, and extended to
$v\colon K \to \Gamma_{\infty}:=\Gamma\cup\{\infty\}$ by $v(0)=\infty$,
where the ordering on $\Gamma$ is extended to a total ordering
on $\Gamma_{\infty}$ by $\gamma<\infty$ for all~$\gamma\in \Gamma$. 
Given such a valuation~$v$ on $K$,  $\mathcal O=\{a\in K:va\ge 0\}$ is a valuation ring of $K$;
conversely, every valuation ring $\mathcal O$ on a field $K$ gives rise to a valuation 
$v\colon K^\times \to \Gamma$  on $K$  
such that~$\mathcal O=\{a\in K:va\geq 0\}$. (See [ADH, 3.1].)
A {\em valued field\/} $K$ is a field 
%(also denoted by $K$) 
together with a valuation ring~$\mathcal O$ of that field.
 Let $K$ be a valued field, let $\mathcal{O}$ be its valuation ring with $v\colon K^\times \to \Gamma$ the corresponding valuation. Then~$\mathcal O$ is a local ring  with maximal ideal $\smallo:=\{a\in K:va>0\}$. From now on $K$ (and any  valued field that gets menstioned)  has {\it equicharacteristic zero}\/, that is, the
residue field~$\res(K):=\mathcal{O}/\smallo$ of $K$   has characteristic zero.
With~$a$,~$b$ ranging over~$K$, set 
\begin{align*} a\asymp b &\ :\Leftrightarrow\ va =vb, & a\preceq b&\ :\Leftrightarrow\ va\ge vb, & a\prec b &\ :\Leftrightarrow\  va>vb,\\
a\succeq b &\ :\Leftrightarrow\ b \preceq a, &
a\succ b &\ :\Leftrightarrow\ b\prec a, & a\sim b &\ :\Leftrightarrow\ a-b\prec a.
\end{align*}
Thus $\mathcal O=\{a:a\preceq 1\}$ and $\smallo=\{a:a\prec 1\}$.
If $a\sim b$, then~$a, b\ne 0$ and $a\asymp b$; moreover, $\sim$ is an equivalence relation on $K^\times$.   

\medskip
\noindent
We have the residue morphism $a\mapsto \overline{a}:=a+\smallo\colon\mathcal O\to\res(K)$, which we extend to the ring morphism $P\mapsto\bar{P}\colon\mathcal O[Y]\to \res(K)[Y]$ with $Y\mapsto Y$.
 The valued field~$K$ is called {\it henselian} if  for each $P\in\mathcal O[Y]$ and $u\in\mathcal O$ such that $\bar{P}(\bar{u})=0$, $\big(\frac{\partial\bar{P}}{\partial Y}\big)(\bar{u})\ne 0$, there exists $y\in u+\smallo$ with $P(y)=0$. (See [ADH, 3.3].)

\medskip
\noindent
Let now~$L$ be a valued field extension of $K$; then
the relations~$\asymp$,~$\preceq$, \dots~on~$L$ displayed above restrict to the corresponding relations on $K$, and
we identify in the usual way the value group~${\Gamma_K}$ of~$K$ with an ordered subgroup of the value group~$\Gamma_L$ of~$L$
and~$\res(K)$ with a subfield of $\res(L)$. 
Such a valued field extension is called {\it immediate} if  ${\Gamma_K=\Gamma_L}$ and $\res(K)=\res(L)$. 
If $L$ is algebraic over $K$, then the group~$\Gamma_L/\Gamma_K$ is torsion and the field extension $\res(L)\supseteq\res(K)$ is algebraic.

\subsection*{Valued differential fields}  
Let $K$ be
a {\em valued differential field}\/, that is, a valued field (of equicharacteristic zero by our earlier convention) together with a derivation on the underlying field.
We say that the derivation $\der$ of   $K$ is    {\it small}\/ if~$\der\smallo\subseteq\smallo$; then~${\der\mathcal O\subseteq\mathcal O}$~[ADH, 4.4.2], hence $\der$ induces a derivation  
 on the residue field $\res(K)$ of $K$ making the residue morphism~${\mathcal O\to\res(K)}$ a differential ring morphism.
 We also say that~$K$ is {\it monotone}\/ if~$f'\preceq f$ for all $f\prec 1$ in $K$. If $K$ is monotone, then it has small derivation;
 the converse holds whenever  $\Gamma$ is archimedean [ADH, 6.1.2].
 If $K$ has small derivation,  then every valued differential field extension
of~$K$ that is algebraic over~$K$ also has small derivation [ADH, 6.2.1]; likewise with ``is monotone'' in place of ``has small derivation''[ADH, 6.3.10].

\subsection*{Liouvillian zeros of linear differential operators}
Let $K$ be a differential field.
A {\it Liouville extension}\/ of $K$ 
is a differential field extension $E$ of $K$ such that $C_E$ is algebraic over $C$ and for each $t\in E$ there are~$t_1,\dots,t_n\in L$ such that~$t\in K(t_1,\dots,t_n)$ and for $i=1,\dots,n$:
 \begin{enumerate}
\item $t_i$ is algebraic over $K(t_1,\dots,t_{i-1})$, or
\item $t_i'\in K(t_1,\dots,t_{i-1})$, or
\item $t_i'\in t_i K(t_1,\dots,t_{i-1})$.
\end{enumerate}
See [ADH, p.~462] for basic properties of Liouville extensions, such as the following:

\begin{lemma}\label{lem:10.6.6}
Let $M|L$ and $L|K$ be differential field extensions. If $M|K$ is a Liouville extension, then so is $M|L$. If $M|L$ and $L|K$ are Liouville extensions, then so is $M|K$.
\end{lemma}

\noindent
In the proof of Lemma~\ref{lem:Liouvillian zeros} we shall use 
a theorem of Kolchin (Corollary~\ref{cor:Kolchin})  which gives information about linear differential operators $A\in K[\der]^{\ne}$
having
zeros~$y\ne 0$    in Liouville extensions of~$K$. 
We include a variant of Rosenlicht's short proof
of a more general result, since it illustrates the effectiveness of asymptotic differential-algebraic reasoning. For this we let
$m\geq 1$,  $P\in K\{Y\}$, $\deg P < m$. 

\begin{lemma}\label{ympy}
Let $L$ be a differential field extension of $K$ and $t\in L$, $t'\in K + tK$,
and suppose $L$ is algebraic over $K(t)$.
If $y^m=P(y)$ for some $y\in L$, then $z^m=P(z)$ for some $z$ in a differential field extension of $K$
which is algebraic over $K$.
\end{lemma}
\begin{proof}
We may assume $t$ is transcendental over $K$.
View $K(t)$ as a subfield of the Laurent series field $F=K(\!(t^{-1})\!)$.  We equip $F$ with the   valuation ring~$\mathcal O_F=K[[t^{-1}]]$ and the continuous derivation extending that of $K(t)$, cf.~[ADH, p.~226]. Then the valued differential field~$F$ is monotone.
Hence the valued differential subfield~$K(t)$ of $F$ is also monotone. We equip $L$ with a valuation ring $\mathcal O_L$ lying over~$\mathcal O_F\cap K(t)$; then~$L$ is monotone, since it is algebraic over $K(t)$. 
We  
identify $K$ with its image under the residue morphism $a\mapsto\bar{a}\colon\mathcal O_L\to\k_L:=\res(L)$. Then $K$ is a differential subfield of the differential residue field $\k_L$ of $L$,
and   $\k_L$ is algebraic over~$K$.
Let now~$y\in L$ with~$y^m=P(y)$, and towards a contradiction suppose~$y\succ 1$.
Then $y^\dagger\preceq 1$, thus~$y^{(n)}=y\,R_n(y^\dagger)\preceq y$ for all~$n$, with $R_n=\Ric(Y^{(n)})$ as in \eqref{eq:Rn}, and hence $y^m=P(y)\preceq y^d$ where~${d=\deg P<m}$, a contradiction. Thus~$y\preceq 1$, and $z:=\bar{y}\in\k_L$ has the required property. 
\end{proof}

\begin{prop}[{Rosenlicht~\cite[p.~371]{Rosenlicht73}}]
Suppose $y^m=P(y)$ for some $y$ in a Liouville extension of~$K$. Then $z^m=P(z)$ for some $z$ in a differential field extension of $K$ which is algebraic over $K$.
\end{prop}
\begin{proof}
A {\it Liouville sequence}\/ over $K$ is by definition a sequence $(t_1,\dots,t_n)$ of elements of a differential field extension $E$ of $K$ such that  for $i=1,\dots,n$:
\begin{enumerate}
\item $t_i$ is algebraic over $K(t_1,\dots,t_{i-1})$, or
\item $t_i'\in K(t_1,\dots,t_{i-1})$, or
\item $t_i'\in t_i K(t_1,\dots,t_{i-1})$.
\end{enumerate}
Note that then $K_i:=K(t_1,\dots,t_{i})$ is a differential subfield of $E$ for $i=1,\dots,n$.
By induction on $d\in\N$ we now show:  if $(t_1,\dots, t_n)$ is a Liouville sequence over~$K$
with~$\operatorname{trdeg}\!\big(K(t_1,\dots,t_n)|K\big) = d$ and $y^m=P(y)$ for
some $y\in K(t_1,\dots,t_n)$, then   the conclusion of the proposition holds. This is obvious for $d=0$, so let $(t_1,\dots, t_n)$ be a Liouville sequence over~$K$
with $\operatorname{trdeg}\!\big(K(t_1,\dots,t_n)|K\big) = d\ge 1$ and $y^m=P(y)$ for
some $y\in K(t_1,\dots,t_n)$.
Take $i\in\{1,\dots,n\}$ maximal such that $t_i$ is transcendental over~$K_{i-1}=K(t_1,\dots,t_{i-1})$.
Then $t_i'\in  K_{i-1}+t_i K_{i-1}$, and $K(t_1,\dots, t_n)$ is algebraic over $K(t_1,\dots, t_i)$. Applying Lemma~\ref{ympy} 
 to $K_{i-1}$, $t_i$ in the role of $K,t$ yields a~$z$ in an algebraic differential field extension of $K_{i-1}$ with  $z^m=P(z)$. Now apply the inductive hypothesis to the Liouville sequence~$(t_1,\dots,t_{i-1},z)$ over~$K$.
\end{proof}

\begin{cor}[Kolchin]\label{cor:Kolchin}
Let $A\in K[\der]^{\neq}$, and suppose $A(y)=0$ for some $y\neq 0$ in a Liouville extension of $K$.
Then $A(y)=0$ for some $y\neq 0$ in a differential field extension of $K$ with $y^\dagger$  algebraic over $K$.
\end{cor}
\begin{proof}
Let $r=\order A$ and
note that $\operatorname{Ri}(A)=Z^r+Q$ with $\deg Q<r$. Now apply  the proposition  above.
\end{proof}

\begin{remark}
Corollary~\ref{cor:Kolchin} goes back to Liouville~\cite{Liouville39} in an analytic setting for $A$ of or\-der~$2$ and $K=\C(x)$ with $C=\C$, $x'=1$.
\end{remark}

\subsection*{Asymptotic fields} An
\textit{asymptotic field}\/ is a valued differential field $K$ such that  for all     $f,g\prec 1$ in~$K^\times$: ${f\prec g}\Longleftrightarrow {f'\prec g'}$. This condition may be thought of as a valuation-theoretic
formulation of l'H\^{o}pital's Rule; it implies various other compatibilities between dominance and derivation, see [ADH, 9.1.3, 9.1.4]:

\begin{lemma}\label{lem:9.1.4(ii)}
Let $K$ be an asymptotic field, and $f,g\in K$. Then
\begin{align*} \text{ if }0\ne f,g\nasymp 1, \text{ then }&f\preceq g\quad \Longleftrightarrow\quad f'\preceq g',\\
\text{ if }f\preceq g\nasymp 1, \text{ then }& f\sim g\quad\Longleftrightarrow\quad f'\sim g'.
\end{align*} 
\end{lemma}

\noindent
Compositional conjugates of   asymptotic fields remain asymptotic. 
In the rest of this subsection, $K$ is an asymptotic field with constant field $C$. Then $C\subseteq\mathcal O$, and
if~$\mathcal O=C+\smallo$ (equivalently, for each $f\asymp 1$ in $K$ there is $c\in C$ with $f\sim c$), then $K$ is  called {\it differential-valued}\/ ({\it $\d$-valued}\/ for short). If $K$ is $\d$-valued, then the residue morphism 
$\mathcal{O}\to \res(K)$ 
maps the constant field $C$ isomorphically onto~$\res(K)$.
Not every asymptotic field extends to a $\d$-valued field; those that do are called {\it pre-$\d$-valued.}\/ By \cite[Theorem~4.4]{AvdD}, 
\begin{equation}\label{eq:pdv}
\text{$K$ is pre-$\d$-valued} \quad\Longleftrightarrow\quad \text{$f'\prec g^\dagger$ whenever $f\preceq 1\prec g$ in $K$.}
\end{equation}
The   {\em asymptotic couple}\/ of $K$ is the pair  $(\Gamma,\psi)$  
where~$\psi\colon\Gamma^{\neq}\to\Gamma$ is given by 
$$\psi(vg)\ =\ v(g^\dagger)\ \text{ for $g\in K^\times$ with $vg\ne 0$}.$$ 
We put $\Psi:=\psi(\Gamma^{\neq})=\big\{v(g^\dagger): g\in K^\times\setminus\mathcal O^\times\big\}$. 
If we want to stress the dependence on $K$, we write~$(\Gamma_K,\psi_K)$ and $\Psi_K$ instead of $(\Gamma,\psi)$ and $\Psi$, respectively.
If $L$ is an asymptotic field extension of $K$, then $(\Gamma_L,\psi_L)$ extends the asymptotic couple~$(\Gamma,\psi)$,
with $\Psi_L=\Psi$ if  $L\supseteq K$ is algebraic.
Note: $K$ is monotone iff $\Psi\subseteq\Gamma^{\ge}$.

\begin{example} Let $C$ be a field of characteristic zero and $x$   an indeterminate. Then ${K=C(x)}$ with derivation $\der=\frac{d}{dx}$ and valuation $v\colon K^\times \to \Z$ given by $v(C^\times)=\{0\}$ and~$v(x)=-1$
is $\d$-valued  with constant field $C$,
and $\Psi=\{1\}$, so $K$ is monotone.
(Asymptotic fields we consider later are typically not monotone.)
\end{example}

\noindent
An {\it asymptotic couple}\/ (without mentioning any asymptotic field) is a pair $(\Gamma,\psi)$ consisting 
of an ordered abelian group $\Gamma$ and a map $\psi\colon\Gamma^{\neq}\to\Gamma$  subject to some axioms
obeyed by the asymptotic couples of asymptotic fields; see~[ADH, 6.5]. 
Let~$(\Gamma,\psi)$ be an asymptotic couple.
We  
extend~$\psi\colon\Gamma^{\neq}\to\Gamma$ to a map~${\Gamma_\infty\to\Gamma_\infty}$
by~$\psi(0):=\psi(\infty):=\infty$.   
If $(\Gamma,\psi)$ is understood from the context and $\gamma\in\Gamma$ we write~$\gamma^\dagger$ and $\gamma'$ instead of $\psi(\gamma)$ and~${\gamma+\psi(\gamma)}$, respectively. We also  set $\Psi:=\psi(\Gamma^{\neq})$; then
$\Psi<(\Gamma^>)'$.
The following  ``contractive'' property of $\psi$ is from~[ADH, 6.5.4(i)]:

\begin{lemma}\label{lem:6.5.4}
For all  $\alpha,\beta< (\Gamma^>)'$ in $\Gamma$ we have
$\psi(\alpha-\beta)>\min\{ \alpha,\beta\}$. %\textup{(}In particular, if~$\gamma,\delta\in\Gamma^{\ne}$, then $\psi\big(\psi(\gamma)-\psi(\delta)\big)>\min\big\{\psi(\gamma),\psi(\delta)\big\}$.\textup{)}
\end{lemma}

\noindent
Call $\gamma\in\Gamma$  a {\it gap}\/ in $(\Gamma,\psi)$
if~$\Psi<\gamma<(\Gamma^>)'$. (There is at most one such $\gamma$.  For this and other facts about gaps, see [ADH, p. 388].)
We  say that~$(\Gamma,\psi)$ is {\it grounded}\/ if~$\Psi$ has a largest element, 
and $(\Gamma,\psi)$ has {\it asymptotic integration}\/ if~$(\Gamma^{\neq})'=\Gamma$.
%The set $\Gamma\setminus (\Gamma^{\neq})'$ has at most one element; if $(\Gamma,\psi)$  is grounded, then $\Gamma\setminus (\Gamma^{\neq})'=\{\max\Psi\}$ [ADH, 9.2.1].
%ee [ADH, 9.1, 9.2] for more on this, in particular for  the following important trichotomy: 
An asymptotic field is said to have a gap if its asymptotic couple does; likewise with ``grounded'' or ``asymptotic integration'' in place of ``has a gap''.  Any $\d$-valued field closed under integration  is closed under asymptotic integration.
A gap in an asymptotic field $K$ remains a gap in an algebraic asymptotic field extension of $K$.

An {\it $H$-asymptotic couple}\/ is 
an asymptotic couple $(\Gamma,\psi)$ such that for all~${\gamma,\delta\in\Gamma^>}$ we have:
$\gamma\leq\delta\Rightarrow \gamma^\dagger\geq\delta^\dagger$. 
%Then $(\Gamma^<)'$ is closed downward and $(\Gamma^>)'$ is closed upward.
Every $H$-asymptotic couple  either
 has a gap, or  is grounded, or has asymptotic integration   [ADH, 9.2.16]. 
An asymptotic field is {\it $H$-asymptotic}\/ (or {\it of $H$-type}\/), if its asymptotic couple is $H$-asymptotic. % (or an asymptotic field of {\it $H$-type}\/).

\subsection*{Dominant part}
Let $K$ be a valued differential field with small derivation, and set~$\k:=\res(K)$.
We extend the residue morphism $\mathcal O\to\k$
to the differential ring morphism~$P\mapsto\bar{P}\colon\mathcal O\{Y\}\to\k\{Y\}$  with $Y\mapsto Y$.
 Let $P$, $Q$ range over $K\{Y\}^{\ne}$.
We extend the valuation $v\colon K^\times \to \Gamma$,  first to $K\{Y\}^{\ne}$ by
%to the so-called {\em gaussian valuation}\/~$v\colon K\<Y\>\to \Gamma_{\infty}$ on the fraction field $K\<Y\>$ of $K\{Y\}$ by requiring that,  in the notation used above,
$$v(P)\ :=\  \min \{ va: \text{$a\in K^\times$ is a nonzero coefficient of $P$},$$
and then (uniquely) to a valuation $v\colon K\<Y\>^\times \to  \Gamma$. 
For each $P$ we
  pick~$\mathfrak d_P\in K^\times$ such that $P\asymp \mathfrak d_P$, with $\mathfrak d_P=\mathfrak d_Q$ if
 $P\sim Q$. Then 
$\mathfrak d_P^{-1}P\in\mathcal O\{Y\}$,   
 and we define the 
{\em dominant part}\/  of~$P$ by~$D_P:=\overline{\mathfrak d_P^{-1}P}\in\k\{Y\}^{\ne}$ and the
{\em dominant degree\/} of $P$ by
$\ddeg P:=\deg D_P$. 
(Another choice of $\mathfrak d_P$ multiplies $D_P$  by an element of $\k^\times$, so $\ddeg P$ does not depend on this choice.)

\subsection*{The map $v_A$}
With $K$   as in the previous subsection, let~$A=\sum_i a_i\der^i\in  K[\der]^{\ne}$ (all $a_i\in K$). In~[ADH, 5.6] we   defined
$v(A) := \min_i v(a_i) \in \Gamma$. 
For $y\in K^\times$, the quantity  $v(Ay)$ only depends on~$vy$, not on~$y$, and we can thus define $v_A(\gamma):=v(Ay)$ for $\gamma:=vy$. 
The map $v_A\colon\Gamma\to\Gamma$ is strictly increasing by [ADH, 4.5.1]. Much less obvious is that
$v_A$ is   surjective; see [ADH, 6.0.1]. We give a proof for the easy case that is used later:

\begin{lemma} \label{lem:vA onto}
Let $A\in K[\der]^{\ne}$ have order $1$ and~$b\in K^\times$. Then there is a $y\in K^\times$ such that $v(A y)=vb$.
\end{lemma}
\begin{proof}
We may assume $A=\der-s$ where $s\in K$.
Then  for $y\in K^\times$ we have
$A y =   y \big( \der-(s-y^\dagger) \big)$. Put $a:=s-b^\dagger$.
If $a\preceq 1$,  then $y:=b$ clearly works, so suppose $a\succ 1$. 
We claim that then $y:=b/a$ works:
from  $a^\dagger/a = -(1/a)^\dagger\cdot (1/a)=-(1/a)'\prec 1$ (since $K$ has small derivation)
we obtain $a^\dagger \prec a$ 
and so
$s-y^\dagger = a + a^\dagger \sim a$.
\end{proof}

\subsection*{Newtonianity}
Let $K$ be an ungrounded  $H$-asymptotic field with 
%asymptotic couple $(\Gamma,\psi)$, where 
${\Gamma\ne \{0\}}$. 
 An element $\phi$ of~$K$ is said to be {\em active in $K$\/} if $v\phi\in\Psi^\downarrow$; in that case the derivation~$\phi^{-1}\der$ of the compositional conjugate~$K^\phi$ is small. %, cf.~[ADH, 11.1]. 
{\em In this subsection $\phi$ ranges over the active elements of $K$}.
A property $S(\phi)$ of (active) elements~$\phi$ is said to hold {\em eventually}\/  if there is an active $\phi_0$ in $K$ such that~$S(\phi)$ holds for all~$\phi\preceq \phi_0$; cf.~[ADH, p.~479].  
Given a differential polynomial $P\in K\{Y\}^{\ne}$, its compositional conjugate $P^\phi$ tends to become more transparent eventually, in particular,
  $\ddeg P^\phi$ is eventually constant, as shown    in~[ADH, 11.1]. The eventual value of $\ddeg P^\phi$ is
  called the   {\it Newton degree}\/ of $P$,    denoted by  $\ndeg P$.
For $r\in\N$ we  say that $K$ is   {\em $r$-newtonian}\/ if  every~$P\in K\{Y\}$ of order $\le r$ with $\ndeg P=1$ has a zero in $\mathcal{O}$,
and   {\em $r$-linearly newtonian}\/ if  every~$P\in K\{Y\}$ of order $\le r$ with $\deg P=\ndeg P=1$ has a zero in $\mathcal{O}$.
We call~$K$   {\em newtonian}\/ if it is $r$-newtonian for all $r\in\N$, and likewise we define ``linearly newtonian''.
By [ADH, 14.2.5], if $K$ is $1$-linearly newtonian and has asymptotic integration, then $K$ is $\d$-valued.
By [ADH, 14.2.2], if $K$ is $r$-linearly newtonian,  then~$K$ is $r$-linearly surjective. We include a proof for~$r=1$:

\begin{lemma}\label{lem:newt=>linsurj}
Suppose $K$ is $1$-linearly newtonian. Then $K$ is $1$-linearly surjective \textup{(}and thus closed under integration\textup{)}. 
\end{lemma}
\begin{proof} 
Let $A\in K[\der]^{\ne}$ have order~$1$ and $b\in K^\times$; to show: $b\in A(K)$.
Replace~$A$,~$K$ by $A^\phi$, $K^\phi$ for some $\phi$ with $v\phi\in \Psi$ to arrange $0\in\Psi$ (so~$K$ has small derivation).
Lemma~\ref{lem:vA onto} yields $g\in K^\times$ with~${v(Ag)=b}$. Replace~$A$,~$b$ by $b^{-1}Ag$, $1$, to arrange~$v(A)=0 \le vb$. Then~$A=a_1\der+a_0$ where~$a_0,a_1\in\mathcal O$ with $a_0\asymp 1$ or~$a_1\asymp 1$.
Consider first the case~$a_0\asymp 1$, and set $P:=a_1Y'+a_0Y-b\in K\{Y\}$.
For $\phi\prec 1$ in~$K$
we have~$P^\phi = a_1\phi Y'+a_0Y-b \sim a_0Y-b$. Hence~$\ndeg P=1$, so $P$   has a zero~$y\in K$, 
that is,~$A(y)=b$. In the remaining case we have $a_0\prec 1$ and $a_1\asymp 1$. Now $0\in \Psi$ gives 
 $e\succ 1$ in~$K$ with~$e^\dagger\asymp 1$. 
 %(think of $e$ as $\ex^x$).
 Then  $A_{\ltimes e}=a_1\der+(a_0+a_1e^\dagger)$, with~$a_0+a_1e^\dagger\asymp 1$, which gives a reduction to the previous case.
% replace~$A$,~$b$ by~$A_{\ltimes e}$,~$b/e$. 
\end{proof}

\noindent
Newtonianity for $H$-asymptotic fields can be viewed as the the appropriate analogue  of henselianity for valued fields. In fact, $K$ is henselian iff it is $0$-newtonian. More relevant for us
is the following [ADH, 14.2.12], where~$\k:=\res(K)$.

\begin{lemma}\label{lem:14.2.12}
Suppose $K$ has asymptotic integration and is $r$-newtonian, $r\ge 1$,  and let $P\in K\{Y\}$,
$u\in\mathcal O$, and~$A\in\k[Y]$ be such that $\order P\le r$, $A(\bar{u})=0$, $\left(\frac{\partial A}{\partial Y}\right)(\bar{u})\ne 0$,
and $D_{P^\phi}\in\k^\times  A$, eventually. Then $P$ has a zero in $u+\smallo$.
\end{lemma}

\noindent
We use this lemma to isolate an argument in the proof of [ADH, 14.2.18]. Let
$$f\in K, \qquad S\ :=\ \sigma(Y)-f\ =\ \omega(-Y^\dagger)+Y^2-f \in K\<Y\>.$$
Then for $y\in K^\times$ we have $\sigma(y)=f$ iff $S(y)=0$.
Let $b\in K^\times$, and set $Q:=Y^2 S(bY)$. The computation preceding [ADH, 14.2.18]
shows that $Q\in K\{Y\}$ and $\order Q=2$; in fact
\begin{align*}
Q &\ =\  b^2Y^4 + \big(\omega(-b^\dagger)-f\big) Y^2 + R\quad\text{ where $R\in K\{Y\}$,} \\
\text{with}\quad R^\phi &\ =\ \phi\cdot \big(
2\phi Y''Y +  2(\phi/b)^\dagger Y'Y - 3\phi (Y')^2 \big).
\end{align*}
We shall also use the following estimate from [ADH, 11.7.6]:
\begin{equation}\label{eq:11.7.6}
\omega(-\phi^\dagger) - \omega(-\theta^\dagger) \prec \phi^2\qquad \text{for active $\theta$ in $K$ with $\phi\succ \theta$.}
\end{equation}
This leads to a result used in the proof of Lemma~\ref{lem:asymptotics of phi'}:

\begin{lemma}\label{lem:ADH 14.2.18}
Suppose $K$ has asymptotic integration and is $2$-newtonian, and let~$b\in K^\times$, $f\in K$, and $\upg$ be active in~$K$ with $b^2=f-\sigma(\upg)\succ\upg^2$.
Then there is a~$y\in K^\times$ with $\sigma(y)=f$ and~$y\sim b$.\end{lemma}
\begin{proof}
Since $b\succ\upg$, $b$ is active in $K$.
Using \eqref{eq:11.7.6} for the last step below, we have
$$\sigma(b)-f\ =\ \sigma(b)-\sigma(\upg)-b^2\  =\  \omega(-b^\dagger) - \omega(-\upg^\dagger) - \upg^2\ \prec\  b^2,$$
and so $\omega(-b^\dagger)-f\sim-b^2$. Eventually $\phi\prec b$, so $(\phi/b)^\dagger\prec b$ by Lemma~\ref{lem:6.5.4}.
Hence with $R$, $Q$ as   above, eventually   $R^\phi\prec b^2$, and thus~$Q^\phi \sim b^2 Y^2 (Y^2-1)$.   Lemma~\ref{lem:14.2.12} yields $u\in K$ with $u\sim 1$ and~$Q(u)=0$, thus~$\sigma(y)=f$ for $y:=bu\sim b$.
\end{proof}

\subsection*{$\upl$-freeness and $\upo$-freeness}
{\it In   this subsection $K$ is  an $H$-asymptotic field.}\/
In [ADH, 11.6] we   defined~$K$ to be {\it $\upl$-free}\/ if it is ungrounded, $\Gamma\ne\{0\}$,   and
for all~$\upl\in K$ there is a  ${g\succ 1}$ in~$K$ with $\upl+g^{\dagger\dagger}  \succeq g^\dagger$. For pre-$\d$-valued $K$ the condition   ``$K$ is ungrounded and~$\Gamma\ne\{0\}$''     is actually superfluous:

\begin{lemma}\label{lem:upl-free}
Suppose $K$ is pre-$\d$-valued. If for all   $\upl\in K$ there is a  ${g\succ 1}$ in~$K$ with $\upl+g^{\dagger\dagger}  \succeq g^\dagger$,
then $K$ has asymptotic integration.
\end{lemma}
\begin{proof}
Suppose $K$ does not have asymptotic integration. Then $K$ is grounded or has a gap.
Either case yields $\upg\in K^\times$ with $\Psi \le v\upg<(\Gamma^>)'$.  Then for each~${\gamma\in \Gamma^{<}}$
we have~$\psi\big({\psi(\gamma)-v\upg}\big)>\psi(\gamma)$, by Lemma~\ref{lem:6.5.4}. 
 Putting $\upl:=-\upg^\dagger$,  we obtain for all~${g\in K}$ with~$g\succ 1$ that~$\upl+g^{\dagger\dagger}=(g^\dagger/\upg)^\dagger \prec g^\dagger$, using \eqref{eq:pdv} when~$g^\dagger\asymp\upg$.    
 \end{proof}

\noindent
The property of $\upl$-freeness helps to preserve asymptotic integration under certain $\d$-algebraic $H$-asymptotic field extensions:
by \cite[Proposition~1.3.12]{ADH4}, $K$ is $\upl$-free iff every $H$-asymptotic field extension $L$ of $K$ with~${\operatorname{trdeg}(L|K)\le 1}$ has asymptotic integration. 
%By Lemma~\ref{lem:newt=>linsurj} and the next lemma, if $K$ is pre-$\d$-valued and $1$-linearly newtonian, then it is $\upl$-free. 
Next a variant of~[ADH, 14.2.3]:

\begin{lemma}\label{lem:1-ls => upl-free}
Suppose  $K$ is $\d$-valued and $1$-linearly surjective. Then $K$ is $\upl$-free.  
\end{lemma}
\begin{proof} The $H$-asymptotic field $K$ has asymptotic integration, as a consequence of~$K$ being $\d$-valued and $\der K=K$. It remains to apply [ADH, 11.6.17]. 
\end{proof} 
%First, $K$ is ungrounded: otherwise we have $g\succ 1$ in $K$ with $vg^\dagger=\max \Psi$, and then $f\in K$ with $f'=g^\dagger$ gives $f\preceq  1$, contradicting that $K$ is pre-$\d$-valued.
%The rest of the proof is also by contradiction: Assume $K$ is not $\upl$-free. Then we have $\upl\in K$ such that $\upl+g^{\dagger\dagger} \prec g^\dagger$ for all~$g\succ 1$ in $K$.
%Take $y\in K$ with $y'-\upl y=1$, and put $a:=-1/y$, so~${\upl+a^\dagger}=\upl-y^\dagger=a$.
%Suppose $g\succ 1$ in $K$ and $a\succeq g^\dagger$.
%Then~$\upl+g^{\dagger\dagger}\prec g^\dagger\preceq a=\upl+a^\dagger$ and so $(a/g^\dagger)^\dagger=a^\dagger-g^{\dagger\dagger}\sim \upl+a^\dagger=a$.
%If $a\succ g^\dagger$, then $(a/g^\dagger)^\dagger \prec a$ by Lemma~\ref{lem:6.5.4}, and if
%$a\asymp g^\dagger$, then $(a/g^\dagger)^\dagger \prec g^\dagger$ by \eqref{eq:pdv}.
%In both cases~$(a/g^\dagger)^\dagger \prec a$, contradicting $(a/g^\dagger)^\dagger \sim a$. 

%Thus $a \prec g^\dagger$ for all~$g\succ 1$ in~$K$, that is, $va>\Psi$. 
%Since $K$ is closed under integration, we have~$h\in K^\times$ with~$a = (1/h)'$.
%Put $b:=h^\dagger$. Then~$b/a = -h\succ (?) 1$ and thus~$\upl+h^{\dagger\dagger}=\upl+a^\dagger+(b/a)^\dagger=a+(b/a)^\dagger \asymp (b/a)^\dagger  = h^\dagger$, again a contradiction.
%\end{proof}

\noindent 
Theorem~\ref{thm:weakly d-closed} below exemplifies that 
newtonianity works best when combined with $\upo$-freeness, a stronger property than $\upl$-freeness.
To introduce it,
recall that in Section~\ref{sec:diffalg}   we defined~${\omega(z)=-(2z'+z^2)}$ for~$z\in K$.
We say that~$K$ is {\it $\upo$-free}\/ if  
 it is ungrounded with~$\Gamma\ne\{0\}$, and for all~$\upo\in K$ there is a $g\succ 1$ in $K$  with~${\upo-\omega(-g^{\dagger\dagger}) \prec (g^\dagger)^2}$; cf.~[ADH, 11.7].
If~$K$ is $\upo$-free, then it is $\upl$-free, by~[ADH, 11.7.3, 11.7.7].
An important use of $\upo$-freeness is to
 prevent  deviant behavior among   $\d$-algebraic pre-$\d$-valued $H$-asymptotic extensions:

 \begin{theorem}[{[ADH, 13.6.1]}]\label{thm:13.6.1}
Suppose $K$ is $\upo$-free and $L$ is a  $\d$-algebraic pre-$\d$-valued $H$-asymptotic field extension
 of~$K$. Then $L$ is also $\upo$-free $($hence has asymptotic integration$)$ and $\Gamma_K^{<}$ is cofinal in $\Gamma_L^{<}$.
 \end{theorem}
 
 \noindent 
 In the context of Hardy fields, $\upo$-freeness is related to oscillation criteria for
  linear differential equations of order~$2$. (See the discussion following Lemma~\ref{lem:7.11}  below.)  
An important source of $\upo$-free $H$-asymptotic fields is the following:

\begin{prop}[{[ADH, 11.7.15]}]\label{prop:11.7.15} If $K$ is ungrounded and a union of ground\-ed $H$-asymptotic subfields, then $K$ is $\upo$-free.
\end{prop}

\begin{remark}[for readers familiar with model theory; not used later]  
Proposition~\ref{prop:11.7.15} has a sort of converse: {\it Suppose $K$ is $\upo$-free. Then $K$ has an elementary extension which is a union of grounded $H$-asymptotic subfields.}\/
To see this, fix an $\abs{K}^+$-saturated elementary extension $K^*$ of $K$. Take $y\in K^*$ with $\Gamma^< < vy <0$. Then the $H$-asymptotic subfield $K\langle y\rangle$ of $K^*$ is grounded, by [ADH, 13.6.7].
The downward L\"owenheim-Skolem Theorem [ADH, B.5.10] yields an elementary substructure $K_1$ of $K^*$ with $K\<y\>\subseteq K_1$
and $\abs{K_1}=\abs{K\langle y\rangle}=\abs{K}$. Repeating this process with $K_1$ in place of $K$,
we construct an elementary chain $(K_n)$ of $H$-asymptotic subfields of $K^*$ with $K_0=K$
such that each $K_n$ is contained in a grounded $H$-asymptotic subfield of $K_{n+1}$. Then $\bigcup_n K_n$ is the
desired elementary extension of~$K$, by [ADH, B.5.13].
\end{remark}
 
\noindent
The newtonian property only postulates the existence of zeros for differential polynomials of Newton degree~$1$.
Under   additional hypotheses on $K$ we get more:

\begin{theorem}[{[ADH, 3.1.17, 14.2.5, 14.5.3]}]\label{thm:weakly d-closed}
If $K$ is  algebraically closed, $\upo$-free, and newtonian, then each $P\in K\{Y\}\setminus K$ has a zero in $K$.
\end{theorem}

\noindent
This yields a result which is at the basis of the main theorems of this paper:  

\begin{cor}\label{cor:weakly d-closed}
Suppose   $K$ is  algebraically closed, $\upo$-free, and newtonian. Then~$K$ is linearly closed and linearly surjective.
\end{cor}
\begin{proof}
By induction on $r\in\N$ we show that $K$ is $r$-linearly closed. The case $r=0$ being trivial, suppose  
  $A\in K[\der]^{\ne}$ has order $r\ge 1$. Then $R:=\Ric(A)\in K\{Z\}\setminus K$ has a zero $z\in K$, by Theorem~\ref{thm:weakly d-closed}.
Take $y\ne 0$ in a differential field extension of~$K$ with $y^\dagger=z$. Then $A(y)=0$ and hence
$A=B(\der-z)$ where $B\in K[\der]^{\ne}$ has order~${r-1}$, by \eqref{eq:5.1.21}.
By inductive hypothesis, $B$ splits over $K$, and hence so does~$A$. Thus $K$ is linearly closed,
and since  $K$ is $1$-linearly surjective by Lemma~\ref{lem:newt=>linsurj},  
 $K$ is linearly surjective.
\end{proof}

%\noindent
%We shall also use:

\begin{prop}[{[ADH, 11.7.23, 14.2.5, 14.5.7]}]\label{prop:newt alg ext}
If $K$ is $\upo$-free, newtonian, and has divisible value group,
then each  $H$-asymptotic field extension of~$K$ which is algebraic over $K$  is $\upo$-free and newtonian.
\end{prop}

\noindent
To use this, let $L$ be a field extension of $K$ that is algebraic over $K$. Then $L$ with the unique derivation on $L$ extending that of $K$ and any valuation ring of $L$ lying over $\mathcal{O}$ is $H$-asymptotic, by [ADH, 9.5]. As a special case, suppose 
$-1$ is not a square in $\res(K)$ and $L=K[\imag]$ with $\imag^2=-1$. Then we construe $L$ accordingly as an $H$-asymptotic field with respect to the unique valuation ring (namely $ \mathcal{O}+\mathcal{O}\imag$) of $L$ that lies over $\mathcal{O}$. 
%If in this situation
%$K$ is real closed, $\upo$-free, and newtonian, then $L$ is algebraically closed, $\upo$-free, and newtonian, by Proposition~\ref{prop:newt alg ext}.  

\subsection*{$H$-fields}   
  An  {\em ordered differential field\/} is a differential field $H$ with an ordering on~$H$ making $H$ an ordered field. Likewise, an  {\em ordered valued differential field\/} is a  valued differential field $H$ equipped with an ordering on $H$ making~$H$ an ordered field (no relation between derivation, valuation, or ordering
being assumed).  Let~$H$ be an ordered differential field and $C=C_H$. Then $H$ has the convex subring 
$$\mathcal{O}\ :=\ \big\{h\in H:\text{$\abs{h} \le c$ for some $c\in C^{>}$}\big\},$$ 
which is a valuation ring of $H$ and has maximal ideal
$$\smallo\ =\ \big\{h\in H:\ \text{$\abs{h} \le c$ for all   $c\in C^>$}\big\}.$$ We call $H$ an {\it $H$-field} 
 if for all $h\in H$ with $h\succ 1$ we have $h^\dagger>0$, and
  $\mathcal{O}=C+\smallo$.

  {\em  In the rest of this subsection $H$ is an $H$-field, construed as an ordered valued differential field with valuation ring $\mathcal{O}$}. 
%{\it Pre-$H$-fields}\/\index{pre-H-field@pre-$H$-field} are the ordered valued differential subfields of $H$-fields. 
Then $H$ is  $\d$-valued of $H$-type and $H^\phi$ is an $H$-field for $\phi\in H^{>}$; 
%and $\res(H)$ is made an ordered field by requiring for all $a,b\in \mathcal{O}$ that $a\le b\Rightarrow \bar{a}\le \bar{b}$;
 see [ADH, 10.5] for this and other basic facts about $H$-fields. 
Call $H$  {\it Liouville closed}\/
 if it is real closed,  closed under integration, and~$H^\dagger=H$. Liouville closed $H$-fields are  $\upl$-free, by Lemma~\ref{lem:1-ls => upl-free}.
  A {\em Liouville closure of $H$} is a Liouville closed $H$-field extension of $H$ which as a differential field is a Liouville extension of~$H$~[ADH, 10.6].  By \cite[Theorem~12.1(1)]{Gehret}, the following are equivalent:

\begin{enumerate}
\item $H$ has a unique Liouville closure up to isomorphism  over~$H$;
\item $H$ is grounded or $\upl$-free.
\end{enumerate}
A key ingredient for the proof of (1)~$\Rightarrow$~(2) is that if   $v\upg$ \textup{(}$\upg\in H^\times$\textup{)} is
a gap in~$H$, then 
 there is $y$ in an $H$-field extension of $H$ with $y\prec 1$ and~$y'=\upg$, by [ADH, 10.2.1, 10.5.10], and there is also $z$ in an $H$-field extension of $H$ with~$z\succ 1$ and~$z'=\upg$, by~[ADH, 10.2.2, 10.5.11].  Here is another ingredient:
 %  way of obtaining gaps  from non-$\upl$-freeness; see~[ADH, 10.4.5(i), remark after 11.5.14, and 11.6.1].

\begin{lemma}\label{lem:11.5.14}
If $H$ is real closed with asymptotic integration and $\upl\in H$ satisfies~$\upl+g^{\dagger\dagger} \prec g^\dagger$ for all $g\succ 1$
in~$H$, then for any $\upg\neq 0$ in any $H$-field extension of~$H$  with
$\upg^\dagger=-\upl$, the $H$-asymptotic field $H(\upg)$ has gap $v\upg$, and~$H$ is cofinal in~$H(\upg)$.
\end{lemma}
\begin{proof} Let $H$, $\upl$ be as in the hypothesis. Then the first part of the conclusion follows from [ADH, 11.5.6, (iv)~$\Rightarrow$~(i), and 11.5.14 with the remark after its proof]. 
The cofinality then follows from [ADH,   11.5.13 and   10.4.5(i)].
\end{proof} 

\noindent
The proof of (2)~$\Rightarrow$~(1) rests on the following:

\begin{prop}[{\cite[Proposition~1.3.15]{ADH4}}]\label{prop:Gehret} 
If $H$  is a $\upl$-free and $E$ is a Liou\-ville $H$-field extension of $H$, 
then $E$ is $\upl$-free and $\Gamma_H^<$ is cofinal in~$\Gamma_E^<$. 
\end{prop}

\noindent
By [ADH, 16.4.1, 16.4.8 and its proof],  $H$ has an  $\upo$-free newtonian Liouville closed $H$-field extension whose constant field is a real closure of~$C$. %\marginpar{still  to check the real closure claim} 
Theorem~\ref{thm:maxhardymainthm} below is a Hardy field version of this fact, and is used in proving the results stated in the introduction.  
The next lemma is extracted from the proof of \cite[Lemma~1.3.18]{ADH4}:

\begin{lemma}\label{lem:1.3.18}
Suppose $H$ has no asymptotic integration, $\Gamma_H\ne \{0\}$, and $E$ is an  $H$-field extension of $H$ closed under integration.
Then $H$ is cofinal in some $\upo$-free $H$-subfield of $E$ that contains $H$ and has constant field $C$.
\end{lemma}
\begin{proof} %\marginpar{to check cofinality in two [[..]] passages of the proof} 
If $H$ is grounded, then the  extension $H_{\upo}$ of $H$ from~[ADH, 11.7.17 and surrounding remarks, with $H$ instead of $F$] is an $\upo$-free $H$-field and   
embeds into $E$ over $H$, and $H$ is cofinal in~$H_{\upo}$ with $C_{H_{\upo}}=C$, by the construction of $H_{\upo}$.
Next, suppose~$H$ has a gap~$vb$, $b\in H^\times$. Take~$a\in  E$
with~$a'=b$ and~$a\nasymp 1$. Then~$H(a)$ is a grounded $H$-subfield of $E$ with  $C_{H(a)}=C$, by~[ADH,  10.2.1, 10.2.2 and remarks following their proofs], and $H$ is cofinal in $H(a)$ by~[ADH, 9.8.2 and subsequent remarks].
 Now apply the previous case to~$H(a)$ in place of~$H$.
\end{proof}

\subsection*{The functions $\omega$ and $\sigma$ in $H$-fields}
Let $H$ be an $H$-field. In [ADH, 11.8] we introduced some special subsets of $H$:   
\begin{align*}
\Upg(H)	&\ :=\  \big\{ h^{\dagger}:\ h\in H,\, h\succ 1\big\}\ \subseteq\  H^{>}, \\
\Upl(H)	&\ :=\  \big\{ {-h^{\dagger\dagger}}:\ h\in H,\, h\succ 1\big\}, \\
\Upd(H)	&\ :=\  \big\{ {-h^{\prime\dagger}}:\  h\in H,\, 0\neq h\prec 1\big\}.
%		&\ \ = \  \big\{ {-h^{\dagger\dagger}}+h^\dagger:\ h\in H,\, h\succ 1\big\}.
\end{align*}
Here $\Upl(H)= - \Upg(H)^\dagger$ and $\Upd(H)$ are disjoint; in fact, $\Upl(H)<\Upd(H)$. 
If $H$ has asymptotic integration, then by~[ADH, 11.6.1, 11.8.16],
\begin{equation}\label{eq:uplfree}
\text{$H$ is $\upl$-free}\quad \Longleftrightarrow\quad  
H\ =\ \Upl(H)^{\downarrow}\cup \Upd(H)^\uparrow.
\end{equation}
If $H$ is Liouville closed, then by [ADH,  11.8.13], $\Upl(H)$ is downward closed,   $\Upd(H)$ is upward closed, 
$H=\Upl(H)\cup\Upd(H)$, and by [ADH, 11.8.19],
\begin{equation}\label{eq:11.8.19}
\Upg(H)\ =\ \{h':\ h\in H^>,\ h\succ 1\}, \quad \Upg(H)\  \text{ is upward closed}.
\end{equation}
By [ADH, 11.8.20, remarks before 11.8.21 and 11.8.30, 11.8.29, 11.8.31]:

\begin{lemma}\label{lem:11.8.29}
The functions $\omega\colon H\to H$ and $\sigma\colon H^\times\to H$ are strictly increasing on $\Upl(H)^\downarrow$ and $\Upg(H)^\uparrow$, respectively,
with $\omega(H)<\sigma\big(\Upg(H)\big)$.
If $H$ is Liouville closed, then $\omega(H)=\omega\big(\Upl(H)\big)=\omega\big(\Upd(H)\big)$ and
$\sigma\big(H^>\setminus\Upg(H)\big)\subseteq\omega(H)^\downarrow$.
\end{lemma}

\noindent
 If $H$ has asymptotic integration, then by [ADH, 11.7.7, 11.8.30]:
\begin{equation}\label{eq:upofree}
\text{$H$ is $\upo$-free}\quad \Longleftrightarrow\quad
H\ =\ \omega\big(\Upl(H)\big){}^{\downarrow}\cup \sigma\big(\Upg(H)\big){}^\uparrow,
\end{equation}
in analogy to \eqref{eq:uplfree}. We use this to obtain a variant of [ADH, 11.8.33]:

\begin{cor}\label{cor:Schwarz closed}
Suppose $H$ is Liouville closed. Then the following are equivalent:
\begin{enumerate}
\item[\textup{(i)}] $H=\omega\big(\Upl(H)\big)\cup \sigma\big(\Upg(H)\big)$;
\item[\textup{(ii)}] $H$ is $\upo$-free, $\omega\big(\Upl(H)\big)$ is downward closed, and $\sigma\big(\Upg(H)\big)$ is  upward closed;
\item[\textup{(iii)}] $H=\omega(H)\cup\sigma(H^\times)$, and $\omega(H)$ is downward closed;
\item[\textup{(iv)}] every $A\in H[\der]$ of order $2$ splits over $H[\imag]$, and $\omega(H)$ is downward closed; here $H[\imag]$ with $\imag^2=-1$ is an algebraic closure of $H$.
\end{enumerate}
\end{cor}
\begin{proof}  Lemma~\ref{lem:11.8.29} gives
$\omega(\Upl(H))=\omega(H) < \sigma\big(\Upg(H)\big)\subseteq\sigma(H^\times)$, from which we obtain
(i)~$\Rightarrow$~(iii). 
That lemma also yields
$$\sigma(H^\times)\ =\ \sigma(H^{>})\ =\ \sigma\big(H^{>}\setminus \Upg(H)\big)\cup \sigma\big(\Upg(H)\big)\ \subseteq\ \omega(H)^{\downarrow}\cup \sigma(\Upg(H)),$$
thus (iii)~$\Rightarrow$~(ii) by \eqref{eq:upofree}.  
%(iii) holds, then
%$$H\setminus \omega\big(\Upl(H)\big){}^{\downarrow}  = H\setminus\omega(H)=  \sigma(H^\times)\setminus\omega(H) =  \sigma\big( \Upg(H) \big)$$
%by Lemma~\ref{lem:11.8.29}, and this yields (ii), using \eqref{eq:upofree}. 
The latter also gives (ii)~$\Rightarrow$~(i).
The equivalence~(iii)~$\Leftrightarrow$~(iv) follows from \eqref{eq:sigma} 
and remarks on linear differential operators in Section~\ref{sec:diffalg}.
\end{proof}

\noindent
We say that $H$ is {\it Schwarz closed}\/ if it is Liouville closed and  satisfies
the equivalent conditions in Corollary~\ref{cor:Schwarz closed}.
From [ADH, 14.2.16, 14.2.18] we obtain:
 
\begin{prop}\label{prop:14.2.16+18}
Suppose $H$ is $2$-newtonian and real closed, and has asymptotic integration. Then $\omega\big(\Upl(H)^\downarrow\big)$ is downward closed  and
$\sigma\big(\Upg(H)^\uparrow\big)$ is upward closed.
\end{prop}

\noindent
Thus every $2$-newtonian $\upo$-free
Liouville closed $H$-field is Schwarz closed.

\section{Functions and Germs}\label{sec:germs}
 
\noindent
After introducing notation to deal with germs and asymptotic relations we
 review the basic facts about second-order linear differential equations, and discuss the composition of germs. Throughout, $a,b$ ranges over $\R$, and $r$ over~${\N\cup\{\infty, \omega\}}$.
 
 \subsection*{Differential rings of germs}
Let 
 $\c_a^r$ be the  $\R$-algebra  of functions $[a,+\infty)\to\R$ which,
 for some open $U\supseteq [a,+\infty)$, extend to a $\c^r$-function~${U\to\R}$; here~$\c^\omega$  means ``analytic''. Hence
 $$ \c_a\ :=\ \c_a^0 \ \supseteq \c_a^1 \ \supseteq\ \c_a^2  \ \supseteq\  \cdots  \ \supseteq\ \c_a^\infty  \ \supseteq\ \c_a^\omega.$$
Note that  $\c_a^r[\imag]=\c_a^r+\c_a^r\imag$  is a subalgebra of the $\C$-algebra~$\c_a[\imag]$
of continuous functions $[a,+\infty)\to\C$. For $f=g+h\imag\in \c_a[\imag]$, $g,h\in\c_a$,  
$$\Re f\ :=\ g,\quad \Im f\ :=\ h,\quad
\bar{f}\ :=\ g - h\imag \ \in\  \c_a[\imag],\qquad |f|\ :=\ \sqrt{g^2+h^2}\ \in\ \c_a.$$  
Let $\c^r$ be the $\R$-algebra of germs at $+\infty$ of functions in $\bigcup_a\c_a^r$. Thus~$\c:=\c^0$ consists of the germs
at $+\infty$ of continuous   functions $[a,+\infty)\to\R$ ($a\in\R$), and
 $$ \c\ =\ \c^0 \ \supseteq \c^1 \ \supseteq\ \c^2  \ \supseteq\  \cdots  \ \supseteq\ \c^\infty  \ \supseteq\ \c^\omega.$$
We also consider the  $\C$-subalgebra  $\Gr[\imag]=\Gr+\Gr\imag$ of $\c[\imag]$.
For~$n\ge 1$ we have a derivation $g\mapsto g'\colon \Gn[\imag]\to\c^{n-1}[\imag]$
such that $\text{(germ of $f$)}'=\text{(germ of $f'$)}$ for $f\in\bigcup_a \Can[\imag]$, and
this derivation restricts to a derivation  $\Gn\to\c^{n-1}$.  
Therefore~$\Gi[\imag]  := \bigcap_{n}\, \Gn[\imag]$ 
is naturally a differential ring with ring of constants $\C$, and
$\Gi  := \bigcap_{n}\, \Gn$
is a differential subring of $\Gi[\imag]$ with ring of constants $\R$.
Thus $\Ginf[\imag]$ is a differential subring of $\Gi[\imag]$,  $\Ginf$ of $\Gi$, and $\Gom$ of $\Ginf$. 

\subsection*{Asymptotic relations} We often use the same notation for a $\C$-valued function 
on a subset of~$\R$ containing an interval $(a, +\infty)$  as for its germ if the resulting ambiguity is harmless.
With this convention, given a property~$P$ of complex numbers
and~${g\in \c[\imag]}$ we say that {\em $P\big(g(t)\big)$ holds eventually\/} if~$P\big(g(t)\big)$ holds for all sufficiently large real $t$.
We equip $\c$ with the partial ordering given by~$f\leq g:\Leftrightarrow f(t)\leq g(t)$, eventually,
and equip $\c[\imag]$ with the asymptotic relations~$\preceq$,~$\prec$,~$\sim$ defined as follows: for $f,g\in \c[\imag]$,
\begin{align*} f\preceq g\quad &:\Longleftrightarrow\quad \text{there exists $c\in \R^{>}$ such that $|f|\le c|g|$,}\\
f\prec g\quad &:\Longleftrightarrow\quad \text{$g\in \c[\imag]^\times$ and $\lim_{t\to \infty} f(t)/g(t)=0$} \\
 &\phantom{:} \Longleftrightarrow\quad \text{$g\in \c[\imag]^\times$ and $\abs{f}\leq c\abs{g}$ for all $c\in\R^>$},\\
f\sim g\quad &:\Longleftrightarrow\quad \text{$g\in \c[\imag]^\times$ and
$\lim_{t\to \infty} f(t)/g(t)=1$}\\ 
\quad&\phantom{:} \Longleftrightarrow\quad f-g\prec g.
\end{align*}
%For $f,g\in \c[\imag]$ we also set
%$$f\asymp g:\ \Leftrightarrow\ f\preceq g \ \&\ g\preceq f,\qquad f\succeq g:\ \Leftrightarrow\ g\preceq f,\qquad f\succ g:\ \Leftrightarrow\ g\prec f,$$
%so $\asymp$ is an equivalence relation on $\c[\imag]$, and $f\sim g\Rightarrow f\asymp g$.  

\subsection*{Oscillation}
We say that $f\in\c_a$ {\it oscillates}\/ if $f(t)=0$ for arbitrarily large~$t\ge a$ and~${f(t)\ne 0}$ for
arbitrarily large $t\ge a$. Hence $f\in\c_a$ is non-oscillating (that is, does not oscillate) if either $f(t)<0$  eventually, or $f(t)=0$  eventually, or $f(t)>0$  eventually.
We also say that a germ   $g\in\c$ oscillates if some (equivalently, every) representative of~$g$ in~$\bigcup_a\c_a$ oscillates. So $g\in\c$ is non-oscillating iff $g=0$ or~$g\in\c^\times$.

\subsection*{Second-order linear differential equations}
Let $f\in \c_a$. We shall  consider the associated differential equation
\begin{equation}\label{eq:2nd order}\tag{L}
 Y'' + fY\ =\ 0. 
 \end{equation}
The solutions~${y\in \Cat}$ to \eqref{eq:2nd order} form
an $\R$-linear subspace $\Sol(f)$ of $\Cat$.   The
solutions~${y\in \Cat[\imag]}$ to \eqref{eq:2nd order} are the $y_1+y_2\imag$ with
$y_1, y_2\in \Sol(f)$ and form
a $\C$-linear subspace $\Sol_{\C}(f)$ of~$\Cat[\imag]$. For any pair $(c,d)\in\C^2$ there is a unique solution~$y\in \Cat[\imag]$ to~\eqref{eq:2nd order} with~${y(a)=c}$ and $y'(a)=d$, and the map that assigns to~$(c,d)$
this unique solution is an isomorphism $\C^2\to \Sol_{\C}(f)$ of 
$\C$-linear spaces, which restricts to an isomorphism~${\R^2\to \Sol(f)}$ of $\R$-linear spaces. 
We have $f\in \Can\Rightarrow\Sol(f)\subseteq \c^{n+2}_a$ (hence~$f\in \c^\infty_a\Rightarrow {\Sol(f)\subseteq \c^\infty_a}$)  and $f\in \Caom\Rightarrow \Sol(f)\subseteq \Caom$. (See~\cite[(10.5.3)]{Dieudonne}.)    

\medskip
\noindent
Let 
$y\in \Sol(f)^{\neq}$, and  let $Z:=y^{-1}(0)$ be the set of zeros of~$y$, so $Z\subseteq [a,+\infty)$ is closed in $\R$, and  $Z$ has no limit point (by Rolle).
In particular, $Z\cap[a,b]$  is finite for every $b\ge a$. Thus
$$\text{$y$ does not oscillate}
\quad \Longleftrightarrow\quad   \text{$Z$ is finite}\quad \Longleftrightarrow\quad \text{$Z$ is bounded.}$$
If $t_0\in Z$    is not the largest element of $Z$, then $Z\cap (t_0,t_1)=\emptyset$ for some $t_1> t_0$ in~$Z$.
Call a pair of zeros~$t_0<t_1$ of $y$   {\it consecutive}\/ if $Z\cap(t_0,t_1)=\emptyset$. Sturm's Separation Theorem says that if $y,z\in\Sol(f)$ are $\R$-linearly independent and   $t_0<t_1$ are consecutive zeros of $z$,  then $(t_0,t_1)$ contains a unique 
zero  of~$y$~\cite[Chapter~2, \S{}6, Theorem~7]{BR}.  Thus
some $y\in\Sol(f)^{\neq}$ oscillates iff every $y\in\Sol(f)^{\neq}$ oscillates.
 We say that {\it $f$ generates oscillation}\/ if some $y\in\Sol(f)^{\ne}$  oscillates. 
If $b\ge a$, then~$f$ generates oscillation iff $f|_{b}\in\c_b$ does. %; cf.~\cite[Lemma~3.8]{ADH5}.
Hence whether $f$ generates oscillation depends only on its germ in $\c$, and this induces the notion of an element of $\c$ generating oscillation. 
By Sturm's Comparison Theorem~\cite[\S{}27, VI]{Walter}, if $g\in\c_a$  generates oscillation and   $f(t)\geq g(t)$, eventually, then~$f$ also generates oscillation.

\begin{exampleNumbered}\label{ex:harmonic osc}
Let $k\in\R^\times$. Then $y\in \Cat$ satisfies $y'' + k^2 y\ =\ 0$  iff for some $c,d\in\R$ and all $t\ge a$,
$$y(t)\ =\  c\cos k(t+d).$$
For $c\ne 0$, any   $y\in \Cat$ as displayed oscillates.  
Thus
if~${f(t)\geq \varepsilon}$, eventually, for some~$\varepsilon\in\R^>$,  then $f$ generates oscillation. 
\end{exampleNumbered}

\noindent
Let now $g\in \Cao$, $h\in\Caz$. The remarks about solutions to  \eqref{eq:2nd order} above also yield information about solutions in $\Cat$ to the differential equation
\begin{equation}\label{eq:2nd order, gen}\tag{$\tilde{\operatorname{L}}$}
Y''+gY'+hY\ =\ 0.
\end{equation}
To see how,  put $f:=-\frac{1}{2}g'-\frac{1}{4}g^2+h\in \c_a$, and let $G\in (\Cat)^\times$ be given by~$G(t):=\exp\!\left(-\frac{1}{2}\int_a^t g(s)\,ds\right)$. Then $y\in\Cat$ is a solution to \eqref{eq:2nd order} iff $Gy$ is a solution to \eqref{eq:2nd order, gen}. Hence the   solutions~${y\in \Cat}$ to \eqref{eq:2nd order, gen} form an $\R$-linear subspace $V$ of $\Cat$ of dimension~$2$, where $V\subseteq\c^\infty_a$ if $g,h\in \c^\infty_a$, and likewise with $\c^\omega_a$ in place of $\c^\infty_a$.  Thus:
$$ \text{ some $y\in V^{\ne}$  oscillates}\ \Longleftrightarrow\ \text{all $y\in V^{\ne}$ oscillate}\ \Longleftrightarrow\  
f \text{ generates oscillation}.
$$
The solutions to~\eqref{eq:2nd order, gen} in $\Cat[\imag]$ form the $2$-dimensional  subspace $V+V\imag$ of the $\C$-linear space~$\Cat[\imag]$. Moreover,

\subsection*{Theorems of Sonin-P\'olya and Trench}
Let $f\in\c_a$. For $y\in\c_a^1$, call $t\ge a$
a {\it critical point}\/ of $y$ if $y'(t)=0$, and 
an {\it extremal point}\/ of $y$ if $y$ has a local extremum at $t$.
Let now~$y\in\Sol(f)^{\ne}$. We already noted that the set $y^{-1}(0)$ of zeros of $y$ has no limit point. 
If $f\in\c_a^\times$ and $y'(t)=0$, $t\ge a$, then $y''(t)\ne 0$ (in view of~$y''(t)=-f(t)y(t)$), so $t$ is isolated in the set  $(y')^{-1}(0)$  of critical points of $y$ and~$t$ is an extremal point of $y$.  Moreover:

\begin{lemma}\label{uniquestat}  Suppose  $f\in\c_a^\times$, $y\in\Sol(f)^{\ne}$,  and  $s_0<s_1$ are consecutive zeros of~$y$. Then there is exactly one critical point of $y$ in the interval $(s_0,s_1)$.
\end{lemma}
\begin{proof} If $t_0< t_1$ were critical points of $y$ in the interval $(s_0, s_1)$, then $y''$ (by Rolle) and thus $y$ (in view of $y''=-fy$) would have a zero in the interval $(t_0, t_1)$. 
\end{proof}

\begin{lemma}[Sonin-P\'olya, see {\cite[\S{}27, XI]{Walter}}]\label{lem:extremal pts}
Suppose $f\in(\c_a^1)^\times$, $y\in\Sol(f)^{\ne}$, and
let $t_0<t_1$ be critical points of $y$.
If $f$ is increasing, then~$|y(t_0)|\geq |y(t_1)|$. If $f$ is decreasing, then~$|y(t_0)|\leq |y(t_1)|$.
 If $f$ is strictly increasing, respectively strictly decreasing, then these inequalities are strict. 
\end{lemma}
\begin{proof}
Put $u:=y^2+ \big((y')^2/f\big)\in\c_a^1$. Then $u'=-f'(y'/f)^2$,
so if $f$ is increasing, then  $u$ is decreasing, and
as $u(t_i)=y(t_i)^2$ for $i=0,1$, we get $|y(t_0)|\geq |y(t_1)|$.
The other cases are similar, using for the strict inequalities  that    $(y')^{-1}(0)$  has no limit point.
\end{proof}

\noindent
The function $u$ in the proof of Lemma~\ref{lem:extremal pts} agrees with the function $y^2$ at the critical points of $y$.
The function $v$ in the next proposition serves a similar purpose, but turns out to be more convenient
for us in Section~\ref{sec:zeros}. To state this proposition, let~$y_1,y_2\in\Sol(f)$, with 
Wronskian $w=y_1y_2'-y_1'y_2$. Then $w\in\R$, and
$$w\neq 0\ \Longleftrightarrow\ \text{$y_1$,~$y_2$ are $\R$-linearly independent.}$$ 
Let $c_1,c_2\in\R$ with $c_1^2+c_2^2=1$ and $y:=c_1y_1+c_2y_2\in\Sol(f)$.
The following observation is due to Trench~\cite[Theorem~1]{Trench}:

\begin{prop} \label{prop:Trench}
Suppose $w\ne 0$, and consider    
$$v:=\frac{\abs{w}}{\sqrt{(y_1')^2+(y_2')^2}}\in\c^1_a.$$ %\qquad u:= \frac{\abs{w}}{\sqrt{(y_1)^2+(y_2)^2}}$$
Then for $t\ge a$ we have:
$y'(t)=0\ \Longrightarrow\ \abs{y(t)}=v(t)$. %\qquad y(t)=0\ \Longrightarrow\ \abs{y'(t)}=u(t).$$
\end{prop}

\noindent
This follows immediately from the next elementary lemma applied to 
$a_i=y_i'(t)$, $b_i=y_i(t)$ ($i=1,2$) where $t\ge a$,  and $c_1$, $c_2$    as above. % and $a_i=y_i(t)$, $b_i=y_i'(t)$ ($i=1,2$), respectively. 

\begin{lemma}
Let $a_1,a_2,b_1,b_2\in\R$ with $a_1^2+a_2^2>0$,   
and set
$$r:=c_1a_1+c_2a_2,\qquad s:=c_1b_1+c_2b_2.$$
If $r=0$, then
$\abs{s} = \abs{a_1b_2-a_2b_1}/\sqrt{a_1^2+a_2^2}$.
\end{lemma}
\begin{proof} For $d:= a_1b_2-a_2b_1$ we have
\[ \begin{pmatrix} a_1 & a_2\\
                           b_1 & b_2 \end{pmatrix}
     \begin{pmatrix} c_1\\
                              c_2 \end{pmatrix} \  =\  \begin{pmatrix} r\\
                                                                                           s \end{pmatrix}\ 
        \text{ and }\  
                             \begin{pmatrix} b_2 & -a_2\\ 
                                                      -b_1 & a_1\end{pmatrix}  \begin{pmatrix} a_1 & a_2\\
                           b_1 & b_2 \end{pmatrix}  = \begin{pmatrix}d & 0\\
                                                                                                 0 & d\end{pmatrix}, \]
so $dc_1=b_2r-a_2s$, $dc_2=-b_1r+a_1s$.  Hence $d^2(c_1^2+ c_2^2)=(b_2r-a_2s)^2+ (-b_1r+a_1s)^2$.   Using $c_1^2 +c_2^2=1$ gives the desired result for $r=0$. 
           \end{proof}

\subsection*{Composition of germs} 
Let $g\in \c$ satisfy~${\lim\limits_{t\to+\infty} g(t)=+\infty}$; equivalently,  $g\geq 0$ and $g\succ 1$. Then the 
composition operation
$$f\mapsto f\circ g\ :\ \c  \to \c, \qquad (f\circ g)(t)\ :=\  f\big(g(t)\big)\ \text{ eventually},$$
is an injective endomorphism of the ring $\c$
that is the identity on the subring~$\R$. For~${f_1, f_2\in \c}$ we have:  $f_1\le f_2 \Leftrightarrow f_1\circ g \le f_2\circ g$, and likewise with $\preceq$, $\prec$, $\sim$.  
Suppose next that $g$ is also eventually strictly increasing.
Then its compositional inverse $g^{\inv}\in \c$ is given by~$g^{\inv}\big(g(t)\big)=t$, eventually,
and~$g^{\inv}$ is also  eventually strictly increasing with~$g^{\inv}\ge 0$ and~$g^{\inv}\succ 1$.
Then~${f\mapsto f\circ g}$  is an automorphism of the ring  $\c$, with inverse~$f\mapsto f\circ g^{\inv}$. 
In particular, $g\circ g^{\inv}=g^{\inv}\circ g=x$ (where~$x$ is the germ of the identity function). Moreover, 
if  $h\in\c$ is eventually strictly increasing with  $g\leq h$, then~$h^{\operatorname{inv}}\leq g^{\operatorname{inv}}$.
For use in Section~\ref{sec:zeros} we note:

\begin{lemma}\label{lem:ginv}
Let $g,h\in\c$ be eventually strictly increasing with $g,h\succ 1$, and suppose~$h\in\c^1$, $h^\dagger\preceq 1/x$.
Then $g\sim h^{\operatorname{inv}}\Rightarrow g^{\operatorname{inv}}\sim h$.
\end{lemma}

\noindent
This follows by taking $f=h^{\operatorname{inv}}$ in \cite[Corollary~2.6]{ADH5}.

\subsection*{Compositional  conjugation of differentiable germs}  
Let~$\ell\in \Go$, $\ell'(t)>0$ eventually (so $\ell$ is eventually strictly increasing) and $\ell(t)\to+\infty$ as $t\to+\infty$. Then~$\phi:=\ell'\in\c^\times$,  and the compositional inverse 
$\ell^{\inv}\in \Go$ of $\ell$ satisfies 
$$\ell^{\inv}\ge t \text{ for all }t\in \R, \qquad (\ell^{\inv})'\ =\ (1/\phi)\circ \ell^{\inv}\in \c.$$
The $\R$-algebra automorphism 
$f\mapsto f^\circ:= f\circ \ell^{\inv}$ of $\mathcal{C}$ (with inverse $g\mapsto g\circ\ell$) maps~$\Go$ onto itself and  
satisfies for $f\in\Go$ a useful identity:
\begin{equation}\label{eq:fcirc}
(f^\circ)'\ =\ (f\circ \ell^{\inv})'\ =\ (f'\circ \ell^{\inv})\cdot (\ell^{\inv})'\ =\ (f'/\ell')\circ \ell^{\inv}
\ =\ (\phi^{-1}f')^\circ.
\end{equation}
Hence if $n\ge 1$ and $\ell\in\c^n$, then $\ell^{\inv}\in \c^n$ and $f\mapsto f^\circ$ maps $\Caln$
onto itself, for each $n$. Therefore, if $\ell\in\Calinf$, then $\ell^{\inv}\in \Calinf$ and $f\mapsto f^\circ$   maps~$\Calinf$ onto itself;
likewise with $\Ginf$ or $\Gom$ in place of $\Calinf$. 

\section{Hardy Fields}\label{sec:prelims}

\noindent 
In this section we place Hardy fields into the algebraic framework established in Sections~\ref{sec:diffalg} and~\ref{sec:ADH}.
We then discuss composition and generation of oscillation in the setting of Hardy fields.
In preparation for the following sections, 
we also review the main results of \cite{ADH5, ADH2} and establish some lemmas.

\subsection*{Hardy fields}
A {\it Hardy field}\/ is a differential subfield of $\Calinf$. {\em Examples}: $\R$; any subfield of $\R$;  the field $\R(x)$ where $x$  is the germ of the identity function on $\R$; also $H_{\text{LE}}$, the Hardy field of (germs of) LE-functions: the smallest subfield of $\Calinf$ that extends $\R(x)$, contains $\ex^f$ whenever it contains $f$, and contains $\log f$ whenever it contains $f$ with $f(t)>0$ eventually.  
{\it In the rest of this section $H$ is a Hardy field.}\/
 The restriction of the partial ordering of $\c$ to $H$ is total and makes $H$ into an ordered field,
 so no germ in $H$ oscillates.
We have~$C_H\subseteq\R$, and by Section~\ref{sec:ADH}, subsection on $H$-fields, $H$ has the valuation ring
$$\mathcal{O}\ =\ \{f\in H:\ \text{$\abs{f}\le n$ for some $n$} \},$$
with maximal ideal
$$\smallo\ =\ \{f\in H:\ \text{$\abs{f}\le 1/n$ for all $n\ge 1$} \},$$
and we consider $H$ accordingly as an ordered and valued differential field. Then~$H$ is $H$-asymptotic with small derivation, 
and if $H\supseteq\R$, then $H$ is an $H$-field. Moreover, $K:=H[\imag]$ is a differential subfield of $\Calinf[\imag]$ with constant field~${K\cap\C}$, and~${\mathcal{O}+\mathcal{O}\imag} = \big\{f\in K: f\preceq 1\big\}$ is the unique valuation ring of $K$ whose intersection with $H$ is $\mathcal{O}$. In this way
$K$ is a valued differential field extension of~$H$, and is $H$-asymptotic   with small derivation and
the same asymptotic couple $(\Gamma,\psi)$ as $K$.
The asymptotic relations~$\preceq$,~$\prec$,~$\sim$ on~$\c[\imag]$
restricted to~$K$ are exactly the asymptotic relations~$\preceq$,~$\prec$,~$\sim$ on $K$ that~$K$ has as a valued field; likewise with $H$ in place of~$K$.  %Unlike for arbitrary $H$-asymptotic fields,
For $f,g\in K$ with $f,g\succ 1$ we have
\begin{equation}\label{eq:9.1.11}
\psi(vf)\ge\psi(vg)  \quad\Longleftrightarrow\quad f^\dagger \preceq g^\dagger  \quad\Longleftrightarrow\quad \text{$\abs{f}\le \abs{g}^n$  for some $n$.}
\end{equation}
Hence the set $\Psi=\big\{\psi(\gamma):\gamma\in\Gamma^{\ne}\big\}$ is finite iff the ordered abelian group $\Gamma$ has finite (archimedean) rank.
This is the case if $H$ is of finite transcendence degree over its constant field, by~\cite[Pro\-po\-sition~5]{Rosenlicht83} (see also~\cite[Lem\-ma~5.26]{ADH5}).

\subsection*{Hardian germs}
A germ $y\in\c$ is said to be {\it hardian}\/ if it lies in a Hardy field  (and thus $y\in\Calinf$), and {\it $H$-har\-dian}  if it lies in a Hardy field extension
of~$H$. (Thus~$y$ is hardian iff~$y$ is $\Q$-hardian.)
We say that $y\in\c$ is {\it perfectly $H$-hardian}\/ if~$y$ is $E$-hardian for each Hardy field extension $E$ of $H$,
and {\it perfectly hardian}\/ if it is perfectly $\Q$-hardian.
Note: $y\in\c$ is perfectly $H$-hardian iff $y$ is contained in each maximal Hardy field extension of $H$,
where a Hardy field is {\it maximal}\/ if it has no a proper Hardy field extension. (By Zorn, each Hardy field has a maximal
Hardy field extension.) Thus the perfectly $H$-hardian germs form a Hardy field extension~$\Ex(H)$ of $H$, the {\it perfect hull}\/ of $H$, which equals
the intersection of all maximal Hardy field extensions of $H$.  
We also consider the {\it $\d$-perfect hull}\/ of~$H$: the Hardy subfield
$$\Dx(H) := \big\{ f\in \Ex(H) : \text{$f$ is $\d$-algebraic over $H$}\big\}$$
of $\Ex(H)$, which contains $H$. 
We say that $H$ is  {\it perfect}\/ if~$H=\Ex(H)$, and {\it $\d$-perfect}\/ if~$H=\Dx(H)$.
If $H$ is $\d$-perfect, then $H\supseteq\R$ and $H$ is a Liouville closed $H$-field by~\cite[Proposition~4.2]{ADH5}.
In particular,  $\Dx(H)\supseteq\R$ is a Liouville closed $H$-field,
and   $H$ is pre-$\d$-valued.
Passing from $H$ to the Hardy subfield   $H(\R)$   of $\Dx(H)$ often allows us to arrange that~$H\supseteq\R$ and hence that~$H$ is an $H$-field.
Likewise, the algebraic closure~$H^{\operatorname{rc}}$ of~$H$ in $\Dx(H)$ is a real closure of the ordered field $H$, and we can often replace~$H$ by~$H^{\operatorname{rc}}$. In  this connection we   observe:

\begin{lemma}\label{lem:1.3.3+}
The value group $\Gamma_H$ of $H$ is cofinal in $\Gamma_{H(\R)}$, and
$\Gamma_{H^{\operatorname{rc}}}$ is the divisible hull of $\Gamma_H$.
If one of the Hardy fields~$H$,~$H(\R)$,~$H^{\operatorname{rc}}$ is $\upl$-free, then so are all three, and  in this case  $H(\R)$ has value group~$\Gamma_H$.
%, and likewise with ``$\upo$-free'' in place of ``$\upl$-free''. 
\end{lemma}
\begin{proof}
The cofinality of $\Gamma_H$ in $\Gamma_{H(\R)}$ follows from [ADH, 10.3.2, 10.5.15], and the claim about $\Gamma_{H^{\operatorname{rc}}}$ from [ADH, 3.1.9]. 
If $H$ is $\upl$-free, then so is~$H(\R)$, with the same value group, by 
\cite[Proposition~1.3.3]{ADH4}, and
if $H(\R)$ is  $\upl$-free, then so is $H$, by~\cite[Lemmas~1.3.13, 1.3.14]{ADH4}.
By~[ADH, 11.6.8], $H$ is $\upl$-free iff  $H^{\operatorname{rc}}$ is $\upl$-free.
\end{proof}

\noindent
Suppose~${H\supseteq\R}$. As in~[ADH, p.~460], the {\it Hardy-Liouville closure $\Li(H)$}\/ of $H$ is the
smallest  Hardy field extension~$L$ of $H$ that is a Liouville  closed $H$-field.
We have~$\Li(H)\subseteq\Dx(H)$, and $\Li(H)$ is a Liouville extension of its differential subfield~$H$. 
If $H$ is grounded or $\upl$-free, then by Section~\ref{sec:ADH}, $\Li(H)$ is the unique, up to isomorphism over $H$,  Liouville closure of the $H$-field $H$.  
Note that~$\Li(\R)\supseteq H_{\text{LE}}$; but
whereas~$\Li(\R)$ contains the $g\in\c^1$ with~$g'=\ex^{x^2}$, no such $g$ is in $H_{\text{LE}}$ (by Liouville).
For the next result, see~\cite[Proposition~6.9]{ADH5} or~\cite[Lemma 11.6(6)]{Boshernitzan82}:

\begin{prop}\label{prop:6.9}
If $\phi\in H$ and $\phi\preceq 1$, then $\cos\phi,\sin\phi\in\Dx(H)$.
\end{prop}

\subsection*{Smooth and analytic Hardy fields}
We also consider Hardy fields with added regularity requirements:
we call $H$ 
a {\em $\Ginf$-Hardy field\/} (or a {\em smooth Hardy field}\/) if~$H\subseteq \Ginf$, and a {\em $\Gom$-Hardy field\/} (or an {\em analytic Hardy field\/}) if~${H\subseteq \Gom}$.
If $H$ is a $\Ginf$-Hardy field and~$y\in\Ginf$ is $H$-hardian, then $y$ generates a $\Ginf$-Hardy field extension~$H\langle y\rangle$ of $H$; likewise with $\Gom$ in place of $\Ginf$.
 Moreover, if~$H\subseteq\Ginf$ 
 and~$E$ is a $\d$-algebraic Hardy field extension of $H$, 
 then~$E\subseteq\Ginf$;
similarly with~$\Gom$ in place of~$\Ginf$; see~\cite[Corollary~7.8]{ADH2}.
In particular, if $H\subseteq\Ginf$, then~$\Dx(H)\subseteq\Ginf$; likewise
with $\Gom$ in place of~$\Ginf$.

\subsection*{Hardy fields and composition}
 Let~$\ell\in \Go$ be such that $\ell>_{\ev}\R$ and~$\ell'\in H$. Then $\ell\in \Calinf$,  $\phi:=\ell'$
is active in $H$, $\phi>0$, and we have a Hardy field~$H(\ell)$. 
%and the compositional inverse 
%$\ell^{\inv}\in \Go$ of $\ell$ satisfies 
%$$\ell^{\inv}>\R, \qquad (\ell^{\inv})'\ =\ (1/\phi)\circ \ell^{\inv}\in H\circ \ell^{\inv}.$$
 The
$\R$-algebra automorphism~$f\mapsto f^\circ:= f\circ \ell^{\inv}$ of~$\mathcal{C}$  restricts to an ordered field isomorphism
$H \to H^\circ:=H\circ \ell^{\inv}$.
The identity~${(f^\circ)'  = (\phi^{-1}f')^\circ}$ from~\eqref{eq:fcirc}, valid for each $f\in\Go$, shows
that~$H^\circ$ is again a Hardy field,  and $h\mapsto h^\circ\colon H^\phi\to H^\circ$ is an isomorphism of ordered and
valued   differential fields. Moreover,  if $H\subseteq \Ginf$ and~$\ell\in \Ginf$, then $H^\circ\subseteq \Ginf$ ; likewise with~$\Gom$ instead of~$\Ginf$.

\begin{exampleNumbered}\label{ex:Hcircx+d}
Let  $d\in\R$. We have the 
Hardy field~$H\circ (x+d)$, with~$H\circ (x+d)\subseteq\Ginf$ if $H\subseteq\Ginf$ and $H\circ (x+d)\subseteq\Gom$ if $H\subseteq\Gom$. The map $h\mapsto h\circ(x+d)\colon H\to H\circ(x+d)$ is an isomorphism of ordered valued differential fields. (Take~$\ell=x-d$ above.)   
\end{exampleNumbered}

\noindent
The next lemma is
 \cite[Corollary~6.5]{Boshernitzan81}. %; see also \cite[Theo\-rem~1.7]{AvdD4}.
 
\begin{lemma}\label{lem:Bosh6.5}
The germ $\ell^{\inv}$ is hardian. Moreover,
if $\ell$ is $\Ginf$-hardian, then $\ell^{\inv}$ is also $\Ginf$-hardian, and
likewise with  $\Gom$  in place of $\Ginf$. 
\end{lemma}
\begin{proof}
Replacing $H$ by $\Li\!\big(H(\R)\big)$ we   arrange that our Hardy field $H$ contains both~$\ell$ and~$x$.
Then $\ell^{\inv}=x\circ \ell^{\inv}$ is an element of the Hardy field $H^\circ=H\circ\ell^{\operatorname{inv}}$.
\end{proof}

\noindent
The proof of Lemma~\ref{lem:Bosh6.5} shows that  $\ell$ and
$\ell^{\operatorname{inv}}$ are both $\R(x)$-hardian; moreover (cf.~\cite[Lem\-ma~14.10]{Boshernitzan82}):

\begin{lemma}\label{lem:trdegellinv}
We have
$$\operatorname{trdeg}\!\big(\R\langle x, \ell^{\operatorname{inv}}\rangle|\R\big)\  =\ 
\operatorname{trdeg}\!\big(\R\langle x, \ell\rangle|\R\big),$$
hence if $\ell$ is $\d$-algebraic over $\R$, then so is 
$\ell^{\operatorname{inv}}$,  with
$$\operatorname{trdeg}\!\big(\R\langle \ell^{\operatorname{inv}}\rangle|\R\big)\  \leq\
1+\operatorname{trdeg}\!\big(\R\langle \ell\rangle|\R\big).$$
\end{lemma}

\noindent
This follows by taking $H(x)$, $x$ for $H$, $h$ in the following lemma:

\begin{lemma}\label{lem:trdeghcirc}
Let $h\in H$. Then 
\begin{equation}
\operatorname{trdeg}\!\big(\R\langle x, \ell^{\operatorname{inv}}, h^\circ\rangle|\R\big)\  =\ 
\operatorname{trdeg}\!\big(\R\langle x, \ell, h\rangle|\R\big), \label{eq:trdegellinv}
\end{equation}
hence if $\ell$, $h$ are $\d$-algebraic over $\R$, then so is $h^\circ$, with
$$\operatorname{trdeg}\!\big(\R\langle h^\circ\rangle|\R\big)\  \leq\ 
\operatorname{trdeg}\!\big(\R\langle x,h\rangle|\R\big)+\operatorname{trdeg}\!\big(\R\langle \ell\rangle|\R\big).$$
\end{lemma}
\begin{proof}
Arrange   $H=\R\langle x, \ell, h\rangle=\R(x)\langle \ell, h\rangle$ and set $\phi:=\ell'$.
With $\der$ and~$\derdelta=\phi^{-1}\der$ the derivations of~$H$ and $H^\phi$  we have $\phi=1/\derdelta(x)$, and for $f\in H$ and $n\geq 1$,
$$\der^n(f)\in \Q\big[\derdelta(f),\derdelta^2(f),\dots,\phi,\derdelta(\phi),\derdelta^2(\phi),\dots\big]$$ 
by the subsection on compositional conjugation in Section~\ref{sec:diffalg}.  The differential fields~$H$ and $H^\phi$ have the same underlying field, and the former is generated as a field over $\R$ by $x$ and the $\ell^{(n)}$, $h^{(n)}$, so  applying the above to $f=\ell$ and $f=h$ shows that $H^\phi$ is generated as a differential field~over~$\R$ by $x$, $\ell$, and $h$.
We also have a differential field isomorphism
$h\mapsto h^\circ\colon H^\phi\to H^\circ=H\circ \ell^{\operatorname{inv}}$.
This yields~$H^\circ=\R\langle\ell^{\operatorname{inv}},x,h^\circ\rangle$ and  
\eqref{eq:trdegellinv}. 
%Suppose now that~$\ell$, $h$ are $\d$-algebraic over~$\R$; then
%by additivity of~$\operatorname{trdeg}$,      
%\begin{align*}
%\operatorname{trdeg}\!\big(\R\langle x, \ell, h\rangle|\R\big) &\ =\ 
%\operatorname{trdeg}\!\big(\R\langle \ell,x,h\rangle|\R\langle\ell\rangle\big)+
%\operatorname{trdeg}\!\big(\R\langle \ell\rangle|\R\big) \\ &\ \leq\ 
%\operatorname{trdeg}\!\big(\R\langle x,h\rangle|\R\big)+\operatorname{trdeg}\!\big(\R\langle \ell\rangle|\R\big),
%\end{align*}
%and so by \eqref{eq:trdegellinv}:
%$$\operatorname{trdeg}\!\big(\R\langle h^\circ\rangle|\R\big)\ \leq\ 
%\operatorname{trdeg}\!\big(\R\langle x, \ell^{\operatorname{inv}},h^\circ\rangle|\R\big)\ \leq\
%\operatorname{trdeg}\!\big(\R\langle x,h\rangle|\R\big)+\operatorname{trdeg}\!\big(\R\langle \ell\rangle|\R\big),$$
%hence  $h^\circ$ is $\d$-algebraic over $\R$.
\end{proof}
 
\begin{cor}
If $H\supseteq\R$ is $\d$-algebraic, then so is $H^\circ\supseteq\R$.
\end{cor}

\subsection*{Addition of sinusoids}  
Recall that for $\phi,\theta\in \R$ we have
\begin{align*} \cos(\phi+\theta)\  &=\ \cos(\phi)\cos(\theta)-\sin(\phi)\sin(\theta), \\  
\cos(\phi-\theta)\  &=\ \cos(\phi)\cos(\theta)+\sin(\phi)\sin(\theta).
\end{align*} 
Recall also the bijection $\arccos\colon [-1,1]\to [0,\pi]$, the inverse of the cosine function on $[0,\pi]$. It follows that for any $a,b\in \R$ we have $d\in \R$ such that $$a\cos(\phi)+b\sin(\phi)\ =\ \sqrt{a^2+b^2}\cdot \cos(\phi+d) \text{ for all $\phi\in \R$:}$$ 
for $a$, $b$ not both $0$ this holds
with $d=\arccos\!\big(a/\sqrt{a^2+b^2}\big)$ when $b\leqslant 0$, and with~$d=-\arccos\!\big(a/\sqrt{a^2+b^2}\big)$
when $b\geqslant 0$. For later use we record some consequences:

\begin{lemma}\label{addsin} 
Let $y\in\c$. Then
$$y=a\cos x+b\sin x \text{ for some $a,b\in\R$} \quad \Longleftrightarrow\quad y=c\cos(x+d) \text{ for some $c,d\in\R$.}$$ 
\end{lemma}
%\begin{proof}
%Suppose $y=a\cos x+b\sin x$. Then with $z=a-b\imag\in\mathbb C$ we have
%$y=\Re(z\ex^{x\imag})$; taking $c,d\in\R$ with $z=c\ex^{d\imag}$ we then have
%$y=\Re(c\ex^{(x+d)\imag})=c\cos(x+d)$. Conversely, if $c,d\in\R$
%satisfy $y=c\cos(x+d)$, then $y=\Re(z\ex^{x\imag})$ for $z=c\ex^{d\imag}$, so $y=a\cos x+b\sin x$ for $a=\Re z$, $b=-\Im z$.
%\end{proof}

\begin{cor}\label{cor:sinusoids}
Let $\phi\in\c$. Then  
$\R \cos\phi + \R \sin\phi\, =\,\big\{  c\cos(\phi+d) :\, c,d\in\R \big\}$.
\end{cor}

\noindent
Next two facts to be used in the proofs of Lemmas~\ref{lem:param V'} and~\ref{lem:param zeros of derivatives}: 

\begin{cor}\label{arccosH} Suppose $H\supseteq \R$ is real closed and closed under integration, and let $g,h\in H$. Then there is $u\in H$ such that $-\pi \leqslant u \leqslant \pi$ and $g\cos \phi + h\sin \phi = \sqrt{g^2+h^2}\cdot \cos (\phi+ u)$ for all $\phi\in \c$:  if $h < 0$ this holds for
$u=\arccos \big(g/\sqrt{g^2+h^2}\big)$, and if $h>0$ it holds for $u=-\arccos \big(g/\sqrt{g^2+h^2}\big)$.
\end{cor} 
\begin{proof} On the interval $(-1,1)$ the function $\arccos$ is real analytic with derivative~$t\mapsto -1/\sqrt{1-t^2}$. 
Thus $\arccos\!\big(g/\sqrt{g^2+h^2}\big)\in H$ for $h\ne 0$. \end{proof}

\begin{cor} \label{gphicos} Let $a\in \R$ and let $g,\phi\in \c_a^1$ have germs in $H$ such that $g(t)\ne 0$ eventually, and
$\phi(t)\to +\infty$ as $t\to +\infty$.  Then there is a real $b\geqslant a$ with the property that if
$s_0, s_1\in [b,+\infty)$ with $s_0 < s_1$ are any successive zeros of $y:= g\cos\phi$, then~$y'$ has exactly one zero in the interval $(s_0,s_1)$.  
\end{cor}
\begin{proof} By increasing $a$ we arrange $g(t)\ne 0$ and $\phi'(t)>0$ for all $t\ge a$. Replacing~$g$ by~$-g$ if necessary we
further arrange $g(t)>0$ for all $t\geqslant a$. Let $s_0, s_1\in [a,+\infty)$ with $s_0<s_1$ be successive zeros of $y$. Later we impose a suitable lower bound $b\geqslant a$ on $s_0$.  Then $\phi(s_1)=\phi(s_0)+\pi$, since $s_1$ is the next zero of $\cos \phi$ after $s_0$. Also
\begin{align*} y'\ &=\ g'\cos\phi - g\phi'\sin \phi\  =\  \sqrt{g'^2+(g\phi')^2}\cos(\phi+u),\ \text{ where}\\
u\ &=\ \arccos\big(g'/\sqrt{g'^2+(g\phi')^2}\big),\ \text{ so }0 < u(t)< \pi \text{ for all $t\geqslant a$.}
\end{align*} 
By Rolle, $y'$ has a zero in $(s_0, s_1)$. Let $t\in (s_0, s_1)$ be a zero of $y'$. Then 
$$\phi(s_0) < \phi(t)< \phi(s_0)+ \pi,\quad \phi(t)+u(t)\in\phi(s_0)+ \Z\pi,$$
so $\phi(t)+ u(t)=\phi(s_0)+ \pi$.  Take $b\geqslant a$ in $\R$ so large that $u$ is differentiable on $[b,+\infty)$ and
$\phi'(t)+u'(t)>0$ for all $t\geqslant b$; this is possible because $u\preccurlyeq 1$ is $H$-hardian by Corollary~\ref{arccosH}, and
$\phi(t)+u(t)\to +\infty$ as $t\to +\infty$. Assuming now that $b\leqslant s_0$, we conclude that $t\in (s_0, s_1)$ is uniquely determined by $\phi(t)+u(t)=\phi(s_0)+\pi$. 
\end{proof} 

\subsection*{Oscillation over Hardy fields}
In this subsection we assume $f\in H$ and consider the 
linear differential equation 
\begin{equation}\label{eq:2nd order, 2} \tag{4L}
4Y''+fY\ =\ 0
\end{equation}
over $H$.  By the facts about \eqref{eq:2nd order} in Section~\ref{sec:germs}, each solution $y\in \Gt$  to \eqref{eq:2nd order, 2} lies in
 $\c^{<\infty}$.
 The solutions to \eqref{eq:2nd order, 2} form a $2$-dimensional subspace~$V$ of the $\R$-linear space~$\c^{<\infty}$. 
 With $A:=4\der^2+f\in H[\der]$, we have
$V=\ker_{\c^{<\infty}} A$ and~$V+V\imag=\ker_{\c^{<\infty}[\imag]}A$. Moreover,
if~$f\in\c^\infty$, then $V\subseteq\c^{\infty}$, and likewise with $\c^\omega$ in place of $\c^\infty$.
By~\cite[Theorem~16.7]{Boshernitzan82} (or \cite[Proposition~6.1]{ADH5}):
 
\begin{prop}\label{prop:2nd order Hardy field}
If $f/4$ does not generate oscillation, then $V\subseteq\Dx(H)$.  
\end{prop}

\noindent
Proposition~\ref{prop:2nd order Hardy field} also has consequences for more general linear differential equations of
order $2$: Let $g,h\in H$, and consider the linear differential equation
\begin{equation}\label{eq:2nd order, 3}\tag{$\tilde{\operatorname{L}}$}
Y''+gY'+hY\ =\ 0
\end{equation}
over $H$. Again by facts about \eqref{eq:2nd order} in Section~\ref{sec:germs} the solution $y\in \Gt$ of \eqref{eq:2nd order, 3} lie in~$\Calinf$. To reduce \eqref{eq:2nd order, 3}
to an equation~\eqref{eq:2nd order, 2} we take 
$$f\ :=\ \omega(g)+4h\ =\ -2g'-g^2+4h\in H,$$ some $a\in \R$ and   a representative of~$g$ in~$\Cao$, also denoted by $g$, and we let~${G\in (\Gt)^\times}$ be 
the germ of 
$t\mapsto \exp\!\left(-\frac{1}{2}\int_a^t g(s)\,ds\right)\colon \R^{\ge a}\to\R$, so $G\in\Li\!\big(H(\R)\big)\subseteq\Dx(H)$.
Then $y\mapsto Gy$ is an isomorphism from the $\R$-linear space of 
solutions to \eqref{eq:2nd order, 2} in~$\Gt$ onto the
$\R$-linear space of solutions to \eqref{eq:2nd order, 3} in~$\Gt$,
and $y\in \Gt$ oscillates iff~$Gy$ does.
Using $\frac{f}{4}=-\frac{1}{2}g'-\frac{1}{4}g^2+h$ we  thus obtain from Section~\ref{sec:germs} and Proposition~\ref{prop:2nd order Hardy field}:
 
\begin{cor}[{cf.~\cite[Theorem~16.8]{Boshernitzan82}}]\label{cor:char osc}
The following are equivalent:
\begin{enumerate}
\item[\textup{(i)}] some solution in $\Gt$ of \eqref{eq:2nd order, 3}  oscillates;
\item[\textup{(ii)}] all nonzero solutions in $\Gt$ of \eqref{eq:2nd order, 3} oscillate;
\item[\textup{(iii)}]  $-\frac{1}{2}g'-\frac{1}{4}g^2+h$ generates oscillation.
\end{enumerate}
Moreover, if $-\frac{1}{2}g'-\frac{1}{4}g^2+h$ does not generate oscillation, then all solutions of~\eqref{eq:2nd order, 3} in $\Gt$ belong to $\Dx(H)$. 
\end{cor}

\noindent
Set $A:=\der^2+ g\der + h\in H[\der]$, and let $f=\omega(g)+4h$ and~$G$ be as above.  Then by~\eqref{eq:5.1.21} and Proposition~\ref{prop:2nd order Hardy field}: 
 
\begin{cor}\label{cor:char osc, 2}
The following are equivalent:
\begin{enumerate}
\item[\textup{(i)}] $f/4$ does not generate oscillation;
\item[\textup{(ii)}] $A$ splits over some Hardy field extension of $H$;
\item[\textup{(iii)}] $A$ splits over $\Dx(H)$.
\end{enumerate}
\end{cor}

\noindent
Recall that $K=H[\imag]$. We have $A_{\ltimes G}=\der^2 + \frac{f}{4}$ and $G^\dagger=-\frac{1}{2}g\in H$, so: 

\begin{cor}\label{cor:char osc, 3}
$A$ splits over $K$ $\Longleftrightarrow$   $\der^2+\frac{f}{4}$ splits over $K$.
\end{cor}

\subsection*{$\upo$-freeness of Hardy fields}  
Let $f\in H$ and $A:=4\der^2+f$.  By \eqref{eq:omega} we have:
$$\ker_H A\ne\{0\}\quad \Longrightarrow\quad \text{$A$ splits over $H$} \quad \Longleftrightarrow\quad f\in\omega(H),$$
and if $H\supseteq\R$ is Liouville closed and 
$A$ splits over $H$, then~$\dim_{\R}\ker_HA=2$, by Lemma~\ref{lem:2.4.5}, and $f/4$ does not generate oscillation.
Hence   
$$\bar{\omega}(H)\ :=\   \big\{ f\in H: \text{$f/4$ does not generate oscillation} \big\}, $$
a downward closed subset of~$H$,  contains $\omega(H)$. If $E$ is a Hardy field extension of~$H$, then $\bar\omega(E)\cap H=\bar\omega(H)$.
 Proposition~\ref{prop:2nd order Hardy field} yields:

\begin{cor}  \label{cor:omega(H) downward closed}
If $H$ is $\d$-perfect \textup{(}in particular, if $H$ is maximal\textup{)}, then 
\[\omega(H)\	=\ \bar\omega(H) 
		\	=\ \big\{ f\in H :\ \text{$4y''+fy=0$ for some $y\in H^\times$} \big\},\]
and so $\omega(H)$ is  downward closed in $H$.
\end{cor}

\noindent
In \cite[Section~6]{ADH5} we determined the possibilities for $\bar{\omega}(H)$;
from this
we only need 
that
$\bar{\omega}(H) < \sigma\big(\Upg(H)\big){}^\uparrow$ 
by the remarks after \cite[Lemma~6.17]{ADH5}, as well as
the following consequences of Lemmas~6.18, 6.19, and Corollary~6.20, respectively, in loc.~cit.:

\begin{lemma}\label{lem:6.18}
Let $\upg\in H^>$ be such that $v\upg=\max\Psi_H$. Then, using [ADH, 9.2.1]:
$$\bar{\omega}(H)\  =\  \omega(-\upg^\dagger) + \upg^2 \smallo_H^\downarrow.$$
\end{lemma}

\begin{lemma}\label{lem:6.19}
If $H$ is $\upl$-free, then, using [ADH, 9.2.17, 11.6.8]: 
$$\bar{\omega}(H)\ =\ \omega\big(\Upl(H)\big){}^\downarrow\ =\ \omega(H)^\downarrow\quad\text{or}\quad \bar{\omega}(H)\ =\ H\setminus \sigma\big(\Upg(H)\big){}^\uparrow.$$
\end{lemma}

\begin{lemma}\label{omuplosc}  
If $H$ is $\upo$-free, then 
$$\bar{\omega}(H)\ =\ \omega\big(\Upl(H)\big){}^\downarrow\ =\ \omega(H)^\downarrow \ = \  H\setminus \sigma\big(\Upg(H)\big){}^\uparrow.$$ 
\end{lemma}

\noindent
Together with the characterization of $\upo$-freeness for $H$-fields in \eqref{eq:upofree},
this yields a version of \cite[Corollary~7.3]{ADH5}:

\begin{cor}\label{cor:char upofree}
If $H\supseteq\R$ is $\upo$-free, then  
$$ \bar{\omega}(H)\ =\ \omega\big(\Upl(H)\big){}^\downarrow\ \text{ and }\ H\setminus\bar{\omega}(H)\ =\ \sigma\big(\Upg(H)\big){}^\uparrow.$$
\end{cor}

\noindent
Not every Liouville closed Hardy field $H\supseteq\R$   is   $\upo$-free: see \cite[Example~1.3.16]{ADH4}.
Let $L\supseteq \R$ be a Liouville closed $\d$-algebraic Hardy field extension of $H$.
If $H$ is $\upo$-free, then so is $L$, by Theorem~\ref{thm:13.6.1}.
But $L$ may also be $\upo$-free for other reasons:

\begin{lemma}[{\cite[Lemma~7.11]{ADH5}}]\label{lem:7.11}
Suppose $H$ is not $\upl$-free or  $\overline{\omega}(H)={H\setminus\sigma\big(\Upg(H)\big){}^\uparrow}$.
If  $\omega(L)=\bar{\omega}(L)$, then $L$ is $\upo$-free. \textup{(}In particular, $\Dx(H)$ is $\upo$-free.\textup{)}
\end{lemma}

\noindent 
Next assume that $H$ is a $\d$-algebraic Hardy field extension of $\R(x)$ which is closed under logarithm, that is, $\log(H^>)\subseteq H$. (Example: $H_{\text{LE}}$.) Then $H$ is ungrounded.
Recursively define the germs $\ell_n\in H^{>\R}$ by~$\ell_0:=x$ and $\ell_{n+1}:=\log\ell_n$.
Then~$(\ell_n)$ is a logarithmic sequence for $H$ as defined in [ADH, p. 499], and $H$ is a union of grounded $H$-subfields containing~$\R(x)$,
 by~\cite[Proposition~14.11]{Boshernitzan82}; see also \cite[Corollary~5.28]{ADH5}.
Thus by Proposition~\ref{prop:11.7.15}, $H$ is $\upo$-free (in particular, has asymptotic integration).
 We introduce the following germs (all in $H$): 
 \begin{align*} \upg_n\ &:=\ \ell_n^\dagger\ =\ \frac{1}{\ell_0\cdots\ell_n},\qquad \upl_n\ :=\ -\upg_n^\dagger\ =\ \frac{1}{\ell_0}+\frac{1}{\ell_0\ell_1} +\cdots + \frac{1}{\ell_0\ell_1\cdots \ell_n}, \\ 
 \upo_n\ &:=\omega(\upl_n)\ =\ \upg_0^2+\upg_1^2+\cdots+\upg_n^2 \,=\, \frac{1}{\ell_0^2}+\frac{1}{(\ell_0\ell_1)^2}+\cdots+\frac{1}{(\ell_0\cdots\ell_{n})^2}, 
 \end{align*} 
 so $\sigma(\upg_n)\ =\ \upo_n+\upg_n^2$. 
These germs are important in oscillation criteria:
%by the remarks following Proposition~\ref{prop:14.2.16+18} and Corollary~\ref{cor:char upofree},
Let $f\in H$, and note that $f\notin \bar{\omega}(H)$ iff  $\frac{f}{4}$ generates oscillation. We have: 
\begin{equation}\label{eq:critosc}
  f\notin  \bar{\omega}(H)\quad\Longleftrightarrow\quad
f > \upo_n\text{ for all $n$} \quad\Longleftrightarrow\quad
f > \upo_n+\upg_n^2 \text{ for some $n$}.
\end{equation}
The first equivalence is due to Hartman~\cite{Hartman48} for $H=H_{\text{LE}}$, and to
Boshernitzan~\cite[Theorem~17.7]{Boshernitzan82} for all $H$ as above; see also \cite[Corollary~7.9]{ADH5}.
The second equivalence follows from \cite[Corollary~7.10]{ADH5} using also 
$ \bar{\omega}(H)<\sigma\big(\Upg(H)\big){}^\uparrow$. 
 
\subsection*{Complex exponentials}
By this we mean germs $\ex^f\in\c[\imag]^\times$ where $f\in\c[\imag]$. In the next subsections these arise naturally when
splitting differential operators over~$K$.
 Recall here that the differential subring $K = H[\imag]$ of $\Calinf[\imag]$ is an
  $H$-asymptotic field extension of~$H$, with valuation ring $\mathcal O=\{f\in K:f\preceq 1\}$.
As in~[ADH, 14.2] we define the $\mathcal O$-submodule
$$\I(K)\ :=\ \{y\in K:\, \text{$y\preceq f'$ for some $f\in\mathcal O$}\}$$ 
of $K$.
The condition  $\I(K)\subseteq K^\dagger$  is related to   {\em trigonometric closedness\/} of $H$  and is equivalent to it for Liouville closed $H$, see \cite[Section~1.2]{ADH4}. More precisely: 

\begin{prop}[{\cite[Proposition~6.11]{ADH5}}]\label{prop:cos sin infinitesimal, 2}
If $H\supseteq\R$ is closed under integration, then
the following conditions are equivalent:
\begin{enumerate}
\item[\textup{(i)}] $\I(K)\subseteq K^\dagger$;
\item[\textup{(ii)}] $\ex^{f}\in K$ for all $f\in K$ with $f\prec  1$;
\item[\textup{(iii)}] $\ex^\phi,\cos \phi, \sin\phi\in H$ for all $\phi\in H$ with $\phi\prec 1$.
\end{enumerate}
\end{prop}

\noindent
Together with Proposition~\ref{prop:6.9} and remarks preceding it, this yields:

\begin{cor}[{\cite[Corollary~6.12]{ADH5}}]\label{cor:cos sin infinitesimal}
If $H$ is $\d$-perfect, then   $\I(K)\subseteq K^\dagger$. 
\end{cor}

\noindent
The following fact~\cite[Lemma~1.10]{ADH6} yields
``polar coordinates'' for   elements of $K^\times$:

\begin{lemma}\label{lem:fexphii}
Suppose $H\supseteq\R$ is Liouville closed  and $f\in \c^1[\imag]^\times$. Then~$f^\dagger\in K$ iff~$\abs{f}\in H^>$ and  $f=\abs{f}\ex^{\phi\imag}$ for some $\phi\in H$.
If in addition $f\in K^\times$, then $f=\abs{f}\ex^{\phi\imag}$ for some $\phi\preceq 1$ in $H$.
\end{lemma}

\noindent
We record some consequences of these results:

\begin{cor}\label{cor:osc => bded}
Let $\phi\in H$, and  suppose $\ex^{\phi\imag}\sim f$ with $f\in E[\imag]^\times$ for some Hardy field extension~$E$ of~$H$.
Then  $\phi\preceq 1$.
\end{cor}
\begin{proof}
We can assume that $E=H$ is Liouville closed and contains $\R$. Towards a contradiction assume $\phi\succ 1$.
Lemma~\ref{lem:fexphii} yields $\theta\preceq 1$ in $H$ such that $f=\abs{f}\ex^{\theta\imag}$. Then
$\ex^{(\phi-\theta)\imag}\sim \abs{f}$ and $\phi-\theta \sim \phi$. Thus replacing $f$, $\phi$ by $\abs{f}$, $\phi-\theta$, respectively,
we  arrange~$f\in H^\times$. Then $\ex^{\phi\imag} = \cos\phi + \imag \sin\phi \sim f$ in $\Calinf[\imag]$ gives $\cos \phi\sim f$, contradicting that
$\cos \phi$ has arbitrarily large zeros.
\end{proof}

\begin{cor}\label{cor:phi preceq 1}
Let $f\in K^\times$, $\phi\in H$, so $y:=f\ex^{\phi\imag}\in\Calinf[\imag]^\times$. Then the following are equivalent:
\begin{enumerate}
\item[\textup{(i)}] $\phi\preceq 1$;
\item[\textup{(ii)}]  $y\in\operatorname{D}(H)[\imag]$;
\item[\textup{(iii)}]  $y\in E[\imag]$ for some Hardy field extension $E$ of $H$;
\item[\textup{(iv)}]  $y\sim g$ for some Hardy field extension $E$ of $H$ and $g\in E[\imag]^\times$.
\end{enumerate}
\end{cor}
\begin{proof} Use Proposition~\ref{prop:cos sin infinitesimal, 2} and Corollaries~\ref{cor:cos sin infinitesimal} and~\ref{cor:osc => bded} to  obtain the chain of implications (i)~$\Rightarrow$~(ii)~$\Rightarrow$~(iii)~$\Rightarrow$~(iv)~$\Rightarrow$~(i). 
\end{proof}

\subsection*{Splitting real operators of order $2$}
{\it In this subsection we assume that $H\supseteq \R$ is   Liouville closed.}\/ Then $K=H[\imag]$ is   algebraically closed  with constant field $\C$. We also let $A\in H[\der]^{\ne}$ be a linear differential operator of order~$2$. 
Here is  a key ingredient for the proof of Theorem~\ref{thm:Bosh} in the next section:  

{\sloppy
\begin{prop}\label{prop:2nd order kernel}
Suppose   $A$ splits over $K$ but not over $H$. Then there are~${g,\phi>0}$ in $H$  
such that   
\begin{equation}\label{eq:2nd order kernel}
\ker_{\Calinf[\imag]} A\ =\ \C g\ex^{\phi\imag}+\C g\ex^{-\phi\imag},\qquad
\ker_{\Calinf} A\ =\ \R g\cos\phi + \R g\sin\phi.
\end{equation}
If $\I(K)\subseteq K^\dagger$, then for all such $g$, $\phi$ we have $\phi\succ 1$.
\end{prop}}
\begin{proof}
The remarks in the subsection ``Linear differential operators'' in Section~\ref{sec:diffalg}
allow us to reduce to the case $A=4\der^2+f$ where $f\in H$, which we assume from now on.
Now \eqref{eq:omega} and \eqref{eq:sigma} yield $u\in H^>$ with $\sigma(u)=f$.
Put $g:=1/\sqrt{u}\in H^>$, take~$\theta\in H$ with $\theta'=\frac{1}{2}u$. 
Arguments before  \cite[Lemma~6.8]{ADH5} show~$A(g\ex^{\theta\imag})=0$, and hence also
$A(g\ex^{-\theta\imag})=0$.
If $\theta > 0$, put $\phi:=\theta$; otherwise, put~$\phi:=-\theta$; then~${\phi> 0}$.
Note that $\dim_{\R}\ker_H A=0$ by \eqref{eq:omega 2} and \eqref{eq:omega},
 hence also $\dim_{\C}\ker_K A=0$. Thus~$\phi\notin\R$, and $y:=g\ex^{\phi\imag},\bar{y}=g\ex^{-\phi\imag}\in \ker_{\Calinf[\imag]} A$ are $\C$-linearly independent.
Now $\Re y = \textstyle\frac{1}{2}(y+\bar{y})=\g\cos\phi$ and $\Im y = \frac{1}{2\imag}(y-\bar{y}) = g\sin\phi$ lie in
$\ker_{\Calinf}A$, hence~\eqref{eq:2nd order kernel}.
For the rest, use Proposition~\ref{prop:cos sin infinitesimal, 2}  and~${\dim_{\R}\ker_H A=0}$.
\end{proof}

\begin{remark}
Let   $A$, $g$, $\phi$   be as in Proposition~\ref{prop:2nd order kernel}. If $\phi\succ 1$, then all $y\in \ker^{\ne}_{\Calinf} A$
oscillate, and so $A$ does not split over any Hardy field extension of $H$. If  $\phi\preceq 1$, then no $y\in\ker_{\Calinf} A$ oscillates, and $A$ splits over $\Dx(H)$. 
(See Corollary~\ref{cor:char osc, 2}.)
\end{remark}

\subsection*{Liouvillian zeros of differential operators} Recall that $K=H[\imag]$. 
{\it In the next lemma we assume~${A\in K[\der]^{\ne}}$, but $A$ can have any order.}\/
%If our Hardy field~$H$ is maximal, then the differential field $K$
%is linearly closed, by Corollary~\ref{cor:maxhardymainthm, 2} below. 
%In \cite{ADHld} \marginpar{result from [9] taken on faith for now} we shall strengthen this: if $H$ is $\upo$-free and $\d$-perfect, then $K$ is linearly closed. \ In particular, if  the $\d$-perfect hull~$\Dx(H)$ of~$H$ is $\upo$-free, then~$A$ splits over
%the algebraic closure~$\Dx(H)[\imag]$ of $\Dx(H)$.
%(Section~\ref{sec:E(H) upo-free} below characterizes in terms of $H$ when $\Dx(H)$  is $\upo$-free.)
%Now if $A$ splits over $\Dx(H)[\imag]$, then there are~$g,\phi\in\Dx(H)$, $g\neq 0$, such that~$A(g\ex^{\phi\imag})=0$.
%The next lemma helps to clarify when for $H\supseteq\R$ we may take here~$g$,~$\phi$ in the Hardy subfield~$\operatorname{Li}(H)$ of~$\Dx(H)$.
The proof of this lemma relies on Kolchin's theorem (Corollary~\ref{cor:Kolchin}).

\begin{lemma}\label{lem:Liouvillian zeros}
Suppose $H\supseteq\R$.   Then the following are equivalent:
\begin{enumerate}
\item[\textup{(i)}]  $A(y)=0$ for some $y\neq 0$ in a Liouville extension of the differential field $K$;
\item[\textup{(ii)}]  $A(\ex^{f})=0$ for some $f\in\Li(H)[\imag]$  with $f'$ is algebraic over $K$;
%are $g,\phi\in\Li(H)$, $g\neq 0$, such that $g^\dagger$, $\phi'$ are algebraic over $H$ and~$A(g\ex^{\phi\imag})=0$;
\item[\textup{(iii)}]  $A(\ex^{f})=0$ for some $f\in\Li(H)[\imag]$.
\end{enumerate}
\end{lemma}
\begin{proof}
Suppose (i) holds. Then~Corollary~\ref{cor:Kolchin}
gives $y\neq 0$ in a differential field extension~$L$ of $K$   such that $A(y)=0$ and~$g:=y^\dagger$ is algebraic over~$K$. % of degree~$\leq I(r)$. 
We arrange that $L$ contains the algebraic closure $K^{\operatorname{a}}=H^{\operatorname{rc}}[\imag]$ of $K$, where $H^{\operatorname{rc}}\subseteq\Calinf$ is the real closure of the Hardy field $H$.
Thus~$g\in K^{\operatorname{a}}$, and
hence $A=B(\der-g)$ where~$B\in K^{\operatorname{a}}[\der]$ by \eqref{eq:5.1.21}.
%and $$\big[H(a,b):H\big] = \big[K(a,b):K\big]=\big[K(z,\overline{z}):K\big]\leq \big[K(z):K\big] \cdot \big[K(\overline{z}):K\big] \leq N,$$
%where for the first equality we used that $H(a,b)$, $K$ are linearly disjoint over $H$. 
Take $f\in\Li(H)[\imag]$ with~$f'=g$ and set~$z:=\ex^f\in\Calinf[\imag]^\times$. Then~$z^\dagger=g$ and thus
$A(z)=0$.
This shows~(i)~$\Rightarrow$~(ii), and
(ii)~$\Rightarrow$~(iii) is trivial. 

% --- From an old version:
%To prove~(iii)~$\Rightarrow$~(i), let $f$ be as in~(iii) and~$y:=\ex^{f}\in \Calinf[\imag]^\times$. By Lemma~\ref{lem:10.6.6} the differential field $L :=\operatorname{Li}(H)[\imag] \subseteq  \Calinf[\imag]$ is a Liouville extension of $K=H[\imag]$.  Now~$L[y]\subseteq \Univ_L:=L\big[\ex^{\Li(H)\imag}\big]  \subseteq   \Calinf[\imag]$ and~$y^\dagger=f'\in L$, so the differential fraction field $L(y)$ of $L$ is a Liouville extension of $L$, and hence of~$K$.

To prove~(iii)~$\Rightarrow$~(i),
let $f$ be as in~(iii) and~$y:=\ex^{f}\in \Calinf[\imag]^\times$.
By Lemma~\ref{lem:10.6.6} the differential field
$L :=\operatorname{Li}(H)[\imag] \subseteq  \Calinf[\imag]$ is a Liouville extension of $K=H[\imag]$. 
Now~$L[y]\subseteq \Calinf[\imag]$ and~$y^\dagger=f'\in L$. If $z\in L^{\ne}$ and
$z^\dagger=f'$, then $y\in \C z$, so~$L[y]=L$; if $f'\notin L^{\dagger}$, then also $mf'\notin L^\dagger$, for all $m\ge 1$, so $L[y]$ is an integral domain  
by~[ADH, 4.6.11]. In either case, $L[y]$ has a 
differential fraction field $L(y)$ which is a Liouville extension of $L$, and hence of~$K$.
\end{proof}

%\begin{remark} 
%Let  $I(r)\in\N$ be as in \marginpar{remark skipped}
%\cite[Proposition~4.18]{vdPS}. Using \cite[Proposition~4.19]{vdPS} one can   show that if one of the equivalent conditions in
%Lemma~\ref{lem:Liouvillian zeros} holds, then
%the germ $f$  in (ii) can be taken so that in addition~$\big[K(f'):K\big]\leq I(r)$.
%\end{remark}

%\begin{cor}\label{cor:Liouvillian zeros}
%Suppose $H\supseteq\R$ is real closed.
%If   $A(y)=0$ for some $y\neq 0$ in a Liouville extension of $K$,
%then $A(\ex^{f})=0$ for some $f \in \Li(H)[\imag]$ with~$f' \in K$. \end{cor}

\subsection*{Differential-algebraic   Hardy field extensions}
The following theorem from ~\cite{ADH2} is a fundamental result about Hardy field extensions:

\begin{theorem}\label{thm:maxhardymainthm}
Every Hardy field has a $\d$-algebraic Hardy field extension that contains $\R$ and is  Liouville closed, $\upo$-free, and newtonian.
\end{theorem}

\noindent
Here are the  consequences of this theorem which are most relevant  for us later:

\begin{cor}\label{cor:maxhardymainthm, 1}
Suppose $H$ is maximal. Then $H\supseteq\R$ is is Liouville closed, $\upo$-free, and newtonian $($and thus Schwarz closed$)$. 
\end{cor}
 
 \noindent
 Here ``Schwarz closed" follows from Proposition~\ref{prop:14.2.16+18} and the remarks before it.
Corollary~\ref{cor:perfect Schwarz closed, 2} below yields ``Schwarz closed" under a weaker assumption. 
%nessimproves on Corollary~\ref{cor:maxhardymainthm, 1}. 
%By Corollaries~\ref{cor:Schwarz closed} and~\ref{cor:maxhardymainthm, 1}, if $H$ is maximal, then each
%$A\in H[\der]$ of order $2$ splits over $K$; in fact:

\begin{cor}\label{cor:maxhardymainthm, 2}
Suppose $H$ is maximal. Then   $K$ is linearly surjective and   linearly closed. Moreover,  $H$ is linearly surjective, and
each monic $A\in H[\der]^{\ne}$ is a product of monic order $1$ and order $2$ operators in $H[\der]$.
\end{cor}
\begin{proof}
By Theorem~\ref{thm:maxhardymainthm} and the remark after Proposition~\ref{prop:newt alg ext}, the algebraic closure $K$ of $H$ is $\upo$-free and newtonian,
and thus $K$ is linearly surjective and linearly closed, by Corollary~\ref{cor:weakly d-closed}. Since $K$ is   linearly surjective, so is $H$.
The rest follows from linear closedness of $K$ and [ADH, 5.1.35].
\end{proof}

\noindent
In particular, each  inhomogeneous linear differential equation
$$y^{(n)}+a_{1}y^{(n-1)}+\cdots+a_n y\ =\  b\qquad (a_1,\dots,a_{n},b\in K)$$
has a solution $y=f+g\imag$ where $f$, $g$ lie  in some Hardy field extension  of $H$.  
%Boshernitzan~\cite[remark on p.~117]{Boshernitzan87} states a consequence of this for  inhomogeneous linear differential equations over $H$: if $a_1$, $a_2$, $b$ are elements of $H$, then some~$y$ in a Hardy field extension of $H$ satisfies~$y''+a_1y'+a_2y=b$.
In the next sections we are going to  apply the results above to homogeneous linear differential equations of order~$2$ over $H$.

\section{Second-Order Linear Differential Equations over Hardy Fields}\label{sec:perfect applications}

\noindent
In this section we analyze the oscillating solutions of second-order linear differential equations over Hardy fields in more detail.
In particular, we prove  Theorems~A and~B and Corollaries~1, 2, and~3  from the Introduction:
see the remarks after   Lemma~\ref{lem:g,phi unique} and Corollaries~\ref{cor:2nd order, f succ 1/x^2}, \ref{naicor},
and~\ref{cor:phi' quadratic}. 
%(Theorem~B is also connected to the $\upo$-freeness of the perfect hull of a Hardy field, which is characterized in Section~\ref{sec:E(H) upo-free}.) % We finish  with some remarks and conjectures about firm and flabby dents. 
{\it Throughout this section~$H$ is a Hardy field and $K:=H[\imag]\subseteq \Calinf[\imag]$.}\/

\subsection*{Parametrizing the solution space}
Let $a,b\in H$. We now continue the study of  the linear differential equation
\begin{equation}\label{eq:2nd order, app}\tag{$\tilde{\operatorname{L}}$}
Y''+aY'+bY\ =\ 0
\end{equation}
over $H$ from Section~\ref{sec:prelims}  (with slightly changed notation), and focus on the oscillating case viewed in the light of our main  result from \cite{ADH2} stated at the end of the last section as Theorem~\ref{thm:maxhardymainthm}. (Corollaries~\ref{cor:char osc} and~\ref{cor:char osc, 2}
  dealt with the non-oscillating case, and this didn't need \cite{ADH2}.)
%Most of the following theorem was claimed without proof by Boshernitzan~\cite[Theorem~5.4]{Boshernitzan87}:

\begin{theorem} \label{thm:Bosh}  
Suppose  \eqref{eq:2nd order, app} has an oscillating solution \textup{(}in $\Calinf$\textup{)}. Then there are $H$-hardian germs
$g>0$, $\phi>\R$  such that for all~$y\in\Calinf$,
$$\text{$y$ is a solution of    \eqref{eq:2nd order, app}} \quad\Longleftrightarrow\quad  \text{$y=cg\cos(\phi+d)$ for some $c,d\in\R$.}$$
Any such $H$-hardian germs~$g$,~$\phi$ are $\d$-algebraic over $H$ and lie in a common Hardy field extension of $H$. 
If~$\Dx(H)$ is $\upo$-free, then these properties force $g,\phi\in\Dx(H)$, and determine~$g$   uniquely  up to multiplying by a positive real constant and $\phi$  uniquely up to  
adding a real constant.
\end{theorem}

{\sloppy
\begin{remarks}
In Section~\ref{sec:E(H) upo-free} we  characterize $\Dx(H)$ being  $\upo$-free in terms of $H$, and show that
in the last sentence of Theorem~\ref{thm:Bosh}  we cannot drop the
condition that~$\Dx(H)$ is $\upo$-free: Theorem~\ref{thm:upo-freeness of the perfect hull}  and Remark~\ref{rem:non-uniqueness}.
We already note here that if
$H$ is not $\upl$-free or~$\overline{\omega}(H)=H\setminus\sigma\big(\Upg(H)\big){}^\uparrow$,
then $\Dx(H)$ is $\upo$-free, by Lemma~\ref{lem:7.11}. 
Recall  that $\upl$-freeness
includes having asymptotic integration.
If $H$ is $\upo$-free, then~$\overline{\omega}(H)={H\setminus\sigma\big(\Upg(H)\big){}^\uparrow}$ by Lemma~\ref{omuplosc},
hence $\Dx(H)$ is $\upo$-free;
likewise,  if~$H$ is
$\upl$-free and~$\bar{\omega}(H)\not\subseteq\omega\big(\Upl(H)\big){}^\downarrow$, then $\Dx(H)$ is $\upo$-free by Lemma~\ref{lem:6.19}.
\end{remarks}}

\noindent
 Let $V$ be an $\R$-linear subspace  of $\c$.  A pair  $(g,\phi)$ is said to {\bf parametrize $V$} if  
\begin{align*} 
 g\in \c^\times,\ g>_{\ev}0, \quad \phi\in \c,\ \phi>_{\ev}  \R,\qquad &
V\ =\  \big\{ cg\cos(\phi+d): c,d\in\R \big\};\\ 
\intertext{equivalently,  by Corol\-lary~\ref{cor:sinusoids}:}
g\in \c^\times,\ g>_{\ev}0, \quad \phi\in \c,\ \phi>_{\ev}  \R,\qquad &
V\ =\   \R g\cos\phi+\R g\sin\phi.
\end{align*}
If $(g,\phi)$ parametrizes~$V$, then so does $(cg,\phi+d)$ for any~$c\in\R^>$,~$d\in\R$.

\begin{example} 
 For $f\in \R^>$    the pair~$(1,\frac{\sqrt{f}}{2} x)$ para\-me\-tri\-zes~$\ker_{\Calinf} (4\der^2+f)$, as is clear from
Example~\ref{ex:harmonic osc}.
\end{example}

%\medskip
%\noindent
%(the commented out material that follows has been checked but is no longer needed)
%For later use we record a consequence of Theorem~\ref{thm:lindiff d-max}.
%(Recall from Section~\ref{sec:splitting}  that $\Sigma(A)\subseteq K/K^\dagger$ denotes the set of eigenvalues of  $A\in %K[\der]^{\neq}$.)
 
% \begin{cor}\label{paralpha} If $H$ is maximal and $(g,\phi)\in H^2$ parametrizes $\ker_{\Calinf} (\der^2+a\der +b)$,
% $then $\Sigma(\der^2+a\der + b)=\{-\alpha,\alpha\}$ for $\alpha:= \phi'\imag+K^\dagger$, and $\phi'\imag\notin K^\dagger$.  
 %\end{cor}
 
\noindent
For   use in Section~\ref{sec:zeros} we record the next lemma, where $V$ is an $\R$-linear subspace of~$\c^1$ and~$V':=\{y':y\in V\}$ (an $\R$-linear subspace of $\c$).

\begin{lemma}\label{lem:param V'}
Suppose $H\supseteq\R$ is real closed and closed under integration, and the pair~$(g,\phi)\in H\times H$ parametrizes $V$. Set 
$$q:=\sqrt{(g')^2+(g\phi')^2},\qquad u:=\arccos(g'/q).$$ Then $q, u\in H$ and $(q, \phi+u )\in H\times H$ parametrizes $V'$.
\end{lemma}
\begin{proof}
Note that $u$  is as in Corollary~\ref{arccosH} with $g'$, $-g\phi'$ in place of~$g$,~$h$.
Let~$y\in V$, so $y=cg\cos(\phi+d)$ where $c,d\in\R$. Then
$$y'\ =\ cg'\cos(\phi+d)-cg\phi'\sin(\phi+d)\ =\ cq\cos(\phi+u+d).$$
Conversely, for $c,d\in \R$ we have $cq\cos(\phi+u+d)=y'$ for $y=cg\cos(\phi+d)\in V$.
\end{proof}

\noindent
In the next lemma $V:=\ker_{\Calinf} (\der^2+a\der+b)$, and $g\in \c^{\times}$, $g>_{\ev} 0$, and~$\phi \in \c$, $\phi>_{\ev} \R$. Then~$(g,\phi)$
parametrizes $V$  iff $g\ex^{\phi\imag}\in\ker_{\Calinf[\imag]}(\der^2+a\der+b)$. Moreover:

\begin{lemma}\label{parphi} Set $f:= -2a'-a^2+4b$. Let $h$ be an $H$-hardian germ such that~$h>0$ and 
$h^\dagger=-\frac{1}{2}a$. 
%Let $g\in \c^\times$, $g>0$ and $\phi\in \c$, $\phi>\R$. 
Then: 
\begin{enumerate}
\item[\rm(i)] $(g,\phi)$ parametrizes $\ker_{\Calinf}(4\der^2+f)$ iff $(gh,\phi)$ parametrizes $V$.
\end{enumerate}
Assume also that $\phi$ is hardian \textup{(}so $\phi' $ is hardian with  $\phi'>0$\textup{)}. Then: \begin{enumerate}
\item[\rm(ii)]  $(1/\sqrt{\phi'}, \phi)$ parametrizes $\ker_{\Calinf}\!\big(4\der^2+\sigma(2\phi')\big)$.
\end{enumerate}
\end{lemma} 
\begin{proof}  The arguments leading up to Corollary~\ref{cor:char osc}   yield (i). 
As to (ii), by the definition of $\sigma$ (see Section~\ref{sec:diffalg}):
$$ \sigma(2\phi')\, =\, \omega\big({-(2\phi')^\dagger} + 2\phi'\imag\big)\, =\,   \omega\big({-\phi'^\dagger}+2\phi'\imag\big)\, =\, 
\omega(2y^\dagger)\quad\text{ where $y:= (1/\sqrt{\phi'})\ex^{\phi\imag}$.}$$
Hence $A(y)=0$ for $A=4\der^2+\sigma(2\phi')$; thus $\big(1/\sqrt{\phi'}, \phi\big)$ parametrizes $\ker_{\Calinf} A$.
\end{proof} 

\noindent 
Item (i) in Lemma~\ref{parphi} reduces the proof of Theorem~\ref{thm:Bosh} to the case $a=0$, and~(ii) is about
that case.

\medskip
\noindent
Let  $A=4\der^2+f\in H[\der]$  where~$f\notin\bar{\omega}(H)$, and set $V:=\ker_{\Calinf} A$. If $H$ is $\upo$-free, then~$f\in \sigma\big(\Upg(H)\big){}^\uparrow$, by Lemma~\ref{omuplosc}. 
Theorem~\ref{thm:Bosh} now follows from Lemmas~\ref{lem:parametrization of ker A}, \ref{lem:asymptotics of phi'}, and \ref{lem:g,phi unique} below, which give more information.

\begin{lemma}\label{lem:parametrization of ker A}
There is a pair of  $H$-hardian germs para\-me\-tri\-zing $V$. For any such pair $(g,\phi)$ we have~$\sigma(2\phi')=f$ and $g^2\phi'\in \R^>$,  so~$g$,~$\phi$ are $\d$-algebraic over $\Q\langle f\rangle$ and lie in a common Hardy field extension of $H$.
If $f\in\Ginf$, then each pair of  $H$-hardian germs  para\-me\-tri\-zing $V$ is in $(\Ginf)^2$;
likewise with $\Gom$ in place of $\Ginf$.
\end{lemma}

\begin{proof}
For the first claim we arrange that $H$ is maximal.
Then~$A$ splits over $K$ by Corollary~\ref{cor:maxhardymainthm, 2}.  Since $f/4$ generates oscillation,
$A$ does not split over $H$. Moreover,
$\I(K)\subseteq K^\dagger$ by Corollary~\ref{cor:cos sin infinitesimal}. By Proposition~\ref{prop:2nd order kernel} this proves the first claim.  Next, let~$(g,\phi)$ be a pair of  $H$-hardian germs parametrizing~$V$. 
Set~$y:=g\ex^{\phi\imag}\in\Calinf[\imag]^\times$; then~$A(y)=0$ 
and thus~$\omega(2y^\dagger)=f$ where~$y^\dagger=g^\dagger+\phi'\imag\in \Calinf[\imag]$. 
For $p,q\in \c^1$ we have~$\omega({p+q\imag})=\omega(p)+q^2-2(pq+q')\imag$, so $$\omega(p+q\imag)\in \c\ 
\Leftrightarrow\ pq+q'=0.$$ Therefore~$2g^\dagger=-(2\phi')^\dagger=-(\phi')^\dagger$ and so $g^2\phi'\in \R^>$,
and~$\sigma(2\phi')=f$. %~[ADH, p.~262]. 
If $f\in\Ginf$, then $y\in\Ginf[\imag]$, so $g^2=\abs{y}^2\in\Ginf$ and hence also $\phi\in\Ginf$ since~$g^2\phi'\in \R^>$;
likewise with $\Gom$ in place of $\Ginf$.
 \end{proof}

\begin{lemma}\label{lem:asymptotics of phi'}
Suppose that $H\supseteq\R$ is real closed with asymptotic integration,  
and that~$f\in\sigma\big(\Upg(H)\big){}^\uparrow$.
Then there is an active $e>0$ in $H$ such that $\phi'\sim e$ for all pairs~$(g,\phi)$ of $H$-hardian germs  parametrizing $V$.
\end{lemma}
\begin{proof} Take a  logarithmic sequence $(\ell_\rho)$ for $H$ and set $\upg_\rho:=\ell_\rho^\dagger$; see [ADH, 11.5]
about this. Then $(\upg_{\rho})$ is strictly increasing and coinitial in $\Upg(H)$ [ADH, p. 528]. 
Take $\rho$ so that $f>\sigma(\upg_\rho)$.  As in the proof of [ADH, 14.2.18]   increase $\rho$ so that~${f-\sigma(\upg_\rho)\succ\upg_\rho^2}$ and take  $e\in H^>$ with $e^2=f-\sigma(\upg_\rho)$. 
Then~$e\succ\upg_\rho$ and so~${e\in \Upg(H)^\uparrow}$.
Let~$(g,\phi)$ be a pair of elements in a Hardy field    $E\supseteq H$ parametrizing~$V$.
We claim that  $\phi'\sim e/2$ (so $e/2$ in place of $e$ has the property desired in the lemma). Arrange that $E$ is  maximal. Then~$E$ is Liouville closed and~$\phi>\R$, so~$e,2\phi'\in\Upg(E)$ by \eqref{eq:11.8.19}.
Moreover,~$E$ is newtonian by  Corollary~\ref{cor:maxhardymainthm, 1}, so  Lemma~\ref{lem:ADH 14.2.18} yields~$u \sim 1$ in~$E$ with~$\sigma(eu)=f$. The map~$y\mapsto \sigma(y)\colon\Upg(E)\to E$ is strictly increasing by
Lemma~\ref{lem:11.8.29}, so~$2\phi'=eu$ by Lemma~\ref{lem:parametrization of ker A},    thus~$\phi'\sim e/2$.
\end{proof}

\begin{lemma}\label{lem:g,phi unique}
Suppose $\Dx(H)$ is $\upo$-free or $f\in\sigma\big(\Upg(H)\big){}^\uparrow$.
Let $H_i$ be a Hardy field extension of $H$ with  $(g_i,\phi_i)\in H_i\times H_i$ parametrizing $V$, for $i=1,2$. 
Then
$${g_1/g_2}\in\R^>, \qquad {\phi_1-\phi_2}\in\R.$$
Thus $g,\phi\in \Dx(H)$ for any pair $(g,\phi)$ of  $H$-hardian germs parametrizing $V$.  
\end{lemma}
\begin{proof}
We arrange that $H_1$, $H_2$ are maximal and thus contain $\Dx(H)$. Replacing~$H$ by $\Dx(H)$  we further arrange that $H$ is $\d$-perfect and  $f\in \sigma\big(\Upg(H)\big){}^\uparrow$.
Then $\phi_1'\sim\phi_2'$ by Lemma~\ref{lem:asymptotics of phi'},
and for $i=1,2$ we have $c_i\in\R^>$ with~$\phi_i' = c_i/g_i^2$, by Lemma~\ref{lem:parametrization of ker A}. 
Replacing~$g_i$ by $g_i/\sqrt{c_i}$ we arrange $c_i=1$ ($i=1,2$), so~$g_1\sim g_2$.
Consider now the elements $g_1\cos\phi_1$, $g_1\sin\phi_1$ of~$V$;
take~$a,b,c,d\in\R$ such that
$$g_1\cos \phi_1\  =\  ag_2\cos(\phi_2+b),\qquad g_1\sin \phi_1\  =\  cg_2\cos(\phi_2+d).$$
Then  
\begin{equation}\label{eq:g1g2}
g_1^2\ =\ g_1^2(\cos^2 \phi_1 + \sin^2 \phi_1)\  =\  g_2^2\big( a^2\cos^2(\phi_2+b) + c^2\cos^2(\phi_2+d) \big),
\end{equation}
and hence
$$ a^2\cos^2(\phi_2+b) + c^2\cos^2(\phi_2+d) \ \sim\  1.$$
Thus the $2\pi$-periodic function
$$t\mapsto F(t)\ :=\ a^2\cos^2(t+b)+c^2\cos^2(t+d)\ \colon\ \R\to\R$$ 
satisfies $F(t)\to 1$ as $t\to+\infty$, hence $F(t)=1$ for all $t$, so $g_1=g_2$ by \eqref{eq:g1g2}.
It follows that~$\phi_1'=\phi_2'$, so~$\phi_1-\phi_2\in\R$.

For the final claim, let $(g,\phi)$ be a pair of $H$-hardian germs parametrizing $V$. Let~$M$ be any maximal extension of $H$.   Lemma~\ref{lem:parametrization of ker A} gives a pair~${(g_M,\phi_M)\in M^2}$ that also parametrizes $V$. By the above, $g/g_M\in \R^{>}$ and $\phi-\phi_M\in \R$, therefore~${g, \phi\in M}$. Since $M$ is arbitrary and $g$, $\phi$ are $\d$-algebraic over $H$, this gives $g,\phi\in \Dx(H)$. 
\end{proof}

\noindent
This finishes the proof of Theorem~\ref{thm:Bosh}. With
Lemma~\ref{lem:parametrization of ker A} it yields Theorem~A from the introduction.
Theorem~B follows from the remarks after Theorem~\ref{thm:Bosh}:
condition (i) in Theorem~B is equivalent to 
$\overline{\omega}(H)=H\setminus\sigma\big(\Upg(H)\big){}^\uparrow$
and (ii) to $H$ not being $\upl$-free (by Lemma~\ref{lem:upl-free} and preceding remark).

\medskip
\noindent
We note a consequence of Lemmas~\ref{lem:parametrization of ker A} and~\ref{lem:g,phi unique}:

\begin{cor}\label{cor:perfect Schwarz closed, 1}
Suppose that $H$ is $\d$-perfect. Then $\omega(H)=\bar{\omega}(H)$ is downward closed and $\sigma\big(\Upg(H)\big)$ is upward closed.
\end{cor}
\begin{proof}
By Corollary~\ref{cor:omega(H) downward closed}, $\omega(H)=\bar{\omega}(H)$ is downward closed. 
Let $f\in\sigma\big(\Upg(H)\big){}^\uparrow$. The last part of Lemma~\ref{lem:g,phi unique} gives $g,\phi\in H$ such that $(g,\phi)$ parametrizes $V$. Then~$\sigma(2\phi')=f$
by Lemma~\ref{lem:parametrization of ker A}. Now $2\phi'\in \Upg(H)$ by \eqref{eq:11.8.19}, so $f$ lies in
$\sigma\big(\Upg(H)\big)$. Thus $\sigma\big(\Upg(H)\big)$ is upward closed.
%this yields  and~\ref{lem:g,phi unique}, $\sigma\big(\Upg(H)\big)$ is upward closed.
\end{proof}

\noindent
This yields a characterization of the $\d$-perfect Hardy fields that are Schwarz closed (as defined after Corollary~\ref{cor:Schwarz closed}):

\begin{cor}  \label{cor:perfect Schwarz closed, 2}
Suppose $H$ is $\d$-perfect. Then the following are equivalent:
\begin{enumerate}
\item[\textup{(i)}] $H$ is Schwarz closed;
\item[\textup{(ii)}] $H$ is $\upo$-free;
\item[\textup{(iii)}] for all $f\in H$ the operator $4\der^2+f\in H[\der]$ splits over $K$;
\item[\textup{(iii)}] for all $a,b\in H$ the operator $\der^2+a\der+b\in H[\der]$ splits over $K$.
\end{enumerate}
\end{cor}
\begin{proof}
The equivalence (iii)~$\Leftrightarrow$~(iv) holds by Corollary~\ref{cor:char osc, 3}. The two
 equi\-va\-len\-ces~(i)~$\Leftrightarrow$~(ii)~$\Leftrightarrow$~(iii) follow from Corollaries~\ref{cor:Schwarz closed} and~\ref{cor:perfect Schwarz closed, 1}.
\end{proof}

\subsection*{Asymptotics of phase functions}
{\em In  the rest of this section
$$A=\der^2+a\der+b\ (a,b\in H),\quad  V:=\ker_{\Calinf} A, \quad
f:=-2a'-a^2+4b,$$ and we take $H$-hardian~$h>0$ such that $h^\dagger=-\frac{1}{2}a$}. Note the role of Lem\-ma~\ref{parphi}(i) in this situation. We study further properties of the pairs of   $H$-hardian germs   parametrizing~$V$.
Here is a sufficient condition for $f/4$ to generate oscillation:

\begin{lemma}\label{lem:2nd order, f succ 1/x^2}
Suppose $f>0$, $f\succ x^{-2}$. Then $f\notin\overline\omega(H)$, 
there is a pair  of $H$-hardian germs parametrizing $V$, and
for 
each such pair $(g,\phi)$, we have  $g=ch/\sqrt{\phi'}$ for some $c\in\R^>$ and $\phi'\sim \frac{1}{2}\sqrt{f}$.
\end{lemma}
\begin{proof}
Replacing $H$ by $\Li\!\big(H(\R)\big)$  and using Lemma~\ref{parphi}(i) and Lemma~\ref{lem:parametrization of ker A} we arrange  $H\supseteq\R$ to be Liouville closed, and $A=\der^2+f/4$, so ${a=0}$, ${b=f/4}$.
Put~$\ell_1:=\log x\in H$, 
$\upg_1:=\ell_1^\dagger\in\Upg(H)$,  
  $\upo_1:=\omega(-\upg_1^\dagger)$.
Then~$\sigma(\upg_1)=\upo_1+\upg_1^2 \sim x^{-2}$, so 
$f>\sigma(\upg_1)$. Hence
 $f\in \sigma\big(\Upg(H)\big){}^\uparrow$,
so ${f\notin \overline\omega(H)}$ by remarks before Lem\-ma~\ref{lem:6.18}.
Lemma~\ref{lem:parametrization of ker A} gives a pair  of $H$-hardian germs
parametrizing~$V$. That lemma says that for any such pair $(g,\phi)$ we have
$g^2\phi'\in\R^>$;   now~$\upg_1$ is active in~$H$ with~$f>\sigma(\upg_1)$ and~$f-\sigma(\upg_1)\sim f\succ \upg_1^2$, so $\phi'\sim \frac{1}{2}\sqrt{f}$ as in
the proof of Lem\-ma~\ref{lem:asymptotics of phi'}.
\end{proof}

%\begin{lemma}\label{lem:2nd order, f succ 1/x^2}
%Suppose $f>0$, $f\succ 1/x^2$. Then $f\notin\overline\omega(H)$, and  for some $H$-hardian germ $\phi$ with $\phi'\sim \frac{1}{2}\sqrt{f}$, and~$g:=1/\sqrt{\phi'}$ we have: $(gh,\phi)$ pa\-ra\-me\-tri\-zes~$V$. \marginpar{does $\phi'\sim \frac{1}{2}\sqrt{f}$ hold for every $H$-hardian phase function $\phi$??}
%\end{lemma}
%\begin{proof}  
%Using Theorem~\ref{thm:maxhardymainthm} we arrange that $H\supseteq\R$ is Liouville closed and $\upo$-free.  Choose a logarithmic sequence $(\ell_\rho)$ for $H$ with $\ell_0=x$, and define the sequences~$(\upg_\rho)$,~$(\upl_\rho)$,~$(\upo_\rho)$ as explained after Proposition~\ref{prop:14.2.16+18}. Then $\upg_0=\upl_0=1/x$ and $\upo_0=\upg_0^2=1/x^2$, and~$\upo_\rho-\upo_0 \sim \upg_{1}^2$ for~$\rho>0$, by [ADH, 11.7.1]. Thus~$\upo_\rho\sim 1/x^2$ for all~$\rho$ and hence~${f/4>\upo_\rho}$ for all~$\rho$, so $f/4$ generates oscillation by  Lemma~\ref{omuplosc}, and~${f\notin \overline\omega(H)}$, $f\in \sigma\big(\Upg(H)\big){}^\uparrow$. Lemma~\ref{lem:parametrization of ker A} gives a pair $(g,\phi)$ parametrizing $\ker_{\Calinf} (4\der^2+f)$ with $H$-hardian $\phi$ and~$g:=1/\sqrt{ \phi'}$. Now~$\upg:=1/x$ is active in~$H$ with $\sigma(\upg)=2\upg^2$ and so $f>\sigma(\upg)$ and  $f-\sigma(\upg)\sim f$.  Then $\phi'\sim \frac{1}{2}\sqrt{f}$  by the proof of Lemma~\ref{lem:asymptotics of phi'}, so~$\phi$ has the property stated in Lemma~\ref{lem:2nd order, f succ 1/x^2}. 
%\end{proof}

\noindent
Here is a special case:

\begin{cor}\label{cor:2nd order, f succ 1/x^2}
If $f\sim cx^{-2+r}$ \textup{(}$c,r\in\R^>$\textup{)}, then $f\notin\overline\omega(H)$,  
there is a pair  of $H$-hardian germs parametrizing $V$, and
for 
each such pair $(g,\phi)$, we have   $\phi\sim \frac{\sqrt{c}}{r}x^{r/2}$. 
\end{cor}

\begin{remark} Let $f,g,\phi$ be as in Lemma~\ref{lem:2nd order, f succ 1/x^2}. The proof of that lemma gives 
$f\in\sigma\big(\Upg(L)\big)$ for $L=\Li\big(H(\R)\big)\subseteq \Dx(H)$, so by Lemma~\ref{lem:g,phi unique} applied to $L$ we have $g,\phi\in \Dx(H)$, $g$ is unique up to multiplication by a positive real constant  and $\phi$ is unique up to addition of a real constant.
This is also the case for $f$, $g$, $\phi$ as in Corollary~\ref{cor:2nd order, f succ 1/x^2}. Lemma~\ref{lem:2nd order, f succ 1/x^2}, 
Corollary~\ref{cor:2nd order, f succ 1/x^2}, and this remark yield Corollary~2 from the introduction. \end{remark}

{\sloppy
\begin{lemma}\label{lem:2nd order, phi prec x}
Suppose $f\notin\bar{\omega}(H)$, and
let $(g,\phi)$ be a pair of $H$-hardian germs parametrizing $V$.
Then~${\phi\prec x}$ iff $f\prec 1$, and the same with $\preceq$ in place of $\prec$.
Also, if $f\sim c\in\R^>$, then $\phi\sim \frac{\sqrt{c}}{2}x$ and $(f<c\Rightarrow
\phi''>0)$, $(f>c\Rightarrow\phi''<0)$.
\end{lemma}}
\begin{proof}
Arrange $H\supseteq\R$ to be  Liouville closed and $g,\phi\in H$.
Then~$y:=2\phi'\in\Upg(H)$ by \eqref{eq:11.8.19}. Lemma~\ref{lem:parametrization of ker A} gives $\sigma(y)=f$; also
$\sigma(c)=c^2$ for all~$c\in\R^>$.
As the restriction of $\sigma$ to $\Upg(H)$ is strictly increasing by Lemma~\ref{lem:11.8.29},  this yields the first part. 
Now suppose  $f\sim c\in\R^>$. Take $\lambda\in\R^>$ with $\phi\sim\lambda x$.
Then $y \sim 2\lambda$, and with~$z:=-y^\dagger\prec 1$ we have
$f=\sigma(y)=\omega(z)+y^2\sim 4\lambda^2$.
Hence~$\lambda=\sqrt{c}/2$.
Suppose~$f<c$;
then~$f=\sigma(y)<c=\sigma(\sqrt{c})$ yields $y<\sqrt{c}$, so~$\phi'<\lambda$.
With~$g:=\phi-\lambda x$ we have~$g\prec x$, $g'\prec 1$ 
 and   $g'=\phi'-\lambda<0$, so~$g''=\phi''>0$. The case~$f>c$ is similar.
\end{proof}

\noindent
Combining Lemmas~\ref{lem:parametrization of ker A} and~\ref{lem:2nd order, phi prec x} yields:

\begin{cor}\label{corcor}
Suppose $f\notin\bar{\omega}(H)$.
Then for every~$y\in V^{\neq}$ we have:
\begin{enumerate}
\item[\textup{(i)}] if $f\prec 1$, then $y\not\preccurlyeq h$ \textup{(}so $y$ is unbounded if in addition $a\leq 0$\textup{)};
\item[\textup{(ii)}] if $f\succ 1$, then $y\prec h$; and
\item[\textup{(iii)}] if $f\asymp 1$, then $y\preceq h$.
\end{enumerate}
\end{cor}

\begin{remarks} 
See \cite[Chapter~6]{Bellman} for related results in a more general setting. \marginpar{remark not checked, except what follows from Cor. 5.12} For example, if
$g\in\c$  is eventually   increasing and either $g\succ 1$,
or~$g\in\c^1$, $g\sim 1$, and $\int^{+\infty} |g'(t)| d t \preceq 1$,  then $y\preceq 1$ for all $y\in\c^2$ with $y''+gy=0$; see \S\S6,~18 in loc.~cit. 
Fatou~\cite{Fatou} claims that if $g\in\c$, $g\sim 1$, then $y\preceq 1$ for all 
$y\in\c^2$ with~$y''+gy=0$. 
By Corollary~\ref{corcor}(iii) this is true for hardian $g$, but it is false in general: counterexamples are in~\cite{Caccioppoli, Perron, Wintner}; see~\cite[Chapter~6, \S{}5]{Bellman} and \cite[3.3]{Cesari}.
\end{remarks}

\begin{cor}\label{cor:parametrization in H}
Suppose  $f\notin\bar{\omega}(H)$, and
$H\supseteq\R$ is Liouville closed. The following are equivalent:
\begin{enumerate}
\item[\textup{(i)}]  $g,\phi\in H$ for every pair
$(g,\phi)$ of $H$-hardian germs parametrizing $V$;
\item[\textup{(ii)}] there is a pair of germs in $H$  parametrizing $V$;
\item[\textup{(iii)}]   $f\in\sigma(H^\times)$.
\end{enumerate}
\end{cor}
\begin{proof}
The implications (i)~$\Rightarrow$~(ii)~$\Rightarrow$~(iii) follow from Lemma~\ref{lem:parametrization of ker A} and
the remarks preceding Lemma~\ref{parphi}.
%(and don't need the hypotheses on $H$, $K$).
Suppose $f\in\sigma(H^\times)$.
Since $f\notin\bar\omega(H)$ and~$\omega(H)^\downarrow\subseteq\bar\omega(H)$,
we have $f\notin\omega(H)^\downarrow$, so
  $f\in\sigma\big(\Upg(H)\big)$ by Lemma~\ref{lem:11.8.29}.
Also, $4\der^2+f$ splits over $K$ but not over $H$, by \eqref{eq:omega} and \eqref{eq:sigma}, and $f/4$ generates oscillation. Hence   Proposition~\ref{prop:2nd order kernel}   and the remark following it yield a pair of germs in $H$ parametrizing~$V$. Now (i) follows from Lemma~\ref{lem:g,phi unique}.
\end{proof}

%\begin{cor}\label{cor:D(H) Schwarz closed}
%Suppose $H$ does not have asymptotic integration or $H$ is $\upo$-free. Then $\Dx(H)$ is Schwarz closed.
%\end{cor}
%\begin{proof}
%By the previous corollary it is enough to show that $\Dx(H)$ is $\upo$-free. If $\Dx(H)$ is $\upo$-free, then this follows from  [ADH, 13.6.1]. Otherwise, first replace $H$ by $H(\R)$ and use [ADH, ] to arrange $H\supseteq\R$; then $\operatorname{Li}(H)$ is $\upo$-free by Lemma~\ref{}, hence $\Dx(H)=\Dx(\operatorname{Li}(H))$ is $\upo$-free by the first case.
%\end{proof}

%\noindent
%If $H$ is bounded, then $\Ex(H)=\Dx(H)$ (see the remarks following Lemma~\ref{lem:Dx Ex}).
%Hence if  $H$ is bounded, and   $H$ does not have asymptotic integration or   $H$ is $\upo$-free, then $\Ex(H)$ is Schwarz closed,
%by   Corollary~\ref{cor:D(H) Schwarz closed}. In \dots we show how to remove the hypothesis of boundedness in this statement. \marginpar{to fill in}

\noindent
The case of Theorem~\ref{thm:Bosh} where $a$, $b$ are $\d$-algebraic (over $\Q$) is important for applications. In that case the $\Psi$-set
of the Hardy subfield $H_0:=\Q\langle a,b\rangle$ of $H$ is finite (see the remarks at the beginning of Section~\ref{sec:prelims}),   so $H_0$ has no asymptotic integration. Thus the relevance of the next result:

\begin{cor}\label{naicor}
Suppose $H$ does not have asymptotic integration and $f\notin \bar{\omega}(H)$. Then there is a pair~$(g,\phi)\in \Dx(H)^2$ parametrizing $V$
such that every pair of $H$-hardian germs parametrizing $V$ equals $(cg, \phi+d)$ for some $c\in \R^{>}$ and $d\in \R$. 
Moreover, for each such pair $(g,\phi)$ we have $(g/h)^2\phi'\in\R^>$.
\end{cor}
\begin{proof} The assumption on $H$ gives that $\Dx(H)$ is $\upo$-free. Now use Theorem~\ref{thm:Bosh}, and for the ``moreover"
part, use Lemma~\ref{parphi}(i) and Lemma  \ref{lem:parametrization of ker A}.
\end{proof} 

\noindent
Suppose $a$, $b$ are $\d$-algebraic, and let $\upo_n$ be as defined before Corollary~1.
By \eqref{eq:critosc} (after  Lemma~\ref{lem:7.11}), applied to $\Li(\R\<f\>)$ in place of $H$, we have
$f\le \upo_n$ for some $n$ iff $f/4$ does not generate oscillation,
and in this case
$V\subseteq\Dx(H)$ by Corollary~\ref{cor:char osc}.
Also note that $f=4b$ whenever $a=0$. Together with Corollary~\ref{naicor} and the remarks preceding it,
this shows   Corollary~1 from the introduction.

\medskip
\noindent
{\em In the rest of this subsection   we assume $f\notin\bar{\omega}(H)$, and
  $(g,\phi)$ is a pair of $H$-hardian germs parametrizing $V$}.
Then   $\sigma(2\phi')=f$   (by~Lemma~\ref{lem:parametrization of ker A}) and  thus $$P(2\phi')\ =\ 0\quad\text{where $P(Y)\ :=\ 2YY''-3(Y')^2+Y^4-fY^2\in H\{Y\}$.}$$ 
{\sloppy Hence \cite[Theorem~5.25]{ADH5} (or Theorem~3 of Rosenlicht~\cite{Rosenlicht83})
%Theorem~\ref{thm:Ros83} 
applied to the Hardy field~$E:=H\<\phi\>=H(\phi, \phi', \phi'')$ gives, for grounded~$H$, ele\-ments~$h_0,h_1\in H^{>}$ and~$m$,~$n$ with  $h_0, h_1\succ 1$ and $m+n\leq 3$, such that
$$\log_{m+1} h_0\ \prec\ \phi\ \preceq\ \exp_n h_1.$$ 
The next two lemmas improve these bounds:}

\begin{lemma}\label{boundphi1}
Suppose  $\ell\in H^>$, $\ell\succ 1$, and $\max\Psi_{H}=v\upg$ for $\upg:=\ell^\dagger$, and 
put~$\upo:=\omega(-\upg^\dagger)$. Then either $f-\upo \succ \upg^2$ and $\phi\succ \log\ell$, or for some $c\in \R^{>}$ we have
   $f-\upo \sim c \upg^2$ and~$\phi \sim \frac{\sqrt{c}}{2} \log\ell$.
\end{lemma}
\begin{proof}
  Lemma~\ref{lem:6.18} yields
$\bar\omega(H) = \upo+\upg^2\smallo_{H}^{\downarrow}$, hence $f= {\upo+\upg^2 u}$ where~${u\succeq 1}$, $u>0$,
and so $f-\sigma(\upg)=\upg^2(u-1)$. If $u\succ 1$, then~$f-\sigma(\upg)\sim u\upg^2\succ\upg^2$, 
and   the proof of Lemma~\ref{lem:asymptotics of phi'} shows that then $\phi'\sim e/2$ where~$e^2=f-\sigma(\upg)$, so~$e\succ \upg$ and thus $\phi\succ \log\ell$. 
Now suppose $u\asymp 1$.
Take $c\in\R^>$ with $u\sim c$, and
put~$\ell_*:=\log \ell$,   $\upg_*:=\ell_*^\dagger$.
Then by  \eqref{eq:11.7.6}, 
$$f-\sigma(\upg_*)\  =\  \omega\big({-\upg^\dagger}\big)-\omega\big({-\upg_*^\dagger}\big)+u\upg^2-\upg_*^2\ \sim\  u\upg^2 \ \sim\  c\upg^2\ \succ\  \upg_*^2,$$
and arguing as in the proof of Lemma~\ref{lem:asymptotics of phi'} as before gives~$\phi \sim \frac{\sqrt{c}}{2} \log\ell$.
\end{proof}

\begin{example}
Let $H=\R\langle \ell\rangle$ where~$\ell:=\ell_n$,
and let
$f=\sigma(\upg_n)(=\upo_n+\upg_n^2)$.
Since also~$f=\sigma(2\phi')$, this gives $\phi=\frac{1}{2}\ell_{n+1}+d$  for some $d\in\R$,  so $\phi\sim\frac{1}{2}\log\ell$.
\end{example}

\begin{lemma}\label{boundphi2}
Suppose $f\in \sigma\big(\Upg(H)\big){}^\uparrow$ or $H$ is not $\upl$-free, and  
$u\in H^>$ is such that~$u\succ 1$ and  $v(u^\dagger)=\min\Psi_{H}$. Then $\phi\leq u^n$ for some $n\geq 1$.
\end{lemma}
\begin{proof} We have $H$-hardian $\phi\succ 1$,  but this is not enough to get $\theta\in H^\times$ with~$\phi\asymp \theta$. 
That is why we consider first the case that $H\supseteq\R$  is real closed with asymptotic integration, and
$f\in \sigma\big(\Upg(H)\big){}^\uparrow$.
Then Lemma~\ref{lem:asymptotics of phi'} gives $e\in H^>$  such that~$\phi'\sim e$, and as $H$ has asymptotic integration
we obtain $\theta\in H^\times$ with $\phi\asymp \theta$. Hence~${\phi^\dagger\asymp \theta^\dagger \preceq u^\dagger}$,
and   thus $\phi\leq u^n$ for some $n\geq 1$, by \eqref{eq:9.1.11}. 

We now reduce the general case to this special case. Take a maximal Hardy field extension $E$ of $H$ with $g,\phi\in E$.
Suppose $H$ is $\upl$-free; so $f\in \sigma\big(\Upg(H)\big){}^\uparrow$. Then $H(\R)$ is $\upl$-free with the same value group as $H$, by 
Lemma~\ref{lem:1.3.3+}, and by that same lemma $L:=H(\R)^{\operatorname{rc}}\subseteq E$ has asymptotic integration, with $\Psi_L=\Psi_H$, so~$v(u^\dagger)=\min\Psi_{L}$ and thus~${\phi\leq u^n}$ for some $n\geq 1$ by the special case applied to~$L$ in the role of $H$.
 
For the rest of the proof we  assume $H$ is not $\upl$-free. 
 Then $H(\R)^{\operatorname{rc}}\subseteq E$ is not $\upl$-free and
$v(u^\dagger)=\min\Psi_{H(\R)^{\operatorname{rc}}}$, by~Lemma~\ref{lem:1.3.3+}. Hence replacing~$H$ by~$H(\R)^{\operatorname{rc}}$
% and noting that~$\Gamma_H^>$ is cofinal in $\Gamma_L^>$,
we arrange that~$H\supseteq\R$ and $H$  is real closed in what follows. 
%If $H$ has asymptotic integration and $f\in \sigma\big(\Upg(H)\big){}^\uparrow$, then we are done.   

Suppose $H$ has no asymptotic integration. Lemma~\ref{lem:1.3.18}
yields an $\upo$-free Hardy subfield $L\supseteq H$ of $E$ such that $\Gamma_H$ is cofinal  in $\Gamma_L$, so~$v(u^\dagger)=\min \Psi_L$. Moreover, $f\in L\setminus\bar\omega(L)= \sigma\big(\Upg(L)\big){}^\uparrow$ by 
Lemma~\ref{omuplosc}.  
Hence replacing $H$ by~$L^{\operatorname{rc}}$ we have a reduction to the special case. 

Suppose $H$ has asymptotic integration. 
%and~$f\notin \sigma\big(\Upg(H)\big){}^\uparrow$. 
Since $H$ is not $\upl$-free, we have $\upl\in H$ such that  $\upl+g^{\dagger\dagger} \prec g^\dagger$ for all $g\succ 1$
in $H$. Take   $y\in E^\times$ with $y^\dagger=-\upl$. By Lemma~\ref{lem:11.5.14}, $vy$ is a gap in $H(y)$ 
and hence in $L:=H(y)^{\operatorname{rc}}$, and $\Gamma_H$ is cofinal in $\Gamma_L$, so $v(u^\dagger)=\min \Psi_L$. So $L$ has no asymptotic integration and thus we can replace
$H$ by $L$ to get a reduction to the ``no asymptotic integration'' case.  
\end{proof}

\noindent
By the  last two lemmas, combined with Corollary~\ref{naicor} and the remarks before it:

\begin{cor}
Suppose $a$, $b$ are $\d$-algebraic and $H=\Q\langle a,b\rangle$. Then 
$\Psi_H$ is finite and
$g,\phi\in\Dx(H)\subseteq\c^\omega$.
Moreover, if $h_0,h_1\in H^>$ with   $h_0,h_1\succ 1$ and~$v(h_0^\dagger)=\max \Psi_{H}$,
$v(h_1^\dagger)=\min \Psi_{H}$, then  $\frac{1}{n}\log h_0 \le\phi \le h_1^n$ for some $n\ge 1$.
\end{cor}

\noindent
The hypothesis of the next corollary is satisfied by $H=\R(x^{\R})\subseteq \Li(\R)$.
Recall that  $f > x^{-2}=\omega_0$, since~$f\notin\bar{\omega}(H)$ and $\omega_0\in \bar{\omega}(H)$.

\begin{cor}\label{cor:upper lower bd for phi}
Assume $x\in H$ and $v(1/x)=\max\Psi_H$.
Then $g,\phi\in\Dx(H)$.
Moreover, let $c\in\R^>$, $r\in\R^{\ge}$ be such that $f\sim cx^{-2+r}$.
If $r>0$, then $\phi\sim \frac{\sqrt{c}}{r}x^{r/2}$,
and if~$r=0$, then $c>1$ and $\phi\sim\frac{\sqrt{c-1}}{2}\log x$.
\end{cor}
\begin{proof} 
If $r>0$, then  $\phi\sim \frac{\sqrt{c}}{r}x^{r/2}$
by Lemma~\ref{lem:2nd order, f succ 1/x^2}. If $r=0$, then from Lemma~\ref{boundphi1} with $\ell_0:= x$
we obtain $c>1$,
$f-x^{-2}\sim (c-1)x^{-2}$, and~${\phi\sim\frac{\sqrt{c-1}}{2}\log x}$.
\end{proof}  

\begin{remark}  
Suppose $a=0$. Then $b=f/4$, and $g^2\phi'\in\R^>$ by Lemma~\ref{lem:parametrization of ker A}; hence bounds on~$\phi$ 
give bounds on $g$.
Thus by Corollary~\ref{cor:upper lower bd for phi}, if $b\in\R(x^{\R})$, then~$g\asymp f^{-1/4}$.
\end{remark}

\noindent
In \cite{ADHbf} we shall strengthen Corollary~\ref{cor:upper lower bd for phi}  in the case $a,b\in H=\R(x)$ by obtaining an asymptotic \marginpar{to check in [8]} 
expansion for $\phi$.  

{\sloppy
\begin{examplesNumbered}\label{ex:boundphi} Assume  $a=0$; so $f=4b$. 
\begin{enumerate}
\item  Let $c\in\R$, $c>1$, and $b=(c/4)x^{-2}\in H$. Then the standing assumption~${f\notin\bar{\omega}(H)}$ holds, since $f=cx^{-2}$. The germ $y=g\ex^{\imag \phi}$ with
$g:=x^{1/2}$ and~$\phi:=(1/2)\sqrt{c-1} \log x$ satisfies $4y''+fy=0$, so $$\big(x^{1/2},\ (1/2)\sqrt{c-1} \log x\big) \quad \text{parametrizes $V$.}$$
%$$y\ :=\ x^{1/2}\cos\left(\frac{\sqrt{c-1}}{2} \log x\right)\in\Gom$$ solves the corresponding   linear differential equation $4Y''+ fY=0$.
\item Let  $r\in\R$, $r>-1$, with $ x^r\in H$. Then $b:= \frac{1}{4}\big({x^{2r}-r(r+2)x^{-2}}\big)\in H$ and the standing assumption $f\notin \bar{\omega}(H)$ holds
in view of $f\sim x^{2r}\succ x^{-2}$. Here $y=g\ex^{\imag\phi}$ with $g:=x^{-r/2}$, $\phi:=x^{r+1}/(2r+2)$ satisfies $4y''+fy=0$, so
$$ \big(x^{-r/2},\ x^{r+1}/(2r+2)\big)\quad \text{parametrizes $V$.}$$
\end{enumerate}
\end{examplesNumbered}}

\subsection*{Liouvillian phase functions}
Suppose $f\notin\overline\omega(H)$.
Theorem~\ref{thm:Bosh} yields  a pair~$(g,\phi)$ of  $H$-hardian germs parametrizing $V$, and any such $H$-hardian germs
are $\d$-algebraic over $H$.
In this subsection we ask when~$g$,~$\phi$ can be taken   in a Liouville Hardy field extension of~$H$.  
{\samepage

\begin{prop}\label{prop:2nd order, Liouvillian solutions} 
Suppose  $H\supseteq\R$, and $H$ does not have asymptotic integration or~$H$ is $\upo$-free. Then the following are equivalent: 
\begin{enumerate}
\item[\textup{(i)}] $A(y)=0$ for some $y\neq 0$ in  a Liouville extension of $K$;
\item[\textup{(ii)}]  some pair $(g,\phi)\in \operatorname{Li}(H)^2$ with $g^\dagger$, $\phi'$ algebraic over~$H$ parametrizes $V$;
\item[\textup{(iii)}]  some pair in $\operatorname{Li}(H)^2$ parametrizes $V$;
\item[\textup{(iv)}] every pair of $H$-hardian germs pa\-ra\-me\-trizing $V$ lies in $\operatorname{Li}(H)^2$.
\end{enumerate}
\end{prop}}
\begin{proof}
Suppose (i) holds. 
Then Lemma~\ref{lem:Liouvillian zeros}  gives $g, \phi\in L:=\operatorname{Li}(H)$, $g\neq 0$, such that
$g^\dagger$, $\phi'$ are algebraic over $H$ and $A(g\ex^{\phi\imag})=0$.
Replacing~$g$ by~$-g$ if necessary we arrange $g>0$.
We have $\phi\succ 1$: otherwise $g\ex^{\phi\imag}\in E[\imag]^\times$ for some Hardy field extension $E$ of $L$, by Corollary~\ref{cor:phi preceq 1}, hence $\Re(g\ex^{\phi\imag})\in V^{\neq}$ does not oscillate, or $\Im(g\ex^{\phi\imag})\in V^{\neq}$ does not oscillate,
a contradiction. Replacing $\phi$ by~$-\phi$ if necessary we arrange $\phi>\R$. Then $(g,\phi)$ parametrizes $V$.
This yields~(ii).  The implication~(ii)~$\Rightarrow$~(iii) is trivial.
By the assumptions on $H$ and the remarks following Theorem~\ref{thm:Bosh}, $\Dx(H)$ is $\upo$-free,  so (iii)~$\Rightarrow$~(iv) follows from Theorem~\ref{thm:Bosh}. Finally, (iv)~$\Rightarrow$~(i)
follows from (iii)~$\Rightarrow$~(i) in Lemma~\ref{lem:Liouvillian zeros}.
\end{proof}

\noindent
In the rest of this subsection we let $(g,\phi)$ be a pair of $H$-hardian germs parametrizing~$V=\ker_{\Calinf}A$.
We now bring in the   operator $B:=\der^3+f\der+(f'/2)\in H[\der]$ of order $3$, already studied in Section~\ref{sec:diffalg}, and observe:

\begin{lemma}\label{lem:B, 1}
$B\big(\frac{1}{\phi'}\big)=0$.
\end{lemma}
\begin{proof}
We  arrange that $H\supseteq\R$ contains $\phi$ and is Liouville closed.
By remarks before Lemma~\ref{lem:kerB} the differential subring $K[\ex^{\phi\imag},\ex^{-\phi\imag}]$ of $\Calinf[\imag]$
is an integral domain. Let~$L$ be its differential fraction field.
Then
$$(\phi')^{-1/2}\ex^{\phi\imag}, (\phi')^{-1/2}\ex^{-\phi\imag}\in \ker_{L}4\der^2+f.$$
Thus $B(1/\phi')=0$ by Lemma~\ref{lem:Appell} applied to $L$ in the role of $K$.
\end{proof}

% --- Below is an older version of this proof.
%\begin{proof}
%We  arrange that $H\supseteq\R$ contains $\phi$ and is Liouville closed, and identify the universal  exponential extension  $\Univ=\Univ_K$ of~$K=H[\imag]$   with a differential subring of~$\Calinf[\imag]$ as explained in Section~\ref{sec:prelims}, with differential fraction field $\Omega=\Omega_K$. Then
%$$(\phi')^{-1/2}\ex^{\phi\imag}, (\phi')^{-1/2}\ex^{-\phi\imag}\in \ker_{\Univ}4\der^2+f.$$
%Thus $B(1/\phi')=0$ by Lemma~\ref{lem:Appell} applied to $\Omega$ in the role of $K$.
%\end{proof}

\begin{lemma}\label{lem:B, 2}
Let $E$ be a  Hardy field extending  $H\langle \phi'\rangle$.
Then $\ker_E B = C_E\frac{1}{\phi'}$.
\end{lemma}
\begin{proof}
By Corollary~\ref{cor:maxhardymainthm, 1} we arrange that $E\supseteq\R$ is Schwarz closed.
Then~$f\notin\bar{\omega}(H)=\omega(E)\cap H$,
so $f\in\sigma(E^\times)$, and thus $\dim_{C_E} \ker_E B=1$ by Lem\-ma~\ref{lem:kerB}.
\end{proof}

\noindent
We can now complement Proposition~\ref{prop:2nd order, Liouvillian solutions}: 

\begin{cor}\label{cor:phi' quadratic}
Suppose   $\phi'$ is algebraic over $H$. Then   $(\phi')^2\in H$ and $g^\dagger\in H$.
\end{cor}
\begin{proof}
Let $E:=H^{\operatorname{rc}}\subseteq\Calinf$, so  $E\supseteq H\langle \phi'\rangle$.
Set  $L:=E[\imag]\subseteq \Calinf[\imag]$, so $L$ is an algebraic closure of the differential field~$H$.
Put
$u:=2\phi'\in E$, and
let~$\tau\in\operatorname{Aut}(L|H)$.
Then~$B(\tau(1/u))=0$ by Lemma~\ref{lem:B, 1}. So $\Re\tau(1/u)$ and~$\Im\tau(1/u)$
in~$E$ are also zeros of~$B$, hence
Lemma~\ref{lem:B, 2} yields $c\in\C^\times$ with~$\tau(1/u)=c/u$ and thus~$\tau(u)=c^{-1}u$.
Now with $$P(Y)\ :=\  2YY''-3(Y')^2+Y^4-fY^2\in H\{Y\}$$ we have $P(u)=0$, so $P(\tau(u))=0$, hence
$$0\ =\ P(u)-c^2P(\tau(u))\ =\ P(u)-c^2P(c^{-1}u)\ =\ (1-c^{-2})u^4$$
and  thus $c\in\{-1,1\}$, so $\tau(u^2)=u^2$. This proves the first statement.
The second statement follows from the first and $g^2\phi'\in\R^>$ by Lemma~\ref{lem:parametrization of ker A}.
\end{proof}

\noindent
Proposition~\ref{prop:2nd order, Liouvillian solutions} and Corollary~\ref{cor:phi' quadratic} yield Corollary~3 from the introduction, which is applied  to the Bessel equation in \cite{ADHbf}. 
We note that there is an algorithm which, upon input of a rational function
$f\in\Q(x)$,
decides whether some pair of germs in $\Li\!\big(\R(x)\big)$ parametrizes
$V=\ker_{\Calinf}(4\der^2+f)$: by Proposition~\ref{prop:2nd order, Liouvillian solutions} this is the same as
deciding whether condition (i) in that proposition holds for $H=\R(x)$ and~$A=\der^2+(f/4)$, and it is well-known that
this can be done algorithmically~\cite{Kovacic}.\marginpar{Kovacic result taken on faith} 
%in connection to   results of Liouville. %\cite{Liouville41}. 

\section{Distribution of Zeros and Critical Points} \label{sec:zeros}

\noindent
{\it In this section $H$ is a Hardy field.}\/
 Consider again a differential equation
\begin{equation} 
\tag{$\tilde{\operatorname{L}}$}
Y''+aY'+bY\ =\ 0\qquad (a,b\in H).
\end{equation}
Below we use Theorem~\ref{thm:Bosh}
to show
that for any oscillating solution~$y\in\Calinf$ of \eqref{eq:2nd order, app}
the sequence of successive zeros  (and the sequence of successive critical points) of~$y$ grows very regularly, with growth comparable to that of a hardian function. 
(For the equation $Y''+fY=0$,   where $f\in\Ex(\Q)$ generates oscillation, this was suggested
after~\cite[\S{}20, Conjecture~4]{Boshernitzan82}.)
To make this precise we first  
define a preordering~$\leq$ on the set  $\R^{\N}$ of sequences of real numbers  by
$$(s_n)\leq (t_n)\quad:\Longleftrightarrow\quad s_n \leq t_n \text{ eventually} 
\quad:\Longleftrightarrow\quad \exists m\,\forall n\geq m\  s_n\leq t_n.$$
(A preordering on a set is a reflexive and transitive  binary relation on that set.)
We say that $(s_n),(t_n)\in\R^{\N}$ are {\bf comparable} if $(s_n) \leq (t_n)$ or $(t_n) \leq (s_n)$. 
The induced equivalence relation $=_{\ex}$ on $\R^{\N}$ is that of having the same tail: 
$$(s_n) =_{\ex}(t_n) \quad:\Longleftrightarrow\quad (s_n) \leq (t_n) \text{ and }(t_n)\leq (s_n)\ \Longleftrightarrow\ s_n=t_n \text{ eventually}. $$
For any germ $f\in \c$ we take a representative in $\c_0$, denoted here also by $f$ for convenience, and associate to this germ the tail of the sequence $\big(f(n)\big)$; note that this tail is independent of the choice of representative. 
Thus if the germs of~${f,g\in\c_0}$ are contained in a common Hardy field, then the sequences~$\big(f(n)\big)$,~$\big(g(n)\big)$ are comparable. 
%\marginpar{for ``limit point'', see 5.2.10} 

\begin{definition}
Given an infinite set $S\subseteq\R$ with a lower bound in $\R$ and without a limit point, the {\bf enumeration}\/ of $S$
is the strictly increasing sequence~$(s_n)$ with~$S=\{s_0,s_1,\dots\}$ (so $s_n\to+\infty$ as $n\to+\infty$).
 \end{definition}
 
 \noindent
We take representatives of $a$, $b$ in $\c^1_e$ with $e\in \R$, denoting these by $a$ and $b$ as well, 
and set~$f:=-2a'-a^2+4b\in \c_e$.
Let $y\in \c_e^2$ be oscillating with 
$$y''+ay'+by\ =\ 0\qquad  \text{(so the germ of $f$ does not lie in $\bar{\omega}(H)$ and is thus $>x^{-2}$).}$$
The set $y^{-1}(0)$ of zeros of $y$ is infinite and has no limit point (see Section~\ref{sec:germs}). Let~$(s_n)$ be the enumeration of $y^{-1}(0)$. Theorem~\ref{thm:Bosh} 
yields~${e_0\geq e}$, $g\in \big(\c^2_{e_0}\big)^\times$, and $\phi\in\c_{e_0}^2$ such that $\phi'(t)>0$ for all $t\in [e_0,+\infty)$,  $y|_{[e_0,+\infty)}=g\cos \phi$, and~$g$,~$\phi$ lie in a common Hardy field extension of $H$ with the pair~$(g,\phi)$ pa\-ra\-me\-tri\-zing $\ker_{\Calinf}(\der^2+a\der+b)$ (where~$g$,~$\phi$ also
denote their own germs). 
The sequence $(s_n)$ can be interpolated by a strictly increasing function with hardian germ:

 \begin{lemma}\label{lem:zeta}
There is a strictly increasing $\zeta\in\c_{n_0}$ \textup{(}$n_0\in\N$\textup{)} 
 such that  $s_n=\zeta(n)$ for all $n\geq n_0$ and the germ of $\zeta$ is hardian with $H$-hardian compositional inverse.
 \end{lemma}
\begin{proof}
Take~$n_0\in \N$ such that $s_n\geq e_0$ for all $n\geq n_0$, and then $k_0\in\frac{1}{2}+\Z$ such that~$\phi(s_n)=(k_0+n)\pi$ for all $n\geq n_0$. Thus $n_0=\big(\phi(s_{n_0})/\pi\big)-k_0$. Let
$\zeta \in\c_{n_0}$ be the compositional
inverse of~$(\phi/\pi)-k_0$ on $[s_{n_0},+\infty)$. Then $\zeta$ has the desired properties: the germ of $\zeta$ is hardian by Lemma~\ref{lem:Bosh6.5}.  
\end{proof}

%\begin{remark}[``continuous ranking'' of zeros; cf.~\cite{Muldoon,Vosmansky}]
%Let $d$ range over $\R$ and put~$y_d:=g\cos(\phi+d\pi)\in\c_{e_0}^2$, so $y_0=y|_{[e_0,+\infty)}$ and~$y_{-1/2}=g\sin\phi$. 
%With~$a$,~$b$ also denoting their restrictions to $[e_0,+\infty)$, we have~$y_d''+ay_d'+by_d = 0$. 
%Now let~${\zeta\in\c_{n_0}}$~\textup{(}${n_0\in\N}$\textup{)} be as in Lemma~\ref{lem:zeta},
%and consider the continuous function 
%$$(d,t)\mapsto\zeta_d(t):=\zeta(t-d)\,\colon\, \big\{ (d,t): t\ge d+n_0 \big\} \to \R.$$
%For every $d$ 
 %the  function $t\mapsto\zeta_d(t) \colon [d+n_0,+\infty)\to\R$ is continuous and strictly increasing 
 %with hardian germ
 %(see Example~\ref{ex:Hcircx+d}),
  %and 
  %$$y_d^{-1}(0)\cap \R^{\ge s_{n_0}} = \big\{\zeta_d(n) : n\ge d+n_0\big\}.$$
  %We don't know whether the germs of the   $\zeta_d$ are contained in a common Hardy field.
% %\marginpar{what if $a$, $b$ are in $\R(x)$?}
%\end{remark}

\noindent
If $a,b\in \Ginf$, then we can choose~$\zeta$ in Lemma~\ref{lem:zeta} such that its germ is in~$\Ginf$;
likewise with $\Gom$ in place of $\Ginf$. We do not know whether $\zeta$ in Lemma~\ref{lem:zeta} can always be chosen
to have $H$-hardian germ. 
In the next corollary we let $H_0:=\R(x)$. Since $H_0\cup\{a,b\}\subseteq \Dx(H)$, we have the Hardy subfield $H_0\langle a, b\rangle$ of $\Dx(H)$. Note that any hardian element of $\c$ is actually $H_0$-hardian. 

\begin{cor}\label{cor:zeta d-alg}
Suppose $a$, $b$ are $\d$-algebraic, and put $d:=\operatorname{trdeg}\!\big(H_0\langle a,b\rangle|H_0)\in\N$. We can choose the germ of $\zeta$ in Lemma~\ref{lem:zeta} to be $\d$-algebraic   such that
$$\operatorname{trdeg}\!\big(H_0\langle \zeta\rangle|H_0)\ \le\  d+2,$$ and then there are  $c\in\R^>$ and
$e,r,s\in\N$ such that   $r+s\le d+2$ and
$$(\log_{r} x)^c\ <\  \zeta\ <\ \exp_s (x^e),$$ so $(\log_r n)^c \le s_n < \exp_s (n^e)$ eventually. $($Here $\log_r$ indicates the $r$ times iterated real logarithm function, and likewise for $\exp_s.)$
\end{cor}
\begin{proof}
With $f$ also denoting its germ we have $f\in \R\langle a,b\rangle$ and so 
$\operatorname{trdeg}\!\big(H_1|H_0)\le d$ for $H_1:=H_0\langle f\rangle$.
We also have $P(2\phi')=0$ where 
$$P(Y)\ := \ 2YY''-3(Y')^2+Y^4-fY^2\in H_1\{Y\},$$ so $H_1\subseteq H_0\langle\phi'\rangle$ with $H_0\langle\phi'\rangle  = H_1\langle\phi'\rangle=
H_1(\phi',\phi'')$, so $H_0\<\phi\>=H_1(\phi, \phi', \phi'')$ and thus $\operatorname{trdeg}\!\big(H_0\langle\phi\rangle|H_0)\le d+3$. 
Let $\zeta$ be as in the proof of Lemma~\ref{lem:zeta}. Then by Lemma~\ref{lem:trdegellinv}, $\zeta$ is $\d$-algebraic  with~$\operatorname{trdeg}\!\big(H_0\langle \zeta\rangle|\R)=
\operatorname{trdeg}\!\big(H_0\langle \phi\rangle|\R)\le d+3$ and hence
$\operatorname{trdeg}\!\big(H_0\langle \zeta\rangle|H_0)=
\operatorname{trdeg}\!\big(H_0\langle \phi\rangle|H_0)\le d+2$.
The rest follows from \cite[Theorem~5.25, Lemma~5.26]{ADH5} applied to $H_0$, $H_0\langle\zeta\rangle$
in place of $H, E$.
\end{proof}

\noindent
For  some $\phi$ we can describe the asymptotic behavior of~$\zeta$ in terms of $\phi$:

\begin{cor}\label{asympzeta}
If $\phi\succeq x^{1/n}$ for some $n\ge 1$, then 
in Lemma~\ref{lem:zeta} one can in addition choose~$\zeta\sim \phi^{\operatorname{inv}}\circ \pi x$.
\end{cor}
\begin{proof} Let $n_0$, $k_0$, $\zeta$ be as in the proof of Lemma~\ref{lem:zeta}. Then $$\zeta^{\operatorname{inv}}\ \sim\  (\phi/\pi)-k_0\ \sim\  \phi/\pi.$$  Now assume $\phi\succeq x^{1/n}$, $n\ge 1$. Then $(\phi/\pi)^{\operatorname{inv}}\preceq x^n$.
Lemma~\ref{lem:ginv} with~$\zeta^{\operatorname{inv}}$, $(\phi/\pi)^{\operatorname{inv}}$, in the role of $g$, $h$   gives $\zeta\sim(\phi/\pi)^{\operatorname{inv}}=\phi^{\operatorname{inv}}\circ \pi x$.
\end{proof}

\noindent
Combining Lemma~\ref{lem:ginv} and Corollaries~\ref{cor:2nd order, f succ 1/x^2} and~\ref{asympzeta}, we obtain:

\begin{cor}\label{cor:zeta asymptotics}
If $f\sim cx^{-2+r}$ \textup{(}$c,r\in\R^>$\textup{)}, then $s_n\sim\displaystyle \left(\frac{r\pi n)}{\sqrt{c}}\right)^{2/r}$ as $n\to\infty$.
\end{cor}

\noindent
We also have a crude bound on the growth of $(s_n)$ when $H=\R(x^{\R})$ and $f\asymp 1/x^2$. More generally:

\begin{cor}\label{cor:zeta asymptotics, R(x)}
Suppose $x\in H$, $v(1/x)=\max\Psi_H$,  and $f\sim cx^{-2}$ \textup{(}$c\in\R^>$\textup{)}.
Then~$c>1$ and 
$\log s_n \sim \displaystyle\frac{2\pi n}{\sqrt{c-1}}$ as $n\to\infty$.
%$\ex^{d(1-\varepsilon)(n+k_0)} \le s_n \le \ex^{d(1+\varepsilon) (n+k_0)}$ eventually.
%If $f \asymp 1/x^2$, then for some $r\in\R^>$ we have~$\ex^{n/r} \leq s_n\leq \ex^{rn}$ eventually. If $f\succ 1/x^2$, then for some $m\geq 1$ and every~${\epsilon\in\R^>}$ we have  $n^{1/m}\leq s_n\leq \ex^{\epsilon n}$ eventually.
\end{cor}
\begin{proof}
By Corollary~\ref{cor:upper lower bd for phi} we have $c>1$ and
$\phi\sim d\log x$ where $d:=\frac{1}{2}\sqrt{c-1}$, hence~$\phi\circ \ex^x \sim dx$.
Take~$\zeta$  as in the proof of Lemma~\ref{lem:zeta}.
Then $\zeta^{\operatorname{inv}} \sim \phi/\pi$ and thus
$$(\log\zeta)^{\operatorname{inv}}=\zeta^{\operatorname{inv}} \circ \ex^x \sim (\phi\circ \ex^x)/\pi \sim (d/\pi)x,$$ so
Lemma~\ref{lem:ginv} with~$(\log\zeta)^{\operatorname{inv}}$,~$(\pi/d)x$ in the role of $g$, $h$   gives
$\log\zeta\sim (\pi/d)x$.
%Then for each $\varepsilon\in (0,1)$ we have $$\left(\frac{d}{\pi(1-\varepsilon)}\log x\right) - k_0 > (\phi/\pi)-k_0 >  \left(\frac{d}{\pi(1+\varepsilon)}\log x\right)-k_0,$$ hence $$\exp \left(\frac{\pi(1-\varepsilon)}{d}n\right) <s_n< \exp \left(\frac{\pi(1+\varepsilon)}{d}n\right)$$ eventually, by the proof of Lemma~\ref{lem:zeta}.
\end{proof}

\noindent
Corollaries~\ref{cor:zeta asymptotics} and~\ref{cor:zeta asymptotics, R(x)} prove   Corollary~4 in the introduction.
The next result is a version of the Sturm Convexity Theorem (\cite[p.~318]{BR}, \cite[p.~173]{Sturm})
concerning the differences between consecutive zeros of $y$:

\begin{cor}\label{cor:zeta, differences}  
If $f\prec 1$, then 
the sequence $(s_{n+1}-s_n)$ is eventually strictly increasing
with
$s_{n+1}-s_n\to +\infty$ as $n\to\infty$. If
 $f\succ  1$, then $(s_{n+1}-s_n)$ is eventually strictly decreasing with~${s_{n+1}-s_n}\to 0$  as $n\to\infty$.
Now suppose~$f\sim c$~\textup{(}$c\in\R^>$\textup{)}. Then~$s_{n+1}-s_n\to 2\pi/\sqrt{c}$ as~$n\to\infty$,
and if $f<c$, then~$(s_{n+1}-s_n)$ is eventually strictly decreasing,
if $f=c$, then~$(s_{n+1}-s_n)$ is eventually constant, and
if $f>c$, then $(s_{n+1}-s_n)$ is eventually strictly increasing.
%Moreover, if~$\phi''>0$, then the sequence $(s_{n+1}-s_n)$ is eventually strictly decreasing; if~$\phi''=0$, then $(s_{n+1}-s_n)$ is eventually constant;  and if $\phi''<0$, then~$(s_{n+1}-s_n)$ is eventually strictly increasing.
\end{cor}
\begin{proof}
Recall that $\phi\in\c_{e_0}^2$ and $\phi'(t)>0$ for all $t\ge e_0$. 
Take  $\zeta$ as in the proof of Lemma~\ref{lem:zeta}. Then~$\zeta\in\c_{n_0}^2$ with
$$\zeta'\ =\ \pi \frac{1}{\phi'\circ\zeta},\qquad \zeta''\ =\ -\pi^2\frac{\phi''\circ\zeta}{(\phi'\circ\zeta)^3}.$$
The Mean Value Theorem gives for every $n\geq n_0$ a $t_n\in (n,n+1)$ such that
$$s_{n+1}-s_n\ =\ \zeta(n+1)-\zeta(n)\ =\ \zeta'(t_n).$$
If $f\prec 1$, then Lemma~\ref{lem:2nd order, phi prec x} gives $\phi\prec x$, so $\zeta\succ x$, hence $\zeta'\succ 1$; this proves the first claim of the lemma. The other claims follow likewise using Lemma~\ref{lem:2nd order, phi prec x} and the above remarks on $\zeta'$ and $\zeta''$. 
\end{proof}

\noindent
The next two corollaries are multiplicative versions of Corollary~\ref{cor:zeta, differences}. We recall that since $(s_n)$ is strictly increasing and unbounded, we have
$s_n>0$ eventually.

\begin{cor}
Suppose $f\succ x^{-2}$. Then   $s_{n+1}/s_n \to 1$ as~$n\to\infty$.
\end{cor} 
\begin{proof} 
Let $\zeta\in \c_{n_0}^2$ and $t_n$ for $n\ge n_0$  be
as in the proof of Corollary~\ref{cor:zeta, differences}. Here we take $n_0$ in the proof of Lemma~\ref{lem:zeta} big enough to
 arrange that $\zeta(t)>0$ and~$\sign\big(\zeta''(t)\big) = \sign\big(\zeta''(n_0)\big)$
for  $t\ge n_0$.
If $\zeta''(n_0)<0$, then for $n\ge n_0$:
$$0<\frac{s_{n+1}}{s_n}-1=\frac{\zeta'(t_n)}{\zeta(n)} < \frac{\zeta'(n)}{\zeta(n)},$$
and if $\zeta''(n_0)\ge 0$, then for $n\ge n_0$: 
 $$0<1-\frac{s_{n}}{s_{n+1}}=\frac{\zeta'(t_n)}{\zeta(n+1)} \le \frac{\zeta'(n+1)}{\zeta(n+1)}.$$
By Lemma~\ref{lem:2nd order, f succ 1/x^2} we have $\phi'\succ 1/x$ and so
$\zeta^\dagger=\big(\pi/(x\phi')\big)\circ\zeta\prec 1$. 
\end{proof}

\noindent
In the next corollary we return to the setting of Corollaries~\ref{cor:upper lower bd for phi} and~\ref{cor:zeta asymptotics, R(x)}: 

\begin{cor}
Suppose $x\in H$, $v(x^{-1})=\max\Psi_H$, and $f\sim cx^{-2}$ where $c\in\R^>$.
Then $c>1$, and with $d:=\frac{1}{2}\sqrt{c-1}$ we have~$s_{n+1}/s_n\to \ex^{\pi/d}$ as~$n\to\infty$. 
\end{cor}
\begin{proof}
Corollary~\ref{cor:upper lower bd for phi} gives
$c>1$ and 
$\phi\sim d\log x$, so $\phi'\sim d/x$ by Lemma~\ref{lem:9.1.4(ii)}.
Take~$\zeta\in\c^1_{n_0}$ as in the proof of Lemma~\ref{lem:zeta} with $n_0$ so large that
$\zeta\in\c_{n_0}^\times$. 
Then~$\zeta^\dagger=\big(\pi/(x\phi')\big)\circ \zeta\sim \pi/d$.
Take~$\varepsilon\in\c_{n_0}$ such that $\zeta^\dagger= (\pi/d)(1+\varepsilon)$, so $\varepsilon\prec 1$.
The Mean Value Theorem yields for each~$t\ge n_0$
a $\xi(t)\in(t,t+1)$ such that
$$\log\!\big(\zeta(t+1)/\zeta(t)\big)=\log\zeta(t+1)-\log\zeta(t)=\zeta'(\xi(t))/\zeta(\xi(t))=
(\pi/d)\big(1+\varepsilon(\xi(t))\big),$$
where $\varepsilon(\xi(t))\to 0$ as $t\to\infty$, hence
$$\zeta(t+1)/\zeta(t) = \ex^{(\pi/d)(1+\varepsilon(\xi(t))}\to \ex^{\pi/d}\qquad\text{as $t\to\infty$,}$$
and this yields the claim since $\zeta(n)=s_n$ for $n\ge n_0$.
\end{proof}

\begin{example}[Cauchy-Euler equation]
Suppose $a=\alpha x^{-1}$, $b=\beta x^{-2}$ ($\alpha,\beta\in\R$), so $$y''+\alpha x^{-1}y' +\beta x^{-2}y\ =\ 0.$$
Then $f=-2a'-a^2+4b=cx^{-2}$ where $c:=2\alpha -\alpha^2  +4\beta>1$.
%so $(\alpha-1)^2<4\beta$. 
With $d:=\frac{1}{2}\sqrt{c-1}$ we have
 $\phi - d\log x\in\R$ by Corollary~\ref{naicor} and Example~\ref{ex:boundphi}(1),
so there are   $C\in\R^>$ and $n_0\in\N$ such that for $n\ge n_0$ we have~$s_n=C\ex^{\pi n/d}$
%$s_n=C\exp\!\big( (\pi n)/d\big)$ 
and thus $s_{n+1}/s_n=\ex^{\pi/d}$.
\end{example}

\noindent
Define the counting function $N\colon \R^{\geq e}\to\N$ by
$$N(t)\ :=\  \big| [e,t]\cap y^{-1}(0) \big|\  =\  \min\{n:s_n>t\},$$
so for $n\ge 1$: $N(t)=n\Leftrightarrow s_{n-1} \le t < s_n$. Thus $N(t)\to+\infty$ as $t\to+\infty$; in fact:

\begin{lemma}\label{lem:N(t)} $N \sim \phi/\pi$.
\end{lemma}
\begin{proof}
Take $n_0$, $k$ as in the proof of Lemma~\ref{lem:zeta}, so
$\phi(s_n)=(k+n)\pi$ for $n\geq n_0$.
Let $t\geq e$ be such that $N(t) \geq n_0+1$ ; then
$s_{N(t)-1} \leq t < s_{N(t)}$ and thus
$$N(t)+k-1\ =\ \phi(s_{N(t)-1})/\pi\  \leq\ \phi(t)/\pi\  <\  \phi(s_{N(t)})/\pi\  =\ N(t)+k.$$
This yields $N \sim \phi/\pi$. 
\end{proof}

\noindent
The quantity $N(t)$ has been studied extensively in connection with second-order linear differential equations; see~\cite[Chapter~IX, \S{}5, and the literature quoted on p.~401]{Hartman}.
For example, the next corollary is a consequence of a more general result due to Wiman~\cite{Wiman} (see \cite[Chapter~IX, Corollary~5.3]{Hartman}), but is also evident from our Hardy field calculus: 
%Here we assume~$a=0$, so $f=4b\in \c_e$.  

\begin{cor}
Suppose $f(t)>0$ for all $t\geq e$, and $f\succ 1/x^2$. % $\big(1/\sqrt{f}\big)'\prec 1$. 
 Then 
$$N(t)\ \sim\  \frac{1}{2\pi} \int_e^t  \sqrt{f(s)}\,ds\quad\text{as $t\to +\infty$.}$$
\end{cor}
\begin{proof} 
Arrange $H$ to be maximal with $(g,\phi)\in H^2$. 
Lemma~\ref{lem:2nd order, f succ 1/x^2} yields $\phi'\sim \frac{1}{2}\sqrt{f}$.
Let~$\phi_1\in \c_e^1$, $\phi_1(t):=\frac{1}{2}\int_e^t \sqrt{f(s)}\, ds$. Then $\phi_1'=\frac{1}{2}\sqrt{f}$, so (the germ~of)~$\phi_1$ lies in $H$ and  $\sqrt{f}\succ 1/x$, hence $\phi_1>\R$, and thus $\phi\sim \phi_1$ by Lemma~\ref{lem:9.1.4(ii)}. Now apply Lemma~\ref{lem:N(t)}.
%From  $\big(1/\sqrt{f}\big){}'\prec 1$ we get $f^\dagger\prec \sqrt{f}$. Now $f$ is hardian, so $f\preceq 1/x^2$ would give $f^\dagger\succeq 1/x$, which together with $f\preceq 1/x^2$  contradicts $f^\dagger\prec \sqrt{f}$. Thus~$f\succ 1/x^2$.   For the rest of the argument we arrange $H$ is maximal with $(g,\phi)\in H^2$.  Lemma~\ref{lem:2nd order, f succ 1/x^2} yields a pair $(g_1,\phi_1)\in H^2$ parametrizing $\ker_{\Calinf}(\der^2+b)$ such that~$\phi_1'\sim (1/2)\sqrt{f}$. Then $\phi-\phi_1\in \R$ by Lemma~\ref{lem:g,phi unique}, so $\phi'\sim (1/2)\sqrt{f}$. Let~$\phi_2\in \c_e^1$ be given by~$\phi_2(t)=(1/2)\int_e^t \sqrt{f(s)}\, ds$. Then $\phi_2'=(1/2)\sqrt{f}$, so (the germ of) $\phi_2$ lies in $H$ and  $\sqrt{f}\succ 1/x$, so $\phi_2>\R$. Hence by Lemma~\ref{lem:9.1.4(ii)} we have $\phi\sim \phi_2$. Now apply Lemma~\ref{lem:N(t)}. 
\end{proof}

\noindent
We now establish the comparability of $(s_n)$ with other sequences of interest:

\begin{cor}\label{cor:zeta, 3}
Let $h\in\c_0$ and suppose the germ of $h$ is in $\Ex(\Q)$. Then the se\-quen\-ces~$(s_n)$ and $\big(h(n)\big)$ are comparable.
\end{cor}
\begin{proof}
Let $\zeta$ be as in Lemma~\ref{lem:zeta}, and note that the germs of $\zeta$ and $h$ lie in a common Hardy field.
% the compositional inverse $\zeta^{\operatorname{inv}}$ of the germ of $\zeta$
%is $H$-hardian, and by Lemma~\ref{lem:comp with E(Q)}, so is $g\circ \zeta^{\operatorname{inv}}$; hence $g\circ \zeta^{\operatorname{inv}}\leq x$ or~$g\circ \zeta^{\operatorname{inv}}\geq x$ and so $g\leq\zeta$ or $g\geq\zeta$.
\end{proof}

\begin{remark} %\marginpar{to generalize this we would need  a version of 5.1.12 in {\tt maxhardy} for $g=x+c\log x+o(\log x)$}
In \cite{ADHbf} we show: if the germ of $f$ is in $\R(x)$ and $f=c+O(x^{-2})$, $c\in\R^>$, then for each~$h\in\c_0$ with hardian germ, the \marginpar{to check in [8]}
sequences~$(s_n)$ and~$\big(h(n)\big)$ are comparable.
\end{remark}

\noindent
Let also $\underline{a}, \underline{b}\in H$, and take representatives of $\underline{a}$, $\underline{b}$ in $\c_{\underline{e}}^1$ ($\underline{e}\in \R$), denoting these by~$\underline{a}$ and~$\underline{b}$ as well. Let $\underline{y}\in \c_{\underline{e}}^2$ be an oscillating solution of the differential equation$$Y''+\underline{a}Y' +\underline{b}Y\ =\ 0,$$
and let $(\underline{s}_n)$ be the enumeration of~$\underline{y}^{-1}(0)$.

\begin{lemma}\label{lem:comp sequ zeros}
The sequences $(s_n)$ and $(\underline{s}_n)$ are comparable.
%Tehere are some $a\in\R$ and some representatives of $y$, $z$ in $\c_a^2$ whose sequences of zeros are comparable.
%whenever $s_1<s_2<\cdots$ are the consecutive zeros of $y|_b$ and
%$t_1<t_2<\cdots$ are the consecutive zeros of $z|_b$, then $(s_n)$,
%, then there is exactly one zero $t\in [t_1,t_2]$ of $z$.
\end{lemma}
\begin{proof}
We arrange that $H$ is maximal and take $\zeta$ as in Lemma~\ref{lem:zeta}. This lemma also  provides a  strictly increasing $\underline{\zeta}\in\c_{\underline{n}_0}$ \textup{(}$\underline{n}_0\in\N$\textup{)} 
 such that  $\underline{s}_n=\underline{\zeta}(n)$ for all~$n\geq \underline{n}_0$ and the germ of $\underline{\zeta}$ is hardian with $H$-hardian compositional inverse. With $\zeta$ and $\underline{\zeta}$ denoting also their germs this gives
 $\zeta^{\operatorname{inv}}\le \underline{\zeta}^{\operatorname{inv}}$ or 
 $\zeta^{\operatorname{inv}}\ge \underline{\zeta}^{\operatorname{inv}}$, hence~$\zeta\ge \underline{\zeta}$ or $\zeta\le \underline{\zeta}$. Thus $(s_n)$ and $(\underline{s}_n)$ are comparable.
\end{proof}

\subsection*{The critical points of an oscillating solution}  Our oscillating solution $y\in \c_e^2$ has also oscillating derivative $y'$, so by Corollary~\ref{gphicos} there is for all sufficiently large~$n$ exactly one $t\in (s_n,s_{n+1})$ with $y'(t)=0$.  Also $b\ne 0$ in $H$, since $b=0$ would mean that~${z:=y'}$  satisfies
$z'+az=0$, so $z$ would be $H$-hardian. Increasing~$e$ and restricting $a$, $b$, $y$ accordingly we
have $y\in\c_e^{3}$. In this subsection~$\zeta$ is as constructed in the proof of Lemma~\ref{lem:zeta}. Then: 
% (and~$y''+ay'+by=0$ with oscillating~$y$ as before); then  the zero sets of $y,y',\dots,y^{(m)}$ are eventually parametrized by hardian germs as follows:
%the germ of $y'$ also satisfies a second-order linear differential equation over $H$. (See the remarks before Lemma~\ref{lem:A^der}.)  Thus Lemma~\ref{lem:zeta} applies to the enumerations of the zero set of~$y'|_{a_1}$ in place of $(s_n)$.

\begin{lemma}\label{lem:param zeros of derivatives} 
For some $n_1\ge n_0$ in $\N$ and strictly increasing function ${\zeta_1\in\c_{n_1}}$:
\begin{enumerate}
\item[\textup{(i)}]   $\zeta_1(n_1)\ge e$ and $\zeta_1(t)\to +\infty$ as $t\to +\infty$;
\item[\textup{(ii)}] the germ of $\zeta_1$ is hardian with $H$-hardian  compositional inverse; 
\item[\textup{(iii)}] $\big\{\zeta_1(n):\, n\ge n_1\big\}=\big\{t\geq \zeta_1(n_1):\, y'(t)=0\big\}$;
\item[\textup{(iv)}]  $\zeta_1^{\operatorname{inv}}-\zeta^{\operatorname{inv}}\preceq 1$;
\item[\textup{(v)}]  $\zeta(n)<\zeta_{1}(n)<\zeta(n+1)$ for all $n\ge n_{1}$.
\end{enumerate}
\end{lemma} 
\begin{proof}
We arrange that $H$ is maximal.  Set 
$A:=\der^2+a\der +b\in H[\der]$. As $b\ne 0$, we have the monic operator $A^{\der}\in H[\der]$ of order $2$ as defined before Lemma~\ref{lem:A^der},   with~$A^\der(y')=0$. Take  a pair $(g_1,\phi_1)$  of elements of $H$ parametrizing $\ker_{\Calinf} A^{\der}$. Then with $A^\der$, $g_1$, $\phi_1$ instead of $A$, $g$, $\phi$, and taking suitable representatives of the relevant germs, the proof of Lemma~\ref{lem:zeta} provides likewise an $n_1\ge n_0$ in $\N$, 
a~$k_1\in \frac{1}{2}+\Z$, and a strictly increasing function 
$\zeta_1\in \c_{n_1}$ satisfying clauses (i), (ii), (iii) and with compositional inverse given by $(\phi_1/\pi)-k_1$.

The remarks preceding Lemma~\ref{lem:kerB} yield a differential subring $R$ of $\Calinf[\imag]$ such that
 $R\supseteq K=H[\imag]$ and $R$ is an integral domain whose differential fraction field has constant field $\C$, with $\dim_{\C} \ker_R A=2$.
Then $\ker_{R} A^\der={\big\{z':\ z\in \ker_{R}A\big\}}$ and~$\dim_{\C}\ker_R A^\der=2$ 
by Lemma~\ref{lem:A^der}, so $$\ker_{\Calinf[\imag]} A^\der\ =\ {\big\{y':\ y\in \ker_{\Calinf[\imag]}A\big\}}.$$ 
In view of~$A\in H[\der]$, this displayed equality also holds with $\Calinf$ in place of $\Calinf[\imag]$.
Now (iv) follows from
Lemmas~\ref{lem:param V'} and~\ref{lem:g,phi unique}.

As to (v), the remark preceding the lemma gives $\ell\in \N$ and $p\in \Z$ such that for all~$n\ge n_1+\ell$ we have:
$n+p\ge n_0$ and $\zeta_1(n)$ is the unique zero of $y'$ in the interval~$\big(\zeta(n+p), \zeta(n+p+1)\big)$.
Set $n_1^*:=n_1+\ell+|p|$, and modify~$\zeta_1$ to~$\zeta_1^*\colon [n_1^*, +\infty)\to \R$ by setting
 $\zeta_1^*(t)=\zeta_1(t-p)$. Then  $\zeta(n)< \zeta_1^*(n) < \zeta(n+1)$ for all $n\ge n_1^*$. 
 The compositional inverse of $\zeta_1^*$ is given by $(\phi_1/\pi)-(k_1-p)$. Thus replacing $\zeta_1$,~$n_1$,~$k_1$ by 
 $\zeta_1^*$,~$n_1^*$,~$k_1-p$, all clauses are satisfied. 
\end{proof}

\noindent
For simplicity we assume next that $a=0$ on $[e,+\infty)$, so  $4y''+fy=0$ on $[e, +\infty)$ with $f=4b\in \c_e^1$. Towards using Lemma~\ref{lem:param zeros of derivatives} we increase $e$  and restrict $f$, $y$ accordingly to arrange that $y\in \c_e^3$ and $f(t)\ne 0$ for all $t\in [e,+\infty)$. Then each interval~$(s_n,s_{n+1})$
contains exactly one zero $t_n$ of $y'$, and  $t_0 < t_1 < t_2 < \cdots$  enumerates the critical points of $y$ that are $\ge s_0$. 

Take
$e_0\ge e$ and $g$, $\phi$ as before Lemma~\ref{lem:zeta} such that in addition 
$w:=g^2\phi'$ takes a constant positive value on $[e_0,+\infty)$ (possible by Lemma~\ref{lem:parametrization of ker A}), and
$z''+fz=0$ on $[e_0, +\infty)$ for $z:=g\sin \phi$.  
%Each~$t_n$ is an extremal point of $y$,
If~$f$ is strictly increasing (strictly decreasing), then the sequence $\big(\abs{y(t_n)}\big)$ is strictly decreasing (strictly increasing, respectively), by Lemma~\ref{lem:extremal pts}; likewise
without ``strictly''.
(Since the germ of $f$ is hardian, it is eventually   increasing or eventually   decreasing.) 
 Remarkably, $\big(\abs{y(t_n)}\big)$ 
can also be  interpolated by a hardian function: 
%\marginpar{if $\phi$ is perfectly $H$-hardian, then so is the germ of $v$; what about that of $u$?}

\begin{prop}\label{prop:v} There exists $n_1\ge n_0$ in $\N$ and a function $u\in \c_{n_1}$ with hardian germ such that
$u(n)=\abs{y(t_n)}$ for all $n\ge n_1$.
\end{prop}
\begin{proof}  We arrange $H$ is maximal, so that (the germs of) $g,\phi$ are in $H$. Then the germs of $\psi:=1/\phi'\in \c_{e_0}$ and $v:=\displaystyle\frac{g}{\sqrt{1+(\psi'/2)^2}}\in\c_{e_0}$ are also in $H$. The proof of Lemma~\ref{lem:param zeros of derivatives}  yields a strictly increasing~${\zeta_1\in\c_{n_1}}$ with $n_1\in\N,\ n_1\ge n_0$, such that its germ is hardian, $\zeta_1(n)=t_n\ge e_0$ for $n\ge n_1$, 
and the germ $\ell$  of the compositional inverse of $\zeta_1$ is in $H$. 
Then $H\circ \ell^{\inv}$ is a Hardy field that contains the germ of~$u:=v\circ\zeta_1\in \c_{n_1}$. (See the subsection ``Hardy fields and composition'' in Section~\ref{sec:prelims}.) 
Consider the elements $$y_1\ :=\ y|_{[e_0,+\infty)}\ =\ g\cos\phi,\qquad y_2\ :=\ g\sin\phi$$
of $\c^2_{e_0}$.
Then on $[e_0,+\infty)$  we have
$$y_1'\ =\ g'\cos\phi-g\phi'\sin\phi,\qquad y_2'=g'\sin\phi+g\phi'\cos\phi,$$
hence
$$y_1y_2'-y_1'y_2\ =\ w,\qquad (y_1')^2+(y_2')^2\  =\  (g')^2+(g\phi')^2.$$
From $g^2\phi'=w\in\R$ we obtain $2g'\phi'+g\phi''=0$, so
$$(2\phi'/g)^2\big( (y_1')^2+(y_2')^2 \big)\  =\  (\phi'')^2+4(\phi')^4 
$$
and thus on $[e_0,+\infty)$: 
$$\frac{w^2}{(y_1')^2+(y_2')^2}\  =\  \frac{4w(\phi')^3}{(\phi'')^2+4(\phi')^4}\ =\ 
\frac{4w/\psi^3}{(\psi')^2/\psi^4+4/\psi^4}\ =\ 
\frac{w\psi}{1+(\psi'/2)^2}\ =\ 
v^2,$$
so $v=\displaystyle\frac{w}{\sqrt{(y_1')^2+(y_2')^2}}$.
Hence $u$ has the desired property by Proposition~\ref{prop:Trench}.
\end{proof}

%\begin{lemma}\label{lem:v}
%Put $\psi:=1/\phi'$.  
%The germ of
%$v:=\displaystyle\frac{g}{\sqrt{1+(\psi'/2)^2}}\in\c_{e_0}$ is $H$-hardian,  
%the germ of $u:=v\circ\zeta_1\in\c_{n_1}$ is hardian, and  
%$u(n)=v(t_n)=\abs{y(t_n)}$ for~$n\ge n_1$.
%\end{lemma}
%\begin{proof}  We arrange $H$ is maximal, so that (the germs of) $g,\phi$ and hence of $v$ are in $H$. The proof of Lemma~\ref{lem:param zeros of derivatives} then yields a strictly increasing~${\zeta_1\in\c_{n_1}}$ with $n_1\in\N,\ n_1\ge n_0$, such that its germ is hardian, $\zeta_1(n)=t_n\ge e_0$ for $n\ge n_1$, 
%and the germ $\ell$  of the compositional inverse of $\zeta_1$ is in $H$. 
%Then $H\circ \ell^{\inv}$ is a Hardy field which contains the germ of $u=v\circ\zeta_1$. (See the subsection ``Hardy fields and composition'' in Section~\ref{sec:prelims}.) {\bf oksofar}
%Consider the elements $$y_1:=y=g\cos\phi,\qquad y_2:=g\sin\phi$$
%of $\c^2_{e_0}$.
%Then
%$$y_1'=g'\cos\phi-g\phi'\sin\phi,\qquad y_2'=g'\sin\phi+g\phi'\cos\phi,$$
%hence
%$$y_1y_2'-y_1'y_2=w,\qquad (y_1')^2+(y_2')^2 = (g')^2+(g\phi')^2.$$
%From $g^2\phi'=w\in\R$ we obtain $2g'\phi'+g\phi''=0$, so
%$$(2\phi'/g)^2\big( (y_1')^2+(y_2')^2 \big) = (\phi'')^2+4(\phi')^4 
%$$
%and thus
%$$\frac{w^2}{(y_1')^2+(y_2')^2} = \frac{4w(\phi')^3}{(\phi'')^2+4(\phi')^4}=
%\frac{4w/\psi^3}{(\psi')^2/\psi^4+4/\psi^4}=
%\frac{w\psi}{1+(\psi'/2)^2}=
%v^2.$$
%Now use Proposition~\ref{prop:Trench}.
%\end{proof}

{\samepage
\begin{remarks} Let $u$, $v$ be as in the proof of Proposition~\ref{prop:v}.
\begin{enumerate}
\item
Suppose the germ of $f$ is $\d$-algebraic. Then the germ of $\phi$ and hence the germ of $v$ are $\d$-algebraic. By
Corollary~\ref{cor:zeta d-alg} and the proof of Lemma~\ref{lem:zeta} the germ of $\zeta$ and of its compositional inverse are then 
$\d$-algebraic as well. Moreover, using also the proof of Lemma~\ref{lem:param zeros of derivatives} we can choose~$\zeta_1$ so that the germ of $\zeta_1$ and the
germ $\ell$ of its compositional  inverse are also $\d$-algebraic.  Then $u$
has $\d$-algebraic germ, by Lemma~\ref{lem:trdeghcirc}. Together with
Lemma~\ref{lem:param zeros of derivatives} and Proposition~\ref{prop:v}, this proves Corollary~4 from the introduction.
\item Suppose $\phi\succ\log x$. (See Corollary~\ref{cor:upper lower bd for phi} for situations where this holds.)
Then $\phi''\prec (\phi')^2$, hence $v\sim g$. (This justifies a remark after Corollary~4.)
\item  Proposition~\ref{prop:v} has an analogue for the sequence
$\big(\abs{y'(s_n)}\big)$:
Differentiating~$g\cos \phi$, evaluating at $s_n$, and using $g=\sqrt{w/\phi'}$ gives
$$|y'(s_n)|\ =\ \sqrt{w\phi'(s_n)}\quad \text{ for all $n\ge n_0$}.$$ 
The germ of $w\phi'$ is $H$-hardian. 
Suppose $\zeta\in\c_{n_0}$ is as in
Lemma~\ref{lem:zeta}.
Then~$u_0:=(w\phi')\circ\zeta\in\c_{n_0}$  satisfies $u_0(n)=|y'(s_n)|$ for all $n\ge n_0$
and has hardian germ.  %Let $n_0\in\N$ with $s_{n_0}\ge e_0$.
%We may  likewise interpolate   $\abs{y'(s_n)}$ ($n\ge n_0$) by a function with
%$H$-hardian germ:   $(\phi')^{1/2}(s_n)=\abs{y'(s_n)}$ for $n\ge n_0$.
\end{enumerate}
\end{remarks}}

\begin{example}
Suppose $r\in\R$, $r>-1$ with $x^r\in H$, and as in  Example~\ref{ex:boundphi}(2):
$$f\ =\  x^{2r}-r(r+2)x^{-2}, \qquad g\ =\ x^{-r/2}, \quad
\phi\ =\ \frac{x^{r+1}}{2r+2}.$$
Then  
$\psi=2x^{-r}$, so $\psi'/2=-rx^{-r-1}$ and hence
$v^2=\displaystyle\frac{g^2}{1+(\psi'/2)^2}=\frac{x^{r+2}}{x^{2r+2}+r^2}$.
Figure~\ref{fig:v} shows a plot of $y$, $-y$, $v$ for~$r=1$.
\end{example}

\begin{figure}
\begin{tikzpicture}

\begin{axis}[
clip = true,
    clip mode=individual,
    axis x line = middle,
    axis y line = middle,
domain=0.8:15, xlabel=$t$, axis lines=middle, xtick={0},ytick={0}, 
    inner axis line style={=>},  restrict x to domain=0.8:15,  ymin=-1, ymax=1,
    enlarge y limits={rel=0.10},
    enlarge x limits={rel=0.07}]
\addplot[samples = 200,
		smooth,
		thick]{(x^(-1/2))*( cos(deg(x^2/4))) )} node[right,pos=0] {$y(t)$};
		
\addplot[samples = 200,
		dotted,
		thick]{-(x^(-1/2))*( cos(deg(x^2/4)) )} node[right,pos=0] {$-y(t)$};

\addplot[samples = 200,
		dashed,
		thick]{sqrt((x^3)/(1+x^4) )} node[above left,pos=1] {$v(t)$};
		
\end{axis}

\end{tikzpicture}

%\begin{tikzpicture}

%\begin{axis}[
%clip = true,
%    clip mode=individual,
%    axis x line = middle,
%    axis y line = middle,
%domain=1:10000, xlabel=$t$, axis lines=middle, xtick={0},ytick={0}, 
%    inner axis line style={=>},  restrict x to domain=0:10000, enlarge y limits={rel=0.10},
%    enlarge x limits={rel=0.07}]
%\addplot[samples = 200,
%		smooth,
%		thick]{(x^(1/2))*( cos(deg(10*ln(x))) )} node[left,pos=1] {$y(t)$};
		
%\addplot[samples = 200,
%		dotted,
%		thick]{-(x^(1/2))*( cos(deg(10*ln(x))) )} node[below left,pos=1] {$-y(t)$};

%\addplot[samples = 200,
%		dashed,
%		thick]{sqrt(400*x/401)} node[above left,pos=1] {$v(t)$};
%\end{axis}

%\end{tikzpicture}

\caption{Plot of $y$, $-y$, $v$}\label{fig:v}
\end{figure}
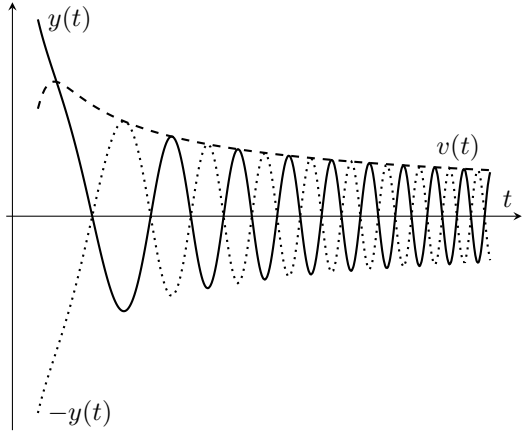

\section{When is the Perfect Hull of a Hardy Field $\upo$-free?}\label{sec:E(H) upo-free}

\noindent
{\it In this section $H$ is a Hardy field.}\/  
Below we use the lemmas that prove Theorem~\ref{thm:Bosh} to characterize $\upo$-freeness of $\Dx(H)$ in terms of $H$.
Recall the relevance of this condition to uniqueness 
in Theorem~\ref{thm:Bosh}; more on this at the end of this section.  

{\samepage
\begin{theorem}\label{thm:upo-freeness of the perfect hull}
The following are equivalent: 
\begin{enumerate}
\item[\textup{(i)}] $H$ is not $\upl$-free or  $\bar{\omega}(H)=H\setminus\sigma\big(\Upg(H)\big){}^\uparrow$;
\item[\textup{(ii)}] $\Dx(H)$ is $\upo$-free;
\item[\textup{(iii)}] $\Ex(H)$ is $\upo$-free.
\end{enumerate}
\end{theorem}}
 
\noindent
The disjuncts of  (i) are exactly the conditions (a) and (b) in Theorem~B of the Introduction. By Corollary~\ref{cor:perfect Schwarz closed, 2}, 
a $\d$-perfect Hardy field is Schwarz closed iff it is $\upo$-free, so in~(ii),~(iii) we can replace 
  ``$\upo$-free'' by ``Schwarz closed''. 
The implication~(i)~$\Rightarrow$~(ii) holds by Lemma~\ref{lem:7.11}. %was shown already in \cite[Lemma 7.11]{ADH5}.  
For the contrapositive of~(iii)~$\Rightarrow$~(i), suppose~$H$ is $\upl$-free and~${\bar{\omega}(H)\neq H\setminus\sigma\big(\Upg(H)\big){}^\uparrow}$. Since~$\bar{\omega}(H)\subseteq H\setminus\sigma\big(\Upg(H)\big){}^\uparrow$ by the remark after Corollary~\ref{cor:omega(H) downward closed},
this yields~$\upo\in H$ with
\begin{equation}\label{eq:upo}
\bar{\omega}(H)\ <\  \upo\   <\  \sigma\big(\Upg(H)\big), 
\end{equation}
and so by Lemma~\ref{lem:upo-freeness of the perfect hull} below, $\operatorname{E}(H)$ is not $\upo$-free.  The proof of Lemma~\ref{lem:upo-freeness of the perfect hull} relies on Corollary~\ref{cor:perfect Schwarz closed, 1}, but
additionally draws on facts from
Sections~\ref{sec:ADH} and~\ref{sec:prelims} and
the following consequence of \cite[Corollaries~7.15, 7.17]{ADH5}:

\begin{lemma}\label{lem:7.15+7.17}
Suppose $H\supseteq \R$ is Liouville closed and  $\upo\in H$ satisfies \eqref{eq:upo}.
Then there are continuum many  $\upg\in (\Calinf)^\times$ with $\upg>0$ and $\sigma(\upg)=\upo$. Each
such germ~$\upg$ is $H$-hardian, and no Hardy field extension of~$H$ contains more than one such   $\upg$.
\end{lemma}

\begin{lemma}\label{lem:upo-freeness of the perfect hull}
Suppose $H$ is $\upl$-free, and $\upo\in H$ satisfies \eqref{eq:upo}. % $\bar{\omega}(H)<\upo<\sigma\big(\Upg(H)\big)$.
Then 
$$\omega(E)\ =\ \bar{\omega}(E)\ <\ \upo\  <\  \sigma\big(\Upg(E)\big)\qquad\text{for $E:=\Ex(H)$,}$$
hence $E$ is not $\upo$-free. $($The end of this section gives $H$, $\upo$ satisfying the hypothesis.$)$ 
\end{lemma}
{\sloppy
\begin{proof}
The restriction of $\sigma$ to $\Upg(H)$ is strictly increasing by
Lemma~\ref{lem:11.8.29}.
Hence we may replace~$H$ by any $\upl$-free Hardy subfield $L$ of $E$ containing~$H$ such that~$\Gamma_H^<$ is cofinal in~$\Gamma^<_L$, since then $\Upg(H)$ is coinitial in $\Upg(L)$ by [ADH, 11.8.14]
and thus~$\sigma\big(\Upg(H)\big)$ is coinitial in $\sigma\big(\Upg(L)\big)$. Using this observation  and Lemma~\ref{lem:1.3.3+}   replace $H$ by~$H(\R)$ to
arrange~$H\supseteq\R$.  Next   
replace~$H$ by $\operatorname{Li}(H)\subseteq E$ to arrange that~$H$ is Liouville closed, using Proposition~\ref{prop:Gehret}. 
Now $\omega(E)=\bar{\omega}(E)$ is downward closed and~${\bar{\omega}(E)\cap H=\bar{\omega}(H)}$,  so $\omega(E) < \upo$. %We have $\omega(E)  < \sigma\big(\Upg(E)\big)$ where~ and 
 %$\sigma\big(\Upg(E)\big)$ is upward closed, by Corollary~\ref{cor:perfect Schwarz closed, 1}.
 %Since 
  Towards a contradiction, assume~${\upo\in \sigma\big(\Upg(E)\big){}^{\uparrow}}$. As~$\sigma\big(\Upg(E)\big)$ is upward closed by Corollary~\ref{cor:perfect Schwarz closed, 1}, we have~${\upg\in \Upg(E)}$ with~$\sigma(\upg)=\upo$.  Lemma~\ref{lem:7.15+7.17}     yields $\tilde\upg\ne \upg$ in $(\Calinf)^\times$ 
with~$\tilde\upg >0$ and~${\sigma(\tilde\upg)=\upo}$. Taking a maximal Hardy field extension~$M$ of $H$
with~$\tilde\upg\in M$ then yields a contradiction with the last part of Lemma~\ref{lem:7.15+7.17}  in view of $\upg\in M$.
%Now $M$ is $\upo$-free by Corollary~\ref{cor:maxhardymainthm, 1} and~${\upo\notin\bar{\omega}(M)}$, so
%~$\upo\in\sigma\big(\Upg(M)\big){}^{\uparrow}$ by Lemma~\ref{omuplosc} and so~
%$\tilde \upg\in\Upg(M)$ by Lemma~\ref{lem:11.8.29}. 
%Since~${E\subseteq M}$ we have~$\Upg(E)\subseteq\Upg(M)$. Then from~$\sigma(\upg)=\upo=\sigma(\tilde \upg)$ we   obtain $\upg=\tilde \upg$ by Lemma~\ref{lem:11.8.29}, a contradiction.
\end{proof}}

\noindent
To finish the proof of Theorem~\ref{thm:upo-freeness of the perfect hull}
it remains to show the implication~(ii)~$\Rightarrow$~(iii), which we do in Lemma~\ref{lem:upo-freeness of the perfect hull, (ii)=>(iii)} below. (By~\cite[Theorem~14.4]{Boshernitzan82}, this  implication holds trivially if $H$ is {\it bounded,}\/ that is,
if there is some $\phi\in\c$ with $h\le\phi$ for each~$h\in H$; see~also~\cite[Theorem~5.20]{ADH5}.) 
%We precede this lemma with a useful consequence of a proposition from  \cite{AvdDan}, for which 
Until  Lemma~\ref{lem:upo-freeness of the perfect hull, (ii)=>(iii)} we assume that
 $H\supseteq\R(x)$ is real closed with asymptotic integration.
Let
$(f_\rho)$ be a pc-sequence in $H$. Here,
as in~[ADH],    {\em pc-sequence}\/   abbreviates 
{\em pseudocauchy sequence}, and ${f_\rho\leadsto f}$ means that~$f$ is a pseudolimit of~$(f_\rho)$.
See [ADH, 2.2,~3.2]
for basic facts about pc-sequences.
We also refer to~[ADH, 4.4] for $(f_\rho)$ being of {\it   $\d$-transcendental type}\/ over $H$.

Suppose $(f_\rho)$ is of $\d$-transcendental type over $H$. Then $(f_\rho)$ has no pseudolimit in~$H$, and
if  $\hat f$ is a pseudolimit of $(f_\rho)$ in a Hardy field extension of $H$, then~$H\langle\hat f\rangle$ is an immediate valued field extension  of $H$ by [ADH, 11.4.7, 11.4.13], and~${v(\hat f-H)\subseteq\Gamma_H}$ is downward closed in $\Gamma_H$.
The following is \cite[Proposition~3.4]{AvdDan}:

\begin{prop}\label{prop:3.4}
Suppose 
$(f_\rho)$ is of $\d$-transcendental type over $H$. Let $\hat f$ in a
Hardy field extension of $H$ be such that $f_{\rho} \leadsto \hat{f}$ and $0\notin v(\hat f-H)$, and let $f\in\Calinf$ be such that $(f-\hat f)^{(m)} \prec x^{-n}$ for all $m$, $n$.
Then $f$ is $H$-hardian, and there is an isomorphism $H\<f\> \to H\<\hat f\>$ of Hardy fields over $H$ sending $f$
to $\hat f$.
\end{prop}

\noindent
From Proposition~\ref{prop:3.4}  we   obtain:

\begin{cor}\label{cor:d-trans => c-sequ}
Suppose $(f_\rho)$ is of $\d$-transcendental type over $H$, and 
 $\hat f\in\Ex(H)$ is a pseudolimit of $(f_{\rho})$. Then $v(\hat f-H)=\Gamma_H$.  
\end{cor}
\begin{proof}
Suppose not. Take  $g\in H^\times$ with
$v(\hat f-H)<vg$ and
replace~$(f_\rho)$,~$\hat f$  by~$(f_\rho/g)$,~$\hat f/g$,  respectively, to arrange $v(\hat f-H)<0$.
Now put~$z:=\ex^{-x}\sin x \in \c^\omega$. Then $\abs{z^{(m)}}\leq 2^m\ex^{-x}$ for all $m$,
hence by Proposition~\ref{prop:3.4},  
the germ~$f:=\hat f+z$ generates a Hardy field $H\langle f\rangle$ over $H$; however,
no maximal Hardy field extension of $H$ contains both $\hat f$ and $f$, contradicting $\hat f\in\Ex(H)$.
\end{proof}

\noindent
The next lemma, with $\Dx(H)$ for $H$, yields  (ii)~$\Rightarrow$~(iii) in
Theorem~\ref{thm:upo-freeness of the perfect hull}:

\begin{lemma}\label{lem:upo-freeness of the perfect hull, (ii)=>(iii)}
Suppose $H$ is $\upo$-free. Then  $\Ex(H)$ is also $\upo$-free.
\end{lemma}
\begin{proof}
Since~$E:=\Ex(H)$ is Liouville closed  and contains $\R$ we may replace $H$ by the Hardy subfield
$\operatorname{Li}\!\big(H(\R)\big)$ of $E$, which remains $\upo$-free by Theorem~\ref{thm:13.6.1}, to arrange that $H\supseteq \R$ and $H$ is Liouville closed (so Corollary~\ref{cor:d-trans => c-sequ} applies).
Towards a contradiction, suppose~$\upo\in E$ satisfies $\omega(E) < \upo < \sigma\big(\Upg(E)\big)$; then
$\omega(H) < \upo < \sigma\big(\Upg(H)\big)$. Take a logarithmic sequence $(\ell_\rho)$ for $H$ as in [ADH, 11.5]
and define~$\upo_\rho:=\omega(-\ell_\rho^{\dagger\dagger})$. 
Then $(\upo_\rho)$ is a   pc-sequence in~$H$ with $\upo_\rho\leadsto\upo$, by [ADH, 11.8.30], 
 and~$(\upo_\rho)$ is of $\d$-transcendental type over~$H$ by~[ADH, 13.6.3].
By~[ADH, 11.7.2] we have~${\Gamma_H\setminus v(\upo-H)=\big\{\gamma\in \Gamma_H:\gamma> 2\Psi_H\big\}}$  which contains~$v(1/x^4)=2v\big((1/x)'\big)$,  
contradicting Corollary~\ref{cor:d-trans => c-sequ}.
\end{proof}

\noindent
Next  we describe for $j=1,2$ a $\upl$-free Hardy field $H_{(j)}\supseteq\R$ and $\upo_{(j)}\in H_{(j)}$ such that~$\omega\big(\Upl(H_{(j)})\big) <  \upo_{(j)}  < \sigma\big(\Upg(H_{(j)})\big)$  (so $H_{(j)}$ is not $\upo$-free by \eqref{eq:upofree}), and  
\begin{enumerate}
\item  $\upo_{(1)}\in\bar{\omega}(H_{(1)})$;
\item $\upo_{(2)}\notin\bar{\omega}(H_{(2)})$.
\end{enumerate}
It follows that $\bar{\omega}(H_{(1)})=H_{(1)}\setminus\sigma\big(\Upg(H_{(1)})\big){}^\uparrow$ by Lemma~\ref{lem:6.19},
hence   condition~(i) in Theorem~\ref{thm:upo-freeness of the perfect hull} is  satisfied for $H=H_{(1)}$, but it is {\it not}\/ satisfied for $H=H_{(2)}$; thus~$\Dx(H_{(1)})$ and $\operatorname{E}(H_{(1)})$ are   $\upo$-free, whereas $\Dx(H_{(2)})$ and $\operatorname{E}(H_{(2)})$ are not. 

{\sloppy
\medskip
\noindent
To construct such $H_{(j)}$ and $\upo_{(j)}\in H_{(j)}$ we start with 
 the Hardy field $$H\ :=\ \R( \ell_0,\ell_1,\ell_2,\dots)\quad\text{ where 
 $\ell_0=x$ and $\ell_{n+1}=\log \ell_n$ for all $n$.}$$ Then $H$ is $\upo$-free by Proposition~\ref{prop:11.7.15}.
Now take
a hardian   germ $\ell_\omega$ such that~${\R<\ell_\omega<\ell_n}$ for each $n$.  (See   
\cite[remarks after Corollary~5.28]{ADH5} for how to obtain such $\ell_\omega$.) 
Set $$\upg\ :=\ \ell_\omega^\dagger,\quad \upl\ :=\ -\upg^\dagger,\quad \upo_{(1)}\ :=\ \omega(\upl), \quad \upo_{(2)}\ :=\ \sigma(\upg)\ =\ \upo_{(1)} +\upg^2,$$
which are in the Hardy field $M:=H\langle\ell_\omega\rangle$. 
By [ADH,   11.5, 11.7],  $\upo_{(1)}$, $\upo_{(2)}$  are pseu\-do\-limits of the pc-sequence~$(\upo_n)$ in~$H$.  
For $j=1,2$,  the  Hardy subfield~$H_{(j)} :=  H\langle \upo_{(j)}\rangle$
of   $M$ is an immediate $\upl$-free extension of $H$ by [ADH, 13.6.3, 13.6.4], and therefore~$\omega\big(\Upl(H_{(j)}\big) <  \upo_{(j)}  < \sigma\big(\Upg(H_{(j)})\big)$ by [ADH, 11.8.30]. Moreover 
$$\bar{\omega}\big(H_{(j)}\big)\ =\ \bar{\omega}(M)\cap H_{(j)}\ \text{  for }j=1,2,$$  so (1) holds
since $\upo_1\in\omega(M)\subseteq\bar{\omega}(M)$,  
whereas $\upo_2\in\sigma\big(\Upg(M)\big)\subseteq M\setminus \bar{\omega}(M)$,  hence~(2) holds.}

{\sloppy
 \begin{remarkNumbered}\label{rem:non-uniqueness}
The above yields $\upl$-free $H$ and $\upo\in H$ with
$\bar{\omega}(H)\ <\  \upo\   <\  \sigma\big(\Upg(H)\big)$, namely $H_{(2)}$.   
We use this to show that the ``$\Dx(H)$ is $\upo$-free" condition in the last part of
Theorem~\ref{thm:Bosh} cannot be dropped. So assume $\Dx(H)$ is not $\upo$-free, equivalently, $H$ satisfies neither (a) nor (b)
in Theorem~B of the introduction. 

As in the proof of Lemma~\ref{lem:upo-freeness of the perfect hull},  arrange 
that~$H\supseteq\R$ is Liouville closed. Lem\-ma~\ref{lem:7.15+7.17}   gives infinitely many $H$-hardian~${\upg>0}$
 with~$\sigma(\upg)=\upo$. Let $M$ be a maximal
Hardy field extension of~$H$ containing such $\upg$. Then $\upg\in \Upg(M)$ by \cite[Corollary~7.16]{ADH5}, so~\eqref{eq:11.8.19}
gives 
$\phi >\R$ in $M$ with $(2\phi)'= \upg$. Hence  by Lem\-ma~\ref{parphi}(ii) the pair~$(1/\sqrt{\phi'},\phi)$  parametrizes~$\ker_{\Calinf} (4\der^2+\upo)$.  But another $\upg$ cannot lie in the same~$M$, and likewise for  the corresponding germs $\phi$.  
 \end{remarkNumbered}}

\newlength\templinewidth
\setlength{\templinewidth}{\textwidth}
\addtolength{\templinewidth}{-2.25em}

\patchcmd{\thebibliography}{\list}{\printremarkbeforebib\list}{}{}

%\addcontentsline{References}
\let\oldaddcontentsline\addcontentsline% Store \addcontentsline
\renewcommand{\addcontentsline}[3]{\oldaddcontentsline{toc}{section}{References}}

\def\printremarkbeforebib{\bigskip\hskip1em The citation [ADH] refers to the book \\

\hskip1em\parbox{\templinewidth}{
M. Aschenbrenner, L. van den Dries, J. van der Hoeven,
\textit{Asymptotic Differential Algebra and Model Theory of Transseries,} Annals of Mathematics Studies, vol.~195, Princeton University Press, Princeton, NJ, 2017.
}

\bigskip

} 

\bibliographystyle{amsplain}

\end{document}